\documentclass[11pt]{amsart}
\usepackage[latin1]{inputenc}
\usepackage[all]{xy}
\usepackage{graphics}
\usepackage{amscd}
\usepackage{amssymb}
\usepackage{amsthm}
\usepackage{enumitem}
\usepackage{comment}
\usepackage[all]{xy}
\usepackage{amscd}
\usepackage{amssymb,amsmath,latexsym}
\usepackage{amsthm}
\usepackage{enumitem}
\usepackage{mathrsfs,bbm}
\usepackage[dvipdfmx]{graphicx}
\usepackage[dvipdfmx]{xcolor,graphicx}
\usepackage[latin1]{inputenc}
\usepackage{graphics}
\usepackage{hyperref}  
\usepackage{soul}
\newtheorem{thm}{Theorem}[section]

\newtheorem{pro}[thm]{Proposition}
\newtheorem{lem}[thm]{Lemma}
\newtheorem{cor}[thm]{Corollary}
\newtheorem{defn}[thm]{Definition}

\newtheorem*{rem*}{Remarks}
\newtheorem{rems}[thm]{Remark}
\newtheorem*{conj*}{Conjecture}

\DeclareMathOperator{\A}{\mathbb{A}}
\DeclareMathOperator{\C}{\mathbb{C}}

\DeclareMathOperator{\F}{\mathbb{F}}

\DeclareMathOperator{\Q}{\mathbb{Q}}
\DeclareMathOperator{\R}{\mathbb{R}}
\DeclareMathOperator{\Z}{\mathbb{Z}}

\DeclareMathOperator{\Hom}{Hom}

\DeclareMathOperator{\im}{Im}

\DeclareMathOperator{\re}{Re}

\newcommand{\mrm}{\mathrm}
\newcommand{\mbb}{\mathbb}

\newcommand{\g}{\mathrm}
\newcommand{\GSp}{\mathrm{GSp}(4)}

\title{Endoscopic congruences modulo adjoint $L$-values for $\mrm{GSp}(4)$}
\author{Francesco Lemma}
\address{Univ Paris Diderot, Institut math\'ematique de Jussieu-Paris Rive Gauche, UMR 7586, B\^atiment Sophie Germain, Case 7012, 75205 Paris Cedex 13, France.}
\email{francesco.lemma@imj-prg.fr}
\author{Tadashi Ochiai}
\address{Graduate School of Science, Osaka University, 1-1 Machikaneyama-cho, Toyonaka,
Osaka 560-0043, Japan}
\email{ochiai@math.sci.osaka-u.ac.jp}

\begin{document}

\begin{abstract} We establish the existence of congruences between a fixed endoscopic, globally generic, cuspidal automorphic representation $\Pi$ of $\GSp$ of square-free conductor and stable cuspidal automorphic representations of the same weight modulo certain prime factors of the value at $1$ of the adjoint $L$-function of $\Pi$ normalized by a suitable period. 
\end{abstract}

\maketitle

\tableofcontents

\section{Introduction}

In modern number theory, the study of congruences modulo a prime number $p$ between automorphic forms plays a central role. It gives a basis of the theory of $p$-adic families of automorphic forms and the theory of $p$-adic families of $p$-adic Galois representations and it also gives us a deeper understanding of some important arithmetic objects. For example, since the prime number $691$ divides  the numerator of the rational number $\frac{\zeta(12)}{\pi^{12}}$, it divides the constant term of the unique normalized Eisenstein series $E_{12}$ of weight $12$. This 
yields a congruence modulo $691$ between $E_{12}$ and the Ramanujan cuspform $\Delta$ of weight $12$ and the existence of such a congruence helps us to understand the action of $\mrm{Gal}(\Q (\mu_{691})/\Q )$ on the $691$-part of the class group of 
$\Q (\mu_{691})$ (see \cite{ribet}).\\

Instead of the Eisenstein series $E_{12}$, we consider an automorphic representation $\Pi$ of $\mrm{GSp}(4, \A)$, where $\A$ denotes the ring of adeles of $\Q$, which is cuspidal but whose functorial lift to $\mrm{GL}(4, \A)$ is not cuspidal. Such an automorphic representation $\Pi$ is called endoscopic. In this article, we investigate congruences between $\Pi$ and cuspidal automorphic representation $\Pi'$ of $\mrm{GSp}(4, \A)$ whose functorial lift to $\mrm{GL}(4, \A)$ remains cuspidal. Such automorphic representations $\Pi'$ are called stable.\\

In \cite{hida}, Hida considers a holomorphic cuspform $f$ of $\mrm{GL}(2, \A)$ and establishes the existence of congruences between $f$ and other cuspforms of the same level and weight as $f$ modulo certain prime numbers dividing a value of the adjoint $L$-function of $f$, normalized by a suitable period. Let us recall Hida's theorem. 
For a primitive cuspform $f \in S_{\kappa}(\Gamma_1(M))$, $Z(\kappa, f)$ denotes the product
$$
Z(\kappa, f)=\prod_{\sigma: E \rightarrow \C} L(\kappa, \mrm{Sym}^2(\,^{\sigma}\!f))=\prod_{\sigma: E \rightarrow \C} L(1, \pi(^{\sigma}\!f), \mathrm{Ad})
$$
where $E$ denotes the number field generated by the Fourier coefficients of $f$, $\pi(^{\sigma}\!f)$ denotes the cuspidal automorphic representation of $\mrm{GL}(2, \A)$ attached to $^{\sigma}\! f$ and $L(s, \pi(^{\sigma}\!f),\mathrm{Ad})$ denotes the imprimitive adjoint $L$-function attached to $^{\sigma\!}f$. 
Let $c(f)$ denote the real number defined as 
\begin{equation} 
c(f)= \varepsilon(M)^{\rho} ((\kappa-1)!\cdot2^{-(\kappa+2)})^{\rho}(M \varphi(M))^{\rho} \frac{Z(\kappa, f)}{u(f) \pi^{(\kappa+1)\rho}}
\end{equation}
where $\rho$ denotes $[E:\Q]$, where $\varepsilon(M)=12, 4$ or $1$ according as $M=1,2$ or $\geq 3$ and where $u(f)$ denotes the period defined in \cite[($6.6_c$)]{hida}. It follows from \cite[Theorem 6.2]{hida} that $c(f)^2 \in \Z\backslash\{0\}$.  

The following theorem is a particular case of the main theorem of \cite{hida} stated assuming that the Nebentypus of $f$ is trivial. 
\begin{thm}[Hida]\label{theorem:Hida1}
 Let $f \in S_{\kappa}(\Gamma_1(M))$ be a primitive cuspform, with $\kappa \geq 2$ and with trivial Nebentypus. Then, if a prime $p$ such that $p > \kappa-2$ and $p \nmid 6M$ divides $c(f)^2$, there exists a congruence between $f$ and another normalized eigen cuspform $g 
\in  S_{\kappa}(\Gamma_1(M))$ which is not conjugate to $f$ modulo a prime ideal above $p$ in $\overline{\Q}$.
\end{thm}

As we announced earlier, we will prove an extension of Theorem \ref{theorem:Hida1} to congruences between a fixed endoscopic, globally generic, cuspidal automorphic representation $\Pi$ of $\GSp$ of square-free conductor and stable cuspidal automorphic representations. 
As we need it in the proof of our main result, we use a variant of Theorem \ref{theorem:Hida1} proved by Ghate \cite{ghate}. 
\par 
Let $\Gamma(s, \mathrm{Ad}(f))$ denote the $\Gamma$-factor of $L(s, \mathrm{Ad}(f))$ and let $W(f)$ be the complex constant that arises in the functional 
equation of the standard $L$-function attached to $f$ (see \cite{ghate} for the precise definitions).

\begin{thm}[Ghate]\label{theorem:Hida2}
 Let $f \in S_{\kappa}(\Gamma_1(M))$ be a primitive cuspform, with $\kappa \geq 2$ and with trivial Nebentypus. 
Let $K$ denote the number field which appears in \cite[Theorem 1]{ghate}. 
Let us fix a prime $\mathfrak{p}$ of $K$ over a prime $p>\kappa-2$ and such that $p \nmid 6M$ and let us denote by $\mathcal{O}_{(\mathfrak{p})}$ 
the ring of integers of $K$ localized at $\mathfrak{p}$. Let $c'(f) \in \mathcal{O}_{(\mathfrak{p})}$ be as follows  
\begin{equation} 
c'(f)= W(f) \Gamma (1,\mathrm{Ad}(f))\frac{L(1, \pi (f) ,\mathrm{Ad})}{\omega(f,+) \omega(f,-)}
\end{equation}
where $\omega(f,+), \omega(f,-) \in \mathbb{C}^\times / \mathcal{O}^\times_{(\mathfrak{p})}$ are complex periods 
given in \cite{ghate}. Then, if $\mathfrak{p}$ divides $c'(f)$, there exists a congruence between $f$ and another normalized eigen cuspform $g 
\in  S_{\kappa}(\Gamma_1(M))$ different from $f$ modulo $\mathfrak{p}$.
\end{thm}

To state our result, let us consider a cuspidal automorphic representation $\Pi \simeq \Pi_\infty \otimes \Pi_f$ which satisfies the conditions (i)-(v) of Section \ref{after_ichino}. In particular $\Pi_\infty$ is a generic discrete series of $\mrm{GSp}(4, \R)$ of Harish-Chandra parameter $(k+3, k'+1)$ for two integers $k \geq k' \geq 0$ and $\Pi$ has square-free paramodular conductor $N$. The conditions on $\Pi$ imply that $\Pi_f$ is defined over its rationality field, which is a number field $E(\Pi_f)$ of degree $r$ over $\Q$. Let us fix a prime number $p$ and assume that $p \notin S_{N,3} \cup S_{\mathrm{weight}} \cup S_{\mathrm{tors}} \cup S'_{\mathrm{tors}}$ where 
$S_{N,3}$, $S_{\mathrm{tors}}$, $S'_{\mathrm{tors}}$, $S_{\mathrm{weight}}$ are defined by \eqref{s_G}, \eqref{s-tors}, \eqref{s'-tors} and \eqref{s-weight} respectively. Considering the Betti cohomology of Siegel threefolds, we define a  free $\Z_{(p)}$-module $L(\Pi_f, V_{\lambda, \Z_{(p)}})$ of finite rank $2r$ endowed with a natural bilinear form
$$
\langle \,,\,\rangle: L(\Pi_f, V_{\lambda, \Z_{(p)}}) \times L(\Pi_f, V_{\lambda, \Z_{(p)}}) \rightarrow \Z_{(p)}
$$
and a period $\Omega(\Pi_f) \in \R^\times/\Z_{(p)}^\times$ (see equation \eqref{period} and Remark \ref{remark-period}). Let $\delta_1, \ldots, \delta_{2r}$ be a $\Z_{(p)}$-basis of $L(\Pi_f, V_{\lambda, \Z_{(p)}})$ and let $d(\Pi_f)= \det ( \langle \delta_i, \delta_j \rangle )_{1 \leq i, j \leq 2r}$ be the discriminant of $\langle \,,\,\rangle$. This is an element of $\Z_{(p)}$ whose image in $\Z_{(p)}/(\Z_{(p)}^\times)^2$ does not depend on the choice of the basis $\delta_1, \ldots, \delta_{2r}$. Let $C_{k,k'}
= \dfrac{(-1)^{k+k'}(k+k' +4)! (k+k' +5)!}{3^3 \cdot 5}$ and $C_N =\prod_{l\vert N} (l+l^{-1})^{-1}(l^2 +1)^{-1}$. For $x, y \in \R$ we write $x \sim y$ if there exists $s \in (\Z_{(p)}^\times)^2$ such that $x=sy$.

\begin{thm}[Theorem \ref{main1}] \label{main11} 
Let $\Pi \simeq \Pi_\infty \otimes \Pi_f$ be a cuspidal automorphic representation given above.
Then we have: 
$$
d(\Pi_f) \sim  \left( \left( \frac{2^{k+k'+13}C_{k,k'}C_N\pi^{3k+k'+12}}{k+k'+5} \right)^r \Omega(\Pi_f)^{-1} \prod_{\sigma: E(\Pi_f) \rightarrow \C} L(1, \,\!^\sigma \Pi, \mathrm{Ad}) \right)^2 .
$$
\end{thm}
We shall prove Theorem \ref{main11} at the end of Section \ref{section:Discriminant and adjoint $L$-values}. 
A direct consequence of Theorem \ref{main11} is that the real number 
$$
\left( \left( \frac{2^{k+k'+13}C_{k,k'}C_N\pi^{3k+k'+12}}{k+k'+5} \right)^r \Omega(\Pi_f)^{-1} \prod_{\sigma: E(\Pi_f) \rightarrow \C} L(1, \,\!^\sigma \Pi, \mathrm{Ad}) \right)^2
$$
is in fact a non-zero rational number which is a $p$-integer. Moreover, the assumption that $\Pi$ is endoscopic implies that it is obtained as a Weil lifting of two primitive classical elliptic modular forms $f_1$ and $f_2$ of weights $k_1=k+k'+4$ and $k_2=k-k'+2$ and level $N_1$ and $N_2$, with $N_1N_2=N$. Given a prime ideal $\mathfrak{P}$ above $p$ in $\overline{\Q}$ and another cuspidal automorphic representation $\Pi'$ of $\mrm{GSp}(4, \A)$, we write $\Pi' \equiv \Pi \pmod{\mathfrak{P}}$ if $\Pi'$ is congruent to $\Pi$ modulo $\mathfrak{P}$ (see Definition \ref{congruence-def}).  

\begin{thm}[Theorem \ref{main}] \label{main12} 

In addition to the setting given above, assume further that $p$ is prime to $N$ and is outside the finite set $S''_{\mathrm{tors}}$ which is defined by \eqref{s''-tors}. Assume the following conditions:
\begin{enumerate}
\item[\rm{(a)}] the residual $\mrm{Gal}(\overline{\Q}/\Q)$-representations $\overline{\rho}_{f_1}$ and $\overline{\rho}_{f_2}$ of $f_1$ and $f_2$ are irreducible,
\item[\rm{(b)}] the prime $\mathfrak{p}$ does not divide $c'(f_1)$ nor $c'(f_2)$ for any prime $\mathfrak{p}$ above $p$ in $\overline{\mathbb{Q}}$,   
\item[\rm{(c)}] the prime $p$ divides $$
\left( \left( \frac{2^{k+k'+13}C_{k,k'}C_N\pi^{3k+k'+12}}{k+k'+5} \right)^r \Omega(\Pi_f)^{-1} \prod_{\sigma: E(\Pi_f) \rightarrow \C} L(1, \,\!^\sigma \Pi, \mathrm{Ad}) \right)^2, 
$$
\end{enumerate}
Then, there exists a prime divisor $\mathfrak{P}$ of $p$ in $\overline{\Q}$ and a cuspidal representation $\Pi' \simeq \bigotimes_v' \Pi'_v$ of $G(\A)$ such that 
\begin{enumerate}
\item[\rm{(1)}] the non archimedean part $\Pi'_f$ of $\Pi'$ satisfies $(\Pi'_f)^{K_N} \neq 0$,
\item[\rm{(2)}] the representation $\Pi'_\infty$ is a discrete series with the same parameter as $\Pi_\infty$,
\item[\rm{(3)}] the cuspidal representation $\Pi'$ is stable,
\item[\rm{(4)}] we have $\Pi' \not\simeq \,^\sigma\!\,\Pi$ for all $\sigma \in \mrm{Aut}(\C)$,
\item[\rm{(5)}] we have $\Pi' \equiv \Pi \pmod{\mathfrak{P}}$.
\end{enumerate}
\end{thm} 
We shall prove Theorem \ref{main12} at the end of Section \ref{section:The congruence criterion}. 

Note that by the work of Lan-Suh \cite{lan-suh}, the primes in $S_{\mrm{tors}}$ and $S'_{\mrm{tors}}$ can be shown to be smaller than and explicit bound when the weight is sufficiently regular. On the other hand $S''_{\mrm{tors}}$ is a subset of the prime numbers dividing the torsion in the cohomology of the boundary and hence is more delicate.\\

Before our work, Hida's theorem has been extended to  cuspforms of $\mrm{GL}(2)$ over an imaginary quadratic field by Urban \cite{urban}, to Hilbert modular cuspforms by Ghate \cite{ghate} and Dimitrov \cite{dimitrov}, to cuspforms of $\mrm{GL}(2)$ over an arbitrary number field by Namikawa \cite{namikawa} and to cuspforms of $\mrm{GL}(n)$ over an arbitrary number filed by Balasubramanyam and Raghuram \cite{br}. Congruences between holomorphic Yoshida lifts and stable cusp forms on $\mrm{GSp}(4)$ has also been discussed in \cite{katsurada}, 
\cite{bdsp} and \cite{ak} but the methods of these papers are entirely different from ours.\\

In the proof of Theorem \ref{main11}, the work of Chen-Ichino \cite[Theorem 1.1]{ch19} relating the Petersson norm of a suitably normalized cuspform $\varphi \in \Pi$ to $L(1, \Pi, \mathrm{Ad})$ plays an important role (\cite{ch19} is a refinement of an unpublished preprint \cite{ichino}). The assumptions (i)-(iv) on $\Pi$ stated above are all already present as hypothesis of  loc. cit. Theorem 1.1, which also covers the case where $\Pi$ is stable. 
Assumption (v) says that $\Pi$ is endoscopic. 
\par 
One of our important contributions for Theorem \ref{main11} is to compute the discriminant $d(\Pi_f)$ in terms of the explicit constant $C_{k,k'}$, the period $\Omega(\Pi_f)$ and the Petersson norms of the normalized cuspforms $^\sigma \!\,\varphi \in\, ^\sigma \!\,\Pi$. Note that in \cite{br}, the rational number analogous to $C_{k,k'}$ is not explicitly computed as in the present work. It seems difficult to perform a similar explicit computation for higher rank reductive groups. Furthermore, let us emphasize that this work is the first application of Hida's ideas to a reductive group different from $\mrm{GL}(n)$ and in particular to the study of endoscopic congruences. Moreover this is the first work considering endoscopic congruences for generic, hence non-holomorphic, cuspforms on $\mrm{GSp}(4)$.\\

\textbf{Aknowledgments.} The first-named author aknowledges support of the ANR Formules des Traces Relatives, P\'eriodes, Fonctions $L$ et Analyse Harmonique [ANR-13-BS01-0012]. The second-named author is supported by KAKENHI (Fostering Joint International Research: Garant Number 16KK0100, 
Fund for the Promotion of Joint International Research and 
Grant-in-Aid for Scientic Research (B): Grant Number 26287005) 
of JSPS. This work progressed during the stay of the second-named author 
at IMJ-PRG and Universit\'e Paris 7 in the year 2017-2018. The second-named 
expresses his gratitude to IMJ-PRG and Universit\'e Paris 7 for their hospitality. We would like also to thank Atsushi Ichino, Giovanni Rosso and Jacques Tilouine for encouragements and feedback.

\section{Notation and conventions}

\subsection{The algebraic group $\GSp$ and its algebraic representations} \label{introGSp} Let $I_2$ be the identity matrix of size two and let $\mathcal{J}$ be the symplectic form whose matrix in the canonical basis of $\Z^4$ is
$$
\mathcal{J}=
\begin{pmatrix}
 & I_2\\
-I_2 & \\
\end{pmatrix}.
$$
The symplectic group $\mathrm{GSp}(4)$ is defined as
$$
\mrm{GSp}(4)=\{g \in \mathrm{GL}(4) \,|\, ^t\!g \mathcal{J} g = \nu(g) \mathcal{J}, \, \nu(g) \in \mathbb{G}_m\}.
$$
This is  a $\Z$-group scheme that we will denote by $G$ in what follows. Then $\nu: G \rightarrow \mathbb{G}_m$ is a character and the derived group of $G$ is $\mathrm{Sp}(4)= \mathrm{Ker} \, \nu$. We denote by $T \subset G$ the diagonal maximal torus defined as
$$
T=\{ \mathrm{diag}(\alpha_1, \alpha_2, \alpha_1^{-1} \nu, \alpha_2^{-1} \nu)| \, \alpha_1, \alpha_2, \nu \in \mathbb{G}_m \}
$$
and by $B=TU$ the standard Borel subgroup of upper triangular matrices in $G$ where $U$ is the unipotent radical
$$
U=\left\{ \begin{pmatrix}
 1 & x_0 & &\\
   &  1  & &\\
   &   & 1& \\
   &   & -x_0 & 1 \\
\end{pmatrix}\begin{pmatrix}
 1 &  & x_1 & x_2\\
 &  1 & x_2 & x_3 \\
   &   &1 & \\
   &   &  &1 \\
\end{pmatrix}\, \middle| \, x_0, x_1, x_2, x_3 \in \mathbb{G}_a \right\}
$$
We identify the group $X^*(T)$ of algebraic characters of $T$ to the subgroup of $\mathbb{Z}^2 \oplus \mathbb{Z}$ of triples $(k, k', c)$ such that $k+k' \equiv c \pmod 2$ via
$$
\lambda(k, k', c): \mathrm{diag}(\alpha_1, \alpha_2, \alpha_1^{-1} \nu, \alpha_2^{-1} \nu) \longmapsto \alpha_1^k \alpha_2^{k'} \nu^{\frac{c-k-k'}{2}}.
$$
Let $\rho_1=\lambda(1, -1, 0)$ be the short simple root and $\rho_2=\lambda(0, 2, 0)$ be the long simple root. Then the set $R \subset X^*(T)$ of roots of $T$ in $G$ is
$$
R=\{ \pm \rho_1, \pm \rho_2, \pm (\rho_1 + \rho_2), \pm (2 \rho_1+\rho_2)\}
$$
and the subset $R^+ \subset R$ of positive roots with respect to $B$ is
$$
R^+=\{ \rho_1, \rho_2, \rho_1+ \rho_2, 2 \rho_1+ \rho_2 \}.
$$
Then, the set of dominant weights is the set of $\lambda(k, k', c)$ such that $k \geq k' \geq 0$. For any dominant weight $\lambda$, there is an irreducible algebraic representation $V_{\lambda, \Q}$ of $G_{\Q}$ in a finite dimensional $\Q$-vector space, of highest weight $\lambda$, unique up to isomorphism, and all isomorphism classes of irreducible algebraic representations of $G_{\Q}$ are obtained in this way. We will be interested in the specific realization of $V_{\lambda, \Q}$ obtained by Weyl's construction as follows. Let $\mathrm{Std}_{\Q} \simeq V_{\lambda(1,0,1), \Q}$ be the standard 4-dimensional representation of $G_{\Q}$, let $\lambda=\lambda(k,k',c)$ be a dominant weight, let $d=k+k'$. If $d \geq 2$, for two integers $1 \leq p, q \leq d$, let $\Psi_{p,q}: \mathrm{Std}_{\Q}^{\otimes d} \rightarrow \mathrm{Std}_{\Q}^{\otimes(d-2)}$ denote the map defined by
\begin{equation} \label{psipq}
\Psi_{p,q}: v_1 \otimes \ldots \otimes v_d \mapsto \, ^t\!v_p \mathcal{J} v_q v_1 \otimes \ldots \otimes \hat{v}_p \otimes \ldots \otimes \hat{v}_q \otimes \ldots \otimes v_d
\end{equation}
Let $\mathrm{Std}_{\Q}^{\langle d \rangle} \subset \mathrm{Std}_{\Q}^{\otimes d}$ be the intersection of the kernels of all the $\Psi_{p,q}$. Let $t=\frac{c-k-k'}{2}$. Then 
$$
V_{\lambda, \Q} = \left( \mathrm{Std}_{\Q}^{\langle d \rangle} \cap \mbb{S}_\lambda(\mathrm{Std}_{\Q}) \right) \otimes \nu^{\otimes t}
$$ 
where $\mbb{S}_\lambda(\mathrm{Std}_{\Q}) \subset \mathrm{Std}_{\Q}^{\otimes d}$ is the image of an explicit element $c_\lambda \in \Z[\mathcal{S}_d]$ acting on $\mathrm{Std}_{\Q}^{\otimes d}$ (see \cite[Theorem 17.11]{fulton-harris}). Here $\mathcal{S}_d$ denotes the symmetric group of order $d!$. By replacing $\mathrm{Std}_{\Q}$ by the standard representation $\mathrm{Std}_{\Z}$ of $G$ over $\Z$, we obtain a free $\Z$-module $V_{\lambda, \Z}$ endowed with a linear action of $G$ such that $V_{\lambda, \Z} \otimes \Q = V_{\lambda, \Q}$.

\subsection{Measures} \label{section-measures} Consider the unitary group $\mathrm{U}(2)=\{g \in \mathrm{GL}(2, \mathbb{C}) \,|\, ^t\!\,\overline{g}g= I_2 \}$ where $\overline{g}$ denotes the complex conjugate of $g$. The map $\kappa: \mathrm{U}(2) \longrightarrow \g{Sp}(4, \mathbb{R})$ defined by $$g=A+iB \longmapsto \begin{pmatrix}
A & B\\
-B & A\\
\end{pmatrix},$$
where $A$ and $B$ denote the real and imaginary parts of $g$, identifies $\g{U}(2)$ with a maximal compact subgroup $K_\infty$ of  $\g{Sp}(4, \mathbb{R})$. Let $K'_\infty$ denote the subgroup $\R^\times_+ K_\infty$ of $G(\R)$, where $\R^\times_+$ denotes the connected component of the center of $G(\R)$. Let $dk_\infty$ be the unique Haar measure on $K_\infty$ such that $\mrm{vol}(K_\infty, dk_\infty)=1$. Let $d\mu$ be a left translation invariant measure on $G(\R)_+/K'_\infty$. Given a measurable function $f: G(\R)_+/\R_+^\times \rightarrow \C$, the function $\overline{f}: G(\R)_+/K'_\infty \rightarrow \C$ defined as $$\overline{f}(\overline{g}_\infty) = \int_{K_\infty} f(g_\infty k) dk_\infty$$ 
where $g_\infty$ is a lift of $\overline{g}_\infty$ by the canonical projection $G(\R)_+/\R_+^\times \rightarrow G(\R)_+/K'_\infty$, is well defined. We define the left translation invariant measure $dg_\infty$ on $G(\R)_+/\R_+^\times$ by the formula
\begin{equation} \label{measures}
\int_{G(\R)_+/\R^\times_+} f(g_\infty) dg_\infty=\int_{G(\R)_+/K'_\infty} \overline{f}d\mu.
\end{equation}
Let $$
\mathcal{H}_+=\{Z=X+iY \in \mathfrak{gl}_{2, \C}\,|\, \,\!^tZ=Z, Y>0\}=G(\R)_+/K'_\infty
$$ be Siegel upper half-plane of genus $2$. 

\begin{lem} \label{invariant-measures} The following statements hold.
\begin{enumerate}
\item[\rm{(1)}] The measure $dXdY/\det(Y)^3$ on $\mathcal{H}_+$ is left translation invariant by $G(\R)_+$.
\item[\rm{(2)}] For any measure $d\mu$ on $\mathcal{H}_+$ which is left translation invariant by $G(\R)_+$, then there exists $c \in \R_+^\times$ such that $d\mu=c dXdY/\det(Y)^3$.
\end{enumerate}
\end{lem}

\begin{proof} The first statement is a particular case of \cite[Proposition 1.2.9]{andrianov}. Let us prove the second statement. Let $d\mu$ be a measure on $\mathcal{H}_+=G(\R)_+/K'_\infty$ which is left translation invariant by $G(\R)_+$. Let $d\nu$ denote the measure $dXdY/\det(Y)^3$. By the construction \eqref{measures} above, we associate to $d\mu$ and $d\nu$ Haar measures $dg_{\infty, \mu}$ and $dg_{\infty, \nu}$ on the Lie group $G(\R)_+/\R_+^\times$. There exists $c \in \R_+^\times$ such that $dg_{\infty, \mu}=c dg_{\infty, \nu}$. Let $s: G(\R)_+/\R_+^\times \rightarrow G(\R)_+/K'_\infty$ denote the canonical projection. By an easy computation we have 
$
s_*(dg_{\infty, \mu})=\mrm{vol}(K_\infty, dk_\infty)d\mu=d\mu
$
and similarly $s_*(dg_{\infty, \nu})=d\nu$. As a consequence $d\mu=c d\nu$ as claimed.
\end{proof}

\subsection{Representations of $K_\infty$ and discrete series classification} \label{discrete_series_classification}  Let $\mathfrak{k}$ and $\mathfrak{k}'$ denote the Lie algebra of $K_\infty$ and $K'_\infty$ respectively. Let $\mathfrak{k}_\mathbb{C}$ and $\mathfrak{k}'_{\C}$ denote their complexifications. The differential of $\kappa$ induces an isomorphism of Lie algebras $d\kappa: \mathfrak{gl}_{2, \mathbb{C}} \simeq \mathfrak{k}_\mathbb{C}$. Let $\mathfrak{sp}_4$ denote the Lie algebra of $\g{Sp}(4, \mathbb{R})$ and by $\mathfrak{sp}_{4, \mathbb{C}}$ its complexification. A compact Cartan subalgebra of $\mathfrak{sp}_4$ is defined as $\mathfrak{h}=\mathbb{R}T_1 \oplus \mathbb{R}T_2$ where
\begin{eqnarray*}
T_1 &=& d\kappa \left(  \begin{pmatrix}
i &  \\
 &  \\
\end{pmatrix}\right)=\begin{pmatrix}
 &   & 1 & \\
 &  &  & \\
-1 &  &  & \\
  &  &  & \\
\end{pmatrix},\\
 T_2 &=& d\kappa \left(  \begin{pmatrix}
 &  \\
 &  i\\
\end{pmatrix}\right)=\begin{pmatrix}
 &  &  & \\
 &  &  & 1\\
 &  &  & \\
 & -1 &  & \\
\end{pmatrix}.
\end{eqnarray*}
Define a $\mathbb{C}$-basis of $\mathfrak{h}_\mathbb{C}^*$ by $e_1(T_1)=i, e_1(T_2)=0, e_2(T_1)=0, e_2(T_2)=i$. The root system $\Delta$ of the pair $(\mathfrak{sp}_{4, \mathbb{C}}, \mathfrak{h}_\mathbb{C})$ is $
\Delta=\{ \pm 2 e_1, \pm 2 e_2, \pm(e_1 \pm e_2)\}.$
We denote by $\Delta_c$, respectively $\Delta_{nc}$, the set of compact, respectively non-compact roots in $\Delta$. We have $\Delta_c=\{ \pm (e_1-e_2)\}$ and $\Delta_{nc}=\Delta - \Delta_c$. We choose the set of positive roots as $\Delta^+=\{e_1-e_2, 2e_1, e_1+e_2, 2e_2\}$. Then, the set of compact, respectively non-compact, positive roots is $\Delta^+_c=\Delta_c \cap \Delta^+$, respectively $\Delta_{nc}^+=\Delta_{nc} \cap \Delta^+$. For each symmetric matrix $Z \in \mathfrak{gl}_{2, \mathbb{C}}$, define the element $p_\pm(Z)$ of $\mathfrak{sp}_{4}$ by
$$
p_\pm(Z)=\begin{pmatrix}
Z & \pm iZ\\
\pm iZ & -Z\\
\end{pmatrix}.
$$
Let us consider the generator 
$$
p_+ \left(  \begin{pmatrix}
1 & \\
 &  \\
\end{pmatrix}\right) \wedge p_+ \left(  \begin{pmatrix}
 & 1\\
1 &  \\
\end{pmatrix}\right)\wedge p_+ \left(  \begin{pmatrix}
 & \\
 &  1\\
\end{pmatrix}\right) \wedge p_- \left(  \begin{pmatrix}
1 & \\
 &  \\
\end{pmatrix}\right) \wedge p_- \left(  \begin{pmatrix}
 & 1\\
1 &  \\
\end{pmatrix}\right) \wedge p_- \left(  \begin{pmatrix}
 & \\
 &  1\\
\end{pmatrix}\right)
$$
of the one-dimensional $\C$-vector space $\bigwedge^6 \mathfrak{sp}_{4, \C}/\mathfrak{k}_{\C}$. To this generator is associated a non-zero left translation invariant measure $d\mu$ on $\mrm{Sp}(4, \R)/K_\infty=G(\R)_+/K'_\infty$ in a standard way. Let $c \in \R_+^\times$ be the constant given by the second assertion of Lemma \ref{invariant-measures}. Let us denote by $\sqrt[6]{c}$ the positive sixth root of $c$ in $\R_+^\times$. Let $X_{(\alpha_1, \alpha_2)} \in \mathfrak{sp}_{4, \mathbb{C}}$ be defined as
$$
X_{\pm(2, 0)}=\sqrt[6]{c} p_\pm \left(  \begin{pmatrix}
1 & \\
 &  \\
\end{pmatrix}\right), X_{\pm(1, 1)}=\sqrt[6]{c} p_\pm \left(  \begin{pmatrix}
 & 1\\
1 &  \\
\end{pmatrix}\right), X_{\pm(0, 2)}=\sqrt[6]{c} p_\pm \left(  \begin{pmatrix}
 & \\
 &  1\\
\end{pmatrix}\right).
$$
It follows from an easy computation that $X_{(\alpha_1, \alpha_2)}$ is a root vector corresponding to the non-compact root $(\alpha_1, \alpha_2)=\alpha_1 e_1+ \alpha_2 e_2$. If we set
\begin{equation} \label{pplus-pmoins}
\mathfrak{p}^\pm= \bigoplus_{\alpha \in \Delta_{nc}^+} \mathbb{C} X_{\pm \alpha},
\end{equation}
we have the Cartan decomposition $\mathfrak{sp}_{4, \mathbb{C}}= \mathfrak{k}_\mathbb{C} \oplus \mathfrak{p}^+ \oplus \mathfrak{p}^-$. Furthermore the inclusion $\mathfrak{sp}_{4\,\C} \subset \mathfrak{g}_{\C}$ induces a canonical isomorphism $\mathfrak{sp}_{4\,\C}/\mathfrak{k}_{\C} = \mathfrak{g}_{\C}/\mathfrak{k}_{\C}''$. In particular $\mathfrak{g}_{\C}/\mathfrak{k}_{\C}'' = \mathfrak{p}^+ \oplus \mathfrak{p}^-$.\\

Integral weights are defined as the $(k, k')=ke_1+k'e_2 \in \mathfrak{h}_\mathbb{C}^*$ with $k, k' \in \mathbb{Z}$ and an integral weight is dominant for $\Delta_c^+$ if $k \geq k'$. Assigning its highest weight to a finite dimensional irreducible complex representation $\tau$ of $K_\infty$, we define a bijection between isomorphism classes of finite-dimensional irreducible complex representations of $K_\infty$ and dominant integral weights, whose inverse will be denoted by $(k ,k') \longmapsto \tau_{(k, k')}$. Let $(k, k')$ be a dominant integral weight and let $d=k-k'$. Then $\dim_\mbb{C} \tau_{(k, k')}=d+1$. More precisely, there exists a basis $(\text{\boldmath$v$}_s)_{0 \leq s \leq d}$ of $\tau_{(k, k')}$, such that
\begin{eqnarray}
\tau_{(k, k')} \left( d\kappa\begin{pmatrix}
1 &  \\
 &  \\
\end{pmatrix} \right) \text{\boldmath$v$}_s &=& (s+k') \text{\boldmath$v$}_s,\\
\tau_{(k, k')} \left( d\kappa\begin{pmatrix}
 &  \\
 & 1 \\
\end{pmatrix} \right) \text{\boldmath$v$}_s &=& (-s+k) \text{\boldmath$v$}_s,\\
\label{formula5}
\tau_{(k, k')} \left( d\kappa\begin{pmatrix}
 & 1 \\
 &  \\
\end{pmatrix} \right) \text{\boldmath$v$}_s &=& (s+1) \text{\boldmath$v$}_{s+1},\\
\label{formula6}
\tau_{(k, k')} \left( d\kappa\begin{pmatrix}
 &  \\
1 &  \\
\end{pmatrix} \right) \text{\boldmath$v$}_s &=& (d-s+1) \text{\boldmath$v$}_{s-1}
\end{eqnarray}
which we call a standard basis of $\tau_{(k, k')}$. In the identities above, we agree to use the convention $\text{\boldmath$v$}_{-1}=\text{\boldmath$v$}_{d+1}=0$. Note that $(\mrm{Ad}, \mathfrak{p}^+)$ is equivalent to $\tau_{(2,0)}$ and that $(\mrm{Ad}, \mathfrak{p}^-)$ is equivalent to $\tau_{(0,-2)}$. Under the identification $(\mrm{Ad}, \mathfrak{p}^+) \simeq \tau_{(2,0)}$, the basis $(\text{\boldmath$v$}_2, \text{\boldmath$v$}_1, \text{\boldmath$v$}_0)=(X_{(2,0)}, X_{(1,1)}, X_{(0,2)})$ is a standard basis and the under the identification $(\mrm{Ad}, \mathfrak{p}^-) \simeq \tau_{(0,-2)}$, the basis $(\text{\boldmath$v$}_2, \text{\boldmath$v$}_1, \text{\boldmath$v$}_0)=(X_{(0,-2)}, -X_{(-1,-1)}, X_{(-2, 0)})$ is a standard basis. As the weights of $\bigwedge^2 \mathfrak{p}^+ \otimes_{\C} \mathfrak{p}^-$ are the sums of two distinct weights of $\mathfrak{p}^+$ and of a weight of $\mathfrak{p}^-$, as $\C[K_\infty]$-modules we have
\begin{eqnarray*}
\bigwedge^2 \mathfrak{p}^+ \otimes_{\C} \mathfrak{p}^- &=& \tau_{(3, -1)} \oplus \tau_{(2,0)} \oplus \tau_{(1,1)},\\
\mathfrak{p}^+ \otimes_{\C} \bigwedge^2 \mathfrak{p}^- &=& \tau_{(1, -3)} \oplus \tau_{(0,-2)} \oplus \tau_{(-1,-1)}.
\end{eqnarray*}

\begin{lem} \label{std-basis} 
\begin{enumerate}
\item[\rm{(1)}] The following is a standard basis of $\tau_{(3,-1)} \subset \bigwedge^2 \mathfrak{p}^+ \otimes_{\C} \mathfrak{p}^-$:
\begin{align*}
& \text{\boldmath$w$}_4=X_{(2,0)} \wedge X_{(1,1)} \otimes X_{(0,-2)},\\
& \text{\boldmath$w$}_3=-X_{(2,0)} \wedge X_{(1,1)} \otimes X_{(-1,-1)}+2 X_{(2,0)} \wedge X_{(0,2)} \otimes X_{(0,-2)},\\
& \text{\boldmath$w$}_2=X_{(2,0)} \wedge X_{(1,1)} \otimes X_{(-2,0)}-2 X_{(2,0)} \wedge X_{(0,2)} \otimes X_{(-1,-1)}+ X_{(1,1)} \wedge X_{(0,2)} \otimes X_{(0,-2)},\\
& \text{\boldmath$w$}_1=2 X_{(2,0)} \wedge X_{(0,2)} \otimes X_{(-2,0)}- X_{(1,1)} \wedge X_{(0,2)} \otimes X_{(-1,-1)},\\
& \text{\boldmath$w$}_0=X_{(1,1)} \wedge X_{(0,2)} \otimes X_{(-2,0)}.
\end{align*}
\item[\rm{(1)}] The following is a standard basis of $\tau_{(2,0)} \subset \bigwedge^2 \mathfrak{p}^+ \otimes_{\C} \mathfrak{p}^-$:
\begin{align*}
& \text{\boldmath$x$}_2= X_{(2,0)} \wedge X_{(1,1)} \otimes X_{(-1,-1)}+2 X_{(2,0)} \wedge X_{(0,2)} \otimes X_{(0,-2)},\\
& \text{\boldmath$x$}_1= -2 X_{(2,0)} \wedge X_{(1,1)} \otimes X_{(-2,0)}+2 X_{(1,1)} \wedge X_{(0,2)} \otimes X_{(0,-2)},\\
& \text{\boldmath$x$}_0= -2 X_{(2,0)} \wedge X_{(0,2)} \otimes X_{(-2,0)}- X_{(1,1)} \wedge X_{(0,2)} \otimes X_{(-1,-1)}.
\end{align*}
\item[\rm{(1)}] The following is a standard basis of $\tau_{(1,1)} \subset \bigwedge^2 \mathfrak{p}^+ \otimes_{\C} \mathfrak{p}^-$:
$$
\text{\boldmath$y$}_0= X_{(2,0)} \wedge X_{(1,1)} \otimes X_{(-2,0)}+X_{(2,0)} \wedge X_{(0,2)} \otimes X_{(-1,-1)}+X_{(1,1)} \wedge X_{(0,2)} \otimes X_{(0,-2)}.
$$
\end{enumerate}
\end{lem}

\begin{lem} \label{projection-matrix} In the basis 
$$
X_{(2,0)} \wedge X_{(1,1)} \otimes X_{(0,-2)},\, X_{(2,0)} \wedge X_{(1,1)} \otimes X_{(-1,-1)},\, X_{(2,0)} \wedge X_{(1,1)} \otimes X_{(-2,0)},
$$
$$
X_{(2,0)} \wedge X_{(0,2)} \otimes X_{(0,-2)},\, X_{(2,0)} \wedge X_{(0,2)} \otimes X_{(-1,-1)},\, X_{(2,0)} \wedge X_{(0,2)} \otimes X_{(-2,0)},
$$
$$
X_{(1,1)} \wedge X_{(0,2)} \otimes X_{(0,-2)},\, X_{(1,1)} \wedge X_{(0,2)} \otimes X_{(-1,-1)},\, X_{(1,1)} \wedge X_{(0,2)} \otimes X_{(-2,0)}
$$
of $\bigwedge^2 \mathfrak{p}^+ \otimes_{\C} \mathfrak{p}^-$ the matrix of the projection $p: \bigwedge^2 \mathfrak{p}^+ \otimes_{\C} \mathfrak{p}^- \rightarrow \tau_{(3,-1)}$ is
$$
\begin{pmatrix}
1 &  0 & 0 & 0 &  0 & 0 & 0 &  0 & 0\\
0 &  \frac{1}{2} & 0 & -\frac{1}{4} & 0 & 0 & 0 & 0 & 0\\
0 &  0 & \frac{1}{6}& 0 & -\frac{1}{3} & 0 & \frac{1}{6} & 0 & 0\\
0 &  -1 & 0 & \frac{1}{2} & 0 & 0 & 0 & 0 & 0\\
0 &  0 & -\frac{1}{3} & 0 & \frac{2}{3} & 0 & -\frac{1}{3} & 0 & 0\\
0 &  0 & 0 & 0 & 0 & \frac{1}{2} & 0 & -1 & 0\\
0 &  0 & \frac{1}{6} & 0 & -\frac{1}{3} & 0 & \frac{1}{6} & 0 & 0\\
0 &  0 & 0 & 0 & 0 & -\frac{1}{4} & 0 & \frac{1}{2} & 0\\
0 &  0 & 0 & 0 & 0 & 0 & 0 & 0 & 1\\
\end{pmatrix}.
$$
\end{lem}

We will denote by $W_{K_\infty}$ the Weyl group of $(\mathfrak{k}_\mathbb{C}, \mathfrak{h}_\mathbb{C})$. According to the classification theorem \cite[Thm. 9.20]{knapp}, as $W_{K_\infty}$ has $4$ elements, we have:

\begin{pro} \label{lpaquet} Let $G(\mbb{R})_+$ be the identity component of $G(\mbb{R})$, let $\xi$ be a character of $\R^\times$ and let $(k, k') \in \mathfrak{h}_\mathbb{C}^*$  be an integral weight. Assume $k \geq k' \geq 0$. Then, there exist $4$ isomorphism classes $\Pi_\infty^{3,0}$, $\Pi_\infty^{2,1}$, ${\Pi}_\infty^{1,2}$, ${\Pi}_\infty^{0,3}$ of irreducible discrete series representations of $G(\mathbb{R})_+$ with Harish-Chandra parameter $(k+2, k'+1)$ and central character $\xi$. Furthermore, the restrictions of these representations to $K_\infty$ contain as minimal $K_\infty$-types the irreducible representations $\tau_{(k+3, k'+3)}$, $\tau_{(k+3, -k'-1)}$, $\tau_{(k'+1, -k-3)}$, $\tau_{(-k'-3, -k-3)}$ respectively and these occur with multiplicity $1$.
\end{pro}

In the proposition above, the discrete series $\Pi_\infty^{3,0}$ is holomorphic, $\Pi_\infty^{2,1}$ and $\Pi_\infty^{1,2}$ are generic, which means that they have a Whittaker model and $\Pi_\infty^{0,3}$ is antiholomorphic. In what follows, we will denote by $\Pi_\infty^H$ and $\Pi_\infty^W$ the discrete series of $\mrm{GSp}(4, \R)$ defined as  
\begin{eqnarray*}
\Pi_\infty^H=\mrm{Ind}_{G(\R)_+}^{G(\R)} \Pi_\infty^{3,0}=\mrm{Ind}_{G(\R)_+}^{G(\R)} \Pi_\infty^{0,3},\\
\Pi_\infty^W=\mrm{Ind}_{G(\R)_+}^{G(\R)} \Pi_\infty^{2,1}=\mrm{Ind}_{G(\R)_+}^{G(\R)} \Pi_\infty^{1,2}.
\end{eqnarray*}
In particular, we have
\begin{eqnarray*}
\Pi_\infty^H|_{G(\R)_+}= \Pi_\infty^{3,0} \oplus \Pi_\infty^{0,3},\\
\Pi_\infty^W|_{G(\R)_+}= \Pi_\infty^{2,1} \oplus\Pi_\infty^{1,2}.
\end{eqnarray*}

\subsection{Hecke algebras} \label{heckealgsection}Let $l$ be a prime number and let $K_l \subset G(\Q_l)$ be a compact open subgroup. Let $\mathcal{H}_l^{K_l}$ be the Hecke algebra of $\Z$-valued compactly supported functions on $G(\Q_l)$, which are invariant by translation on the left and on the right by $K_l$, with product given by the convolution product with respect to the Haar measure $dx_l$ on $G(\Q_l)$ giving measure  $1$ to the maximal compact subgroup $G(\Z_l)$ of $G(\Q_l)$. If $K=\prod'_l K_l \subset G(\A_f)$ is a compact open subgroup, we denote by $\mathcal{H}^K$ the restricted tensor product
$
\mathcal{H}^K=\bigotimes'_l \mathcal{H}_l^{K_l}$. For any ring $R$, we shall denote by $\mathcal{H}^K_R$ the $R$-algebra $\mathcal{H}^K \otimes R$. Let $\Pi_l$ denote a smooth admissible representation of $G(\Q_l)$. Then $\Pi_l^{K_l}$ is endowed with the action of $\mathcal{H}^{K_l}_{\C}$ as follows. For $f \in \mathcal{H}^{K_l}_{\C}$ and $\chi \in V_{\Pi_l}$ a vector invariant by $K_l$, we define $f.\chi$ as
$$
f.\chi=\int_{G(\Q_l)} \Pi_l(x)(\chi) f(x) dx_l.
$$

\begin{lem} \label{unimod} Let $K_l \subset G(\Q_l)$ be a compact open subgroup. Then, for any $g \in G(\Q_l)$ we have
$$
[K_l: gK_l g^{-1} \cap K_l]=[K_l: g^{-1}K_l g \cap K_l].
$$ 
\end{lem}

\begin{proof} Let  ${1}_{g K_l g^{-1} \cap K_l}$ and ${1}_{g^{-1}K_l g \cap K_l}$ denote the characteristic functions of the subgroups $g K_l g^{-1} \cap K_l$ and $g^{-1}K_l g \cap K_l$ of $K_l$ respectively. We have
\begin{eqnarray*}
\mrm{vol}(g K_l g^{-1} \cap K_l, dx_l)  &=& \int_{G(\Q_l)} {1}_{g K_l g^{-1} \cap K_l}(x_l)dx_l\\
&=& \int_{G(\Q_l)} {1}_{gK_lg^{-1} \cap K_l}(gx_lg^{-1})dx_l\\
&=& \int _{G(\Q_l)} {1}_{g^{-1}K_l g \cap K_l}(x_l) dx_l\\
&=& \mrm{vol}(g^{-1} K_l g \cap K_l, dx_l)
\end{eqnarray*}
where the second equality follows the change of variable $x_l \rightarrow g^{-1}x_lg$ and from the equality $d(gx_lg^{-1})=dx_l$ which follows from the unimodularity of $G(\Q_l)$ (see \cite[Proposition V.5.4]{renard}). The conclusion follows.
\end{proof}

Let $dx^{\mathrm{Tam}}$ denote the Tamagawa measure on $G(\A)$. Given two cusp forms $\varphi_1$ and $\varphi_2$ on $G(\A)$ with trivial central character, let $\langle \varphi_1, \varphi_2 \rangle$ denote the Petersson inner product
$$
\langle \varphi_1, \varphi_2 \rangle=\int_{Z(\A)G(\Q)\backslash G(\A)} \varphi_1(x) \overline{\varphi_2(x)} dx^{\mathrm{Tam}}.
$$

\begin{pro} \label{adjoint-hecke}Assume that $\varphi_1$ and $\varphi_2$ have trivial central character and are invariant by right translation by $K_l$. Let $g \in G(\Q_l)$, let $T_g \in \mathcal{H}^{K_l}$ be the characteristic function of $K_l g K_l$ and let $T_{\nu(g)g^{-1}}$ be the characteristic function of $K_l \nu(g)g^{-1} K_l$. Then
$$
\langle T_g \varphi_1, \varphi_2 \rangle = \langle \varphi_2, T_{\nu(g)g^{-1}} \varphi_2 \rangle.
$$
\end{pro}

\begin{proof} Let $n$ denote the integer
\begin{equation} \label{integer}
n=[K_l: gK_l g^{-1} \cap K_l]=[K_l: g^{-1}K_l g \cap K_l]
\end{equation}
where the second equality follows from Lemma \ref{unimod}. According to equality \eqref{integer} and to \cite[Lemma 5.5.1 (c)]{diamond-shurman} , there exists $\beta_1, \ldots, \beta_n \in G(\Q_l)$ such that
\begin{equation} \label{left-right}
K_l g K_l=\bigsqcup_{j=1}^n \beta_j K_l=\bigsqcup_{j=1}^n K_l \beta_j.
\end{equation}
Note that \cite[Lemma 5.5.1 (c)]{diamond-shurman} is proved in a global situation and for $\mrm{SL}_2$ but that the proof is purely group theoretical and hence works when applied to our situation once $\Gamma$ in loc. cit. is replaced by $G(\Z_l)$ and $\alpha$ in loc. cit. is replaced by $g$. Hence we have
\begin{eqnarray*}
\langle T_g \varphi_1, \varphi_2 \rangle &=& \int_{Z(\A)G(\Q)\backslash G(\A)} \int_{G(\Q_l)} \varphi_1(xh_l)\textbf{1}_{K_l g K_l}(h_l)\overline{\varphi_2(x)} dh_l  dx^{\mathrm{Tam}}\\
&=& \int_{Z(\A)G(\Q)\backslash G(\A)} \sum_{j=1}^n \int_{K_l} \varphi_1(x\beta_j h_l) \overline{\varphi_2(x)} dh_l  dx^{\mathrm{Tam}}\\
&=& \mrm{vol}(K_l, dh_l) \int_{Z(\A)G(\Q)\backslash G(\A)} \sum_{j=1}^n \varphi_1(x\beta_j) \overline{\varphi_2(x)} dx^{\mathrm{Tam}}\\
&=&  \mrm{vol}(K_l, dh_l) \int_{Z(\A)G(\Q)\backslash G(\A)} \sum_{j=1}^n \varphi_1(x) \overline{\varphi_2(x\beta_j^{-1})} dx^{\mathrm{Tam}}\\
&=& \int_{Z(\A)G(\Q)\backslash G(\A)} \varphi_1(x) \sum_{j=1}^n \int_{K_l} \varphi_2(x\beta_j^{-1}h_l)dh_l dx^{\mathrm{Tam}}\\
&=& \int_{Z(\A)G(\Q)\backslash G(\A)} \varphi_1(x) \sum_{j=1}^n \int_{K_l} \varphi_2(x\nu(g)\beta_j^{-1}h_l)dh_l dx^{\mathrm{Tam}}
\end{eqnarray*}
where the first equality is the definition of the action of $T_g$, the second follows from the first equality in \eqref{left-right}, the third follows from the fact that $\varphi_1$ is right translation invariant by $K_l$, the fourth follows from an obvious change of variable in the integral, the fifth is similar as the third and the sixth follows from the fact that $\varphi_2$ has trivial central character. Note that \eqref{left-right} implies that $K_l \nu(g) g^{-1} K_l=\bigsqcup_{j=1}^n \nu(g) \beta_j^{-1} K_l$. As a consequence 
\begin{eqnarray*}
\langle T_g \varphi_1, \varphi_2 \rangle &=& \langle \varphi_1, T_{\nu(g)g^{-1}}\varphi_2 \rangle
\end{eqnarray*}
as claimed.
\end{proof}

\section{Adjoint $L$-value and Petersson norm after Ichino} \label{after_ichino}

For the convenience of the reader, let us recall a particular case of the main result of \cite{ch19} expressing the value at $1$ of the adjoint $L$-function of some cuspidal automorphic representations of $G(\A)$ in terms of a Petersson norm.\\

Let $k \geq k' \geq 0$ be two integers. Let $\Pi$ be an irreducible cuspidal automorphic representation of $G(\A)$. Let us fix an isomorphism \begin{equation} \label{iso}
\Pi \simeq {\bigotimes_v}' \Pi_v.
\end{equation}
We make the following assumptions on $\Pi$:
\begin{enumerate}
\item[(i)] the central character of $\Pi$ is trivial,
\item[(ii)] $\Pi$ is globally generic,
\item[(iii)] the paramodular conductor $N \in \Z_{>0}$ of $\Pi$ (see \cite{roberts-schmidt}) is square-free,
\item[(iv)] $\Pi_\infty|_{G(\R)_+}=\Pi_\infty^{2,1} \oplus \Pi_\infty^{1,2}$,
\end{enumerate}
where $\Pi_\infty^{2,1}$ and $\Pi_\infty^{1,2}$ are the discrete series representations with Harish-Chandra parameter $(k+2, k'+1)$ and respective minimal $K_\infty$-types the irreducible representations $\tau_{k+3, -k'-1}$ and $\tau_{k'+1, -k-3}$ (see Proposition \ref{lpaquet}). According to \cite{asgari-shahidi}, the automorphic representation $\Pi$ has a functorial lift $\Sigma$ to $\mrm{GL}(4, \A)$ and we say that $\Pi$ is endoscopic if $\Sigma$ is not cuspidal. From now on, we assume that
\begin{enumerate}
\item[(v)]  $\Pi$ is endoscopic.
\end{enumerate}
Let $\eta=\otimes_v \eta_v$ be the standard additive character of $\Q \backslash \A$, so that $\eta_\infty(x)=\exp(2\pi i x)$ for $x \in \R$. By abuse of notation, we denote again by $\eta$ the character of $U(\Q) \backslash U(\A) \rightarrow \C^\times$ defined by
$$
\begin{pmatrix}
 1 & x_0 & &\\
   &  1  & &\\
   &   & 1& \\
   &   & -x_0 & 1 \\
\end{pmatrix}\begin{pmatrix}
 1 &  & x_1 & x_2\\
 &  1 & x_2 & x_3 \\
   &   &1 & \\
   &   &  &1 \\
\end{pmatrix} \mapsto \eta(-x_0-x_3)
$$
Let $\Pi$ be an irreducible cuspidal automorphic representation of $G(\A)$. The Whittaker function of a cusp form $\psi \in \Pi$ is
$$
W_\psi(g)=\int_{U(\Q) \backslash U(\A)} \psi(ug)\overline{\eta(u)} du
$$
for $g \in G(\A)$, where $du$ is the Tamagawa measure on $U(\A)$. Let us consider the unique element $\varphi=\bigotimes'_v \varphi_v \in \Pi$ normalized as in \cite[\S 1]{ch19} in the following way:
\begin{enumerate}
\item[(a)] for every prime $l$ not dividing $N$, the vector $\varphi_l$ is invariant by $G(\Z_l)$,
\item[(b)] for every prime $l$ dividing $N$, the vector $\varphi_l$ is invariant by the paramodular subgroup 
$$
K_l=\left\{ g \in G(\Q_l) \,|\, \nu(g) \in \Z_l^\times, g \in \begin{pmatrix}
 \Z_l & \Z_l & l^{-1} \Z_l & \Z_l\\
  l\Z_l &  \Z_l & \Z_l & \Z_l\\
  l\Z_l &  l\Z_l & \Z_l & l\Z_l\\
  l\Z_l &  \Z_l & \Z_l & \Z_l\\
\end{pmatrix} \right\} 
$$ of level $l$,
\item[(c)] the archimedean component $\varphi_\infty$ is a lowest weight vector of the minimal $\mrm{U}(2)$-type of  $\Pi_\infty^{1,2}$,
\item[(d)] 
we have $W_{\varphi_l}(1)=1$ for any finite prime $l$,
\item[(e)] we have
$$
W_{\varphi_\infty}(1)=e^{-2\pi}\int_{c_1-\sqrt{-1}\infty}^{c_1+\sqrt{-1}\infty}\frac{ds_1}{2\pi \sqrt{-1}}\int_{c_2-\sqrt{-1}\infty}^{c_2+\sqrt{-1}\infty}\frac{ds_2}{2\pi \sqrt{-1}}(4\pi^3)^{(-s_1+k+4)/2}(4\pi)^{(-s_2-k'-1)/2}
$$
$$
\times \Gamma \left( \frac{s_1+s_2-2k'-1}{2} \right) \Gamma \left( \frac{s_1+s_2+1}{2}\right)\Gamma \left( \frac{s_1}{2} \right)\Gamma \left( \frac{s_2}{2} \right)
$$
where $c_1, c_2 \in \R$ satisfy $c_1+c_2+1>0$ and $c_1>0>c_2$.
\end{enumerate}
Let $\langle \varphi , \varphi \rangle$ denote the Petersson norm 
$$
\langle \varphi , \varphi \rangle= \int_{Z(\A)G(\Q)\backslash G(\A)} |\varphi(x)|^2 dx^{\mathrm{Tam}}.
$$
The following statement is a particular case of Theorem 2.1 in \cite{ch19}.

\begin{thm} \label{ichino-main} We have
\begin{equation}\label{equation:chen_ichino}
\langle \varphi , \varphi \rangle = 2^{k+k'+13} \cdot 3^3 \cdot 5 \pi^{3k+k'+9}(k+k'+5)^{-1}\prod_{l |N } (l+l^{-1})^{-1} L(1, \Pi, \mathrm{Ad}).
\end{equation}
\end{thm}
\begin{proof}
The statement is a copy of the statement of Theorem 2.1 in \cite{ch19}. However, we need to explain the change of parameters 
from those used in \cite{ch19} to those used in our paper.   
The right hand of our equality \eqref{equation:chen_ichino} is 
$2^c \dfrac{L(1, \Pi, \mathrm{Ad})}{\Delta_{\mathrm{PGSp}_4}} \cdot \prod_v C(\pi_v )$ in \cite{ch19}. As explained in \cite{ch19}, we find that 
$c=2$ since our representation $\pi$ is endoscopic and we have 
$$
\frac{2^c}{\Delta_{\mathrm{PGSp}_4}}=2^2 \zeta (2)^{-1} \zeta(4)^{-1}=2^2\dfrac{6}{\pi^2}\dfrac{90}{\pi^4}=\dfrac{2^4 \cdot 3^3 \cdot 5}{\pi^6}. 
$$ 
The factor $C(\pi_\infty)$ at $v=\infty$ is given by 
$$
C(\pi_\infty) = 2^{\lambda_{1} - \lambda_{2} +5} \pi^{3\lambda_{1} - \lambda_{2} +5} 
(1+ \lambda_{1} - \lambda_{2})^{-1}
$$
where $(\lambda_1 , \lambda_2 )$ are Blattner parameter corresponding to $\pi_\infty$. We note that $\lambda_1$, $\lambda_2$ are 
presented as $\lambda_1 = k+3$ and $\lambda_2 =-k'-1$. Thus we have 
$$
C(\pi_\infty) = 2^{k +k' +9} \pi^{3k + k' +15} 
(k+k'+5)^{-1}. 
$$   
As for $C (\pi_l )$, the value is equal to $1$ when $l \nmid N$. When $l\vert N$, 
$$
C (\pi_l )=\frac{1}{l} \zeta_{l}(2)^{-1} \zeta_{l} (4) = \frac{1}{l} \left( 1-\frac{1}{l^2}\right) \frac{1}{1-\frac{1}{l^4}} 
 = (l+l^{-1})^{-1}. 
$$
This completes the proof.
\end{proof}
\section{Integral Betti cohomology of Siegel threefolds} 

\subsection{Integral local systems on Siegel threefolds} Siegel threefolds are the Shimura varieties associated to the group $G$. Let us briefly recall their definition. Let $\mathbb{S}=Res_{\mathbb{C}/\mathbb{R}} \mathbb{G}_{m, \mathbb{C}}$ be the Deligne torus and let $\mathcal{H}$ be the $G(\mathbb{R})$-conjugacy class of the morphism $h: \mathbb{S} \longrightarrow G_\mathbb{R}$ given on $\mathbb{R}$-points by
$$
x+iy \longmapsto \begin{pmatrix}
x & & y & \\
 & x &  & y\\
-y &  & x & \\
 &  -y &  & x\\
\end{pmatrix}.
$$
The pair $(G, \mathcal{H})$ verifies the axioms of Deligne-Shimura (see \cite[Lemme 2.1]{laumon}). Let $K_{N}$ denote the compact open subgroup of $G(\A_f)$ defined as 
$$
K_{N}=\prod_{l \in S(N)} K_l \times \prod_{l \notin S(N)} G(\Z_l).
$$
In this work, we are interested in the level $K_{N}$, but as $K_N$ is not neat we need to consider smaller congruence subgroups. Let $L_{\Z} \subset \Q^4$ be a $\Z$-lattice such that $K_{N}$ stabilizes $L_{\Z} \otimes \widehat{\Z}$ and let $K_{N}(3)=\{g \in K_{N}\,|\, (g-1) L_{\Z} \otimes \widehat{\Z} \subset 3  L_{\Z} \otimes \widehat{\Z}\}.$ Then $K_{N}(3)$ is a normal compact open subgroup of $K_{N}$ which is neat. Let $S_{K_N(3)}$ the Siegel threefold of level $K_{N}(3)$. This is a smooth quasiprojective $\Q$-scheme such that, as complex analytic varieties, we have
$
S_{K_{N}(3)} \simeq G(\mathbb{Q} )\backslash ( \mathcal{H} \times G(\mathbb{A}_f)/K_{N}(3)),
$
and carrying a universal abelian surface $a: A \rightarrow S_{K_{N}(3)}$ whith a polarization of degree $N^2$ and a principal level $3$ structure.\\

Let $\lambda(k, k', c)$ be a dominant weight. We associate to the representation $V_{\lambda, \Z}$ of $G$ a local system of $\Z$-modules for the classical topology on $S_{K_{N}(3)}$ as follows. To the standard representation $\mathrm{Std}_{\Z}$ we associate the first relative homology group $\mathcal{T}_{\Z}=\underline{\mrm{Hom}}(R^1a_*\Z, \Z)$ of $a$. The polarization on $A$ induces a pairing $\Psi:\mathcal{T}_{\Z} \otimes \mathcal{T}_{\Z} \rightarrow \Z$. If $d \geq 2$, let $\mathcal{T}_{\Z}^{\langle d \rangle} \subset \mathcal{T}_{\Z}^{\otimes d}$ be the intersection of the kernels of all the contractions $\mathcal{T}^{\otimes d}_{\Z} \rightarrow \mathcal{T}^{\otimes d-2}$ defined anlogously as the $\Psi_{p, q}$ (\ref{psipq}) using $\Psi$ instead of $\mathcal{J}$ and let $\tilde{V}_{\lambda, \Z}=\mathcal{T}_{\Z}^{\langle d \rangle}\cap \mbb{S}_\lambda(\mathcal{T}_{\Z}) \otimes (2\pi i)^t\Z$, where $\mbb{S}_\lambda(\mathcal{T}_{\Z})$ denotes the image of $c_\lambda \in \Z[\mathcal{S}_d]$ acting on $\mathcal{T}^{\otimes d}_{\Z}$ and where $t=\frac{c-k-k'}{2}$. Then $\tilde{V}_{\lambda, \Z}$ is a local system of $\Z$-modules on $S_N$. From now on the representation $V_{\lambda, \Z}$ and the local system associated to it $\tilde{V}_{\lambda, \Z}$ will be denoted by the same symbol $V_{\lambda, \Z}$.\\

We are interested in the Betti cohomology $H^3(S_{K_{N}(3)}, V_{\lambda, \Z})$, in the Betti cohomology with compact support $H^3_c(S_{K_{N}(3)}, V_{\lambda, \Z})$ and in the inner cohomology $H^3_!(S_{K_{N}(3)}, V_{\lambda, \Z})$ defined as
$$
H^3_!(S_{K_{N}(3)}, V_{\lambda, \Z})=\mrm{Im}(H^3_c(S_{K_{N}(3)}, V_{\lambda, \Z}) \rightarrow H^3(S_{K_{N}(3)}, V_{\lambda, \Z})).
$$
These are finitely generated $\Z$-modules. To prove that they are endowed with a natural action of $\mathcal{H}^{K_{N}(3)}$, we shall make use of the formalism of Grothendieck's functors $(f^*, f_* f_!, f^!)$ on the derived categories of sheaves of $\Z$-modules on locally compact topological spaces (see \cite{ayoub1}, \cite{ayoub2} and \cite{kashiwara-schapira}).\\

\subsection{Hecke action on integral cohomology} Let $g \in G(\A_f)$. Then we have the Hecke correspondence
   \begin{center}
    $
\xymatrix{
 & S_{K_{N}(3)\cap g K_{N}(3) g^{-1}} \ar@{->}[rd]^{c_2(g)} & \\ 
S_{K_{N}(3)} \ar@{<-}[ru]^{c_1(g)} &
& S_{K_{N}(3)}  \\
}
$
   \end{center}
where $c_1(g)$ is induced by the inclusion $K_{N}(3)\cap g K_{N}(3) g^{-1} \subset K_N(3)$ and $c_2(g)$ is induced by the morphism $K_{N}(3)\cap g K_{N}(3) g^{-1} \hookrightarrow K_N(3)$ defined by $k \mapsto gkg^{-1}$. This diagram depends only on the class of $g$ in $K_{N}(3) \backslash G(\A_f)/ K_{N}(3)$. 

\begin{lem} \label{hecke-sheaf} There exists a canonical morphism
$$
\varphi_*: c_1(g)^* V_{\lambda, \Z} \rightarrow c_2(g)^* V_{\lambda, \Z}.
$$
\end{lem}

\begin{proof} It follows from the modular description of $c_1(g)$ and $c_2(g)$ which, in our case, is similar to the one given at p. 9 of \cite{laumon} that we have an isogeny $\varphi: c_1(g)^*A \rightarrow c_2(g)^*A$ over $S_{K_{N}(3)\cap g K_{N}(3) g^{-1}}$. Taking homology, we obtain a morphism $\varphi_*: c_1(g)^*\mathcal{H}_{\Z} \rightarrow c_2(g)^* \mathcal{H}_{\Z}$ which induces the morphism of the statement.
\end{proof}

Note that as $c_i(g)$ is finite and \'etale for $i=1, 2$, we have canonical isomorphisms of functors $c_i(g)_! \simeq c_i(g)_*$ and $c_i(g)^!\simeq c_i(g)^*$. Let us denote by $\bullet$ the topological space reduced to a point and let $p_{K_{N}(3)}: S_{K_{N}(3)} \rightarrow \bullet$ the canonical continuous map. Then $H^*(S_{K_{N}(3)}, V_{\lambda, \Z})$ is the cohomology of the complex $p_{K_{N}(3) *}V_{\lambda, \Z}$ and $H^*_c(S_{K_{N}(3)}, V_{\lambda, \Z})$ is the cohomology of the complex $p_{K_{N}(3) !}V_{\lambda, \Z}$.

\begin{defn} \label{def-hecke}
The endomorphism
$$
T_g: H^*(S_{K_{N}(3)}, V_{\lambda, \Z}) \rightarrow H^*(S_{K_{N}(3)}, V_{\lambda, \Z})
$$
is defined as the endomorphism induced by the map
\begin{multline*}
p_{K_{N}(3)\, *} V_{\lambda, \Z} \longrightarrow
p_{K_{N}(3)\, *} c_1(g)_* c_1(g)^*V_{\lambda, \Z} \longrightarrow
p_{K_{N}(3)\,*} c_2(g)_*c_1(g)^*V_{\lambda, \Z}\\ \overset{\varphi_*}{\longrightarrow} 
p_{K_{N}(3)\, *} c_2(g)_* c_2(g)^*V_{\lambda, \Z} \longrightarrow
p_{K_{N}(3)\, *} c_2(g)_! c_2(g)^!V_{\lambda, \Z} \longrightarrow
p_{K_{N}(3)\, *}V_{\lambda, \Z}
\end{multline*}
in the derived category of $\Z$-modules, where the first arrow is induced by the adjunction morphism, the second arrow is induced by the isomorphism of functors
$$
p_{K_{N}(3)\,*} c_1(g)_* \simeq (p_{K_{N}(3)} \circ c_1(g))_* \simeq (p_{K_{N}(3)} \circ c_2(g))_* \simeq p_{K_{N}(3)\, *} c_2(g)_*,
$$
the third arrow is induced by the canonical morphism $\varphi_*$ of Lemma \ref{hecke-sheaf}, the fourth  arrow is induced by the isomorphism of functors $c_2(g)_* c_2(g)^* \simeq c_2(g)_! c_2(g)^!$ and the last arrow is induced by trace map. The endomorphism  
$$
T_{g, c}: H^*_c(S_{K_{N}(3)}, V_{\lambda, \Z}) \rightarrow H^*_c(S_{K_{N}(3)}, V_{\lambda, \Z})
$$
is defined similarly as the endomorphism induced by 
\begin{multline*}
p_{K_{N}(3)\, !} V_{\lambda, \Z} \longrightarrow
p_{K_{N}(3)\, !} c_1(g)_* c_1(g)^*V_{\lambda, \Z} \longrightarrow
p_{K_{N}(3)\,!} c_2(g)_! c_1(g)^*V_{\lambda, \Z}\\ \overset{\varphi_*}{\longrightarrow} 
p_{K_{N}(3)\, !} c_2(g)_! c_2(g)^*V_{\lambda, \Z} \longrightarrow
p_{K_{N}(3)\, !} c_2(g)_! c_2(g)^!V_{\lambda, \Z} \longrightarrow
p_{K_{N}(3)\, !}V_{\lambda, \Z}.
\end{multline*}
\end{defn}

\begin{pro} \label{hecke-action} For any $g \in G(\A_f)$ the diagram
$$
\begin{CD}
H^*_c(S_{K_{N}(3)}, V_{\lambda, \Z}) @>>> H^*(S_{K_{N}(3)}, V_{\lambda, \Z})\\
@VVT_{g, c} V                                                @VVT_gV\\
H^*_c(S_{K_{N}(3)}, V_{\lambda, \Z}) @>>> H^*(S_{K_{N}(3)}, V_{\lambda, \Z})
\end{CD}
$$
where the maps are defined above is commutative and depends only on the double class $K_{N}(3)g{K_{N}(3)}$.
\end{pro}

\begin{proof} The statement follows from the commutativity of the following diagram in the derived category of $\Z$-modules
$$
\begin{CD}
p_{K_{N}(3)\, !} V_{\lambda, \Z} @>>> p_{K_{N}(3)\, *} V_{\lambda, \Z}\\
@VVV                  @VVV\\
p_{K_{N}(3)\, !} c_1(g)_* c_1(g)^*V_{\lambda, \Z} @>>> p_{K_{N}(3)\, *} c_1(g)_* c_1(g)^*V_{\lambda, \Z}\\
@VV\sim V                                                  @VV\sim V\\
 p_{K_{N}(3)\, !} c_2(g)_! c_1(g)^*V_{\lambda, \Z}       @>>> p_{K_{N}(3)\,*} c_2(g)_*c_1(g)^*V_{\lambda, \Z}\\
@VV\varphi_* V                                                         @VV\varphi_* V\\
p_{K_{N}(3)\, !} c_2(g)_! c_2(g)^*V_{\lambda, \Z} @>>> p_{K_{N}(3)\, *} c_2(g)_* c_2(g)^*V_{\lambda, \Z}\\  
@VVV                                                 @VVV\\
p_{K_{N}(3)\, !} c_2(g)_! c_2(g)^! V_{\lambda, \Z} @>>> p_{K_{N}(3)\, *} c_2(g)_! c_2(g)^!V_{\lambda, \Z}\\
@VVV                                                  @VVV\\
p_{K_{N}(3)\, !}V_{\lambda, \Z}  @>>> p_{K_{N}(3)\, *}V_{\lambda, \Z},\\
\end{CD}
$$
where the vertical lines are defined in Definition \ref{def-hecke} and where the first, the second, the fifth and the sixth horizontal arrows are induced by the morphism of functors $p_{K_{N}(3)\, !} \rightarrow p_{K_{N}(3)\, *}$ and the third and the fourth horizontal arrows are induced by the morphism of functors 
$$
\begin{CD}
p_{K_{N}(3)\,!}c_2(g)_! @>>> p_{K_{N}(3)\,!}c_2(g)_* @>>> p_{K_{N}(3)\,*}c_2(g)_*.
\end{CD}
$$
The commutativity of all but the second square follows from functoriality. Let us prove the commutativity of the second square. According to \cite[Proposition 1.7.3]{ayoub1}, the following diagram of functors
$$
\begin{CD}
p_{K_{N}(3)\,!} c_1(g)_! @>>> p_{K_{N}(3)\,*} c_1(g)_*\\
@VV\sim V                            @VV\sim V\\
(p_{K_{N}(3)} \circ c_1(g))_! @>>> (p_{K_{N}(3)} \circ c_1(g))_*
\end{CD}
$$
where the upper horizontal map is the composite 
$$
\begin{CD}
p_{K_{N}(3)\,!} c_1(g)_! @>>> p_{K_{N}(3)\,!}c_1(g)_* @>>> p_{K_{N}(3)\,*} c_1(g)_*
\end{CD}
$$
is commutative. As the first map above is an equality, the diagram
$$
\begin{CD}
p_{K_{N}(3)\,!} c_1(g)_* @>>> p_{K_{N}(3)\,*} c_1(g)_*\\
@VV\sim V                            @VV\sim V\\
(p_{K_{N}(3)} \circ c_1(g))_! @>>> (p_{K_{N}(3)} \circ c_1(g))_*
\end{CD}
$$
is commutative. Furthermore, the diagram
$$
\begin{CD}
(p_{K_{N}(3)} \circ c_1(g))_! @>>> (p_{K_{N}(3)} \circ c_1(g))_*\\
@|                          @|\\
(p_{K_{N}(3)} \circ c_2(g))_! @>>> (p_{K_{N}(3)} \circ c_2(g))_*\\
@VV\sim V                          @VV\sim V\\
p_{K_{N}(3)\,!} c_2(g)_! @>>> p_{K_{N}(3)\,*} c_2(g)_*
\end{CD}
$$
where the last horizontal map is defined as above, is commutative. By combining the commutativity of the two previous diagrams, we complete the proof.
\end{proof}

As each function $f \in \mathcal{H}^{K_{N}(3)}$ is a finite linear combination with $\Z$-coefficients of characteristic functions of double cosets ${K_{N}(3)}g{K_{N}(3)}$ with $g \in G(\A_f)$, by the previous result we define a ring homomorphism 
\begin{equation} \label{heckeaction}
\mathcal{H}^{K_{N}(3)} \rightarrow \mrm{End}_{\Z}(H^*_!(S_{K_{N}(3)}, V_{\lambda, \Z}))
\end{equation}
by sending the characteristic function $1_{{K_{N}(3)}g{K_{N}(3)}}$ to the endomorphism of $H^*_!(S_{K_{N}(3)}, V_{\lambda, \Z})$ induced by $T_g$. 

\subsection{Isotypical component of cohomology} For any $\Z$-algebra $R$ let us denote by $V_{\lambda, R}$ the representation of $G_R$ deduced from $V_\lambda$ by extending the scalars to $R$. By abuse of notation, we will denote by the same symbol the local system of $R$-modules on $S_{K_N(3)}$ naturally associated to $V_{\lambda, R}$. By the universal coefficients theorem, we have $H^*_!(S_{K_N(3)}, V_{\lambda, \C})=H^*_!(S_{K_N(3)}, V_{\lambda, \Z}) \otimes \C$. Let 
$$
P(V_{\lambda, \C})=\{\Pi_\infty^{3,0}, \Pi_\infty^{2,1}, {\Pi}_\infty^{1,2}, {\Pi}_\infty^{0,3} \}
$$ 
be the discrete series $L$-packet attached to the complex representation $V_{\lambda, \C}$. If $\lambda=\lambda(k, k', c)$ with $k \geq k' \geq 0$ and $c \equiv k+k' \pmod 2$ then the set $
P(V_{\lambda, \C})$ is the set of discrete series of $G(\R)_+$ with Harish-Chandra parameter $(k+2, k'+1)$ and central character $x \mapsto x^{-c}$ given by Proposition \ref{lpaquet}. Let us denote by $m(\Pi)$ the cuspidal multiplicity of an automorphic representation $\Pi$ of $G(\A)$.

\begin{pro}  \label{dec} There is a canonical isomorphism of $\mathcal{H}^{K_{N}(3)}_{\C}$-modules
\begin{equation*}
H^3_{!}(S_{K_{N}(3)}, V_{\lambda, \C})=\bigoplus_{\Pi=\Pi_\infty \otimes \Pi_f} H^3(\mathfrak{g}_{\C},K_\infty, V_{\lambda, \C} \otimes \Pi_\infty)^{\oplus m(\Pi)} \otimes_{\C} \Pi_f^{K_{N}(3)}
\end{equation*}
where the sum is indexed by isomorphism classes of cuspidal automorphic representations $\Pi \simeq \Pi_\infty \otimes \Pi_f$ of $G(\A)$ such that $\Pi_\infty|_{G(\R)_+} \in P(V_{\lambda, \C})$.
\end{pro}

\begin{proof} This is well known and follows for example from \cite{lemma2} (8) and (9).
\end{proof}

In order to avoid difficulties arising in the presence of torsion in the cohomology, we are going to work with $\Z_{(p)}$ coefficients instead of $\Z$ coefficients for sufficiently big prime numbers $p$. Let 
\begin{equation} \label{s_G}
S_{N, 3}=\{p \text{ prime}, p| \# K_N/K_N(3) \}.
\end{equation} 
Let $H^3_!(S_{K_N(3)}, V_{\lambda, \Z})_{\mathrm{tors}}$ denote the torsion subgroup of $H^3_!(S_{K_N(3)}, V_{\lambda, \Z})$. Let
\begin{equation} \label{s-tors}
S_{\mathrm{tors}}=\{p \text{ prime}, p| \# H^3_!(S_{K_N(3)}, V_{\lambda, \Z})_{\mathrm{tors}}\} 
\end{equation}
and fix a prime number $p$ such that $p \notin S_{N, 3} \cup S_{\mathrm{tors}}$.\\

The finite group $K_{N}/K_{N}(3)$ acts on $S_{K_{N}(3)}$ by automorphisms and hence acts on the cohomology $H^3_!(S_{K_{N}(3)}, V_{\lambda, \Z_{(p)}})$. The action of $\mathcal{H}^{K_N(3)}$ on $H^3_!(S_{K_{N}(3)}, V_{\lambda, \Z_{(p)}})$ induces an action of $\mathcal{H}^{K_N}$ on the submodule $\tilde{H}^3_!(S_{K_N}, V_{\lambda, \Z_{(p)}})$ defined as 
$$
\tilde{H}^3_!(S_{K_N}, V_{\lambda, \Z_{(p)}})=H^3_!(S_{K_{N}(3)}, V_{\lambda, \Z_{(p)}})^{K_N/K_N(3)}
$$ 
via the natural map $\mathcal{H}^{K_N} \rightarrow (\mathcal{H}^{K_
N(3)})^{K_N/{K_{N}(3)}}$. Note that the fact that $p \notin S_{N, 3}$ implies that $\tilde{H}^3_!(S_{K_N}, V_{\lambda, \Z_{(p)}})$ is a direct factor of $H^3_!(S_{K_{N}(3)}, V_{\lambda, \Z_{(p)}})$. Let $\tilde{H}^3_!(S_{K_N}, V_{\lambda, \C})$ denote $\tilde{H}^3_!(S_{K_N}, V_{\lambda, \Z_{(p)}}) \otimes_{\Z_{(p)}} \C$. It follows from the decomposition of Proposition \ref{dec} that we have a  $\mathcal{H}^{K_N}$-equivariant decomposition
\begin{equation} \label{dec-level1}
\tilde{H}^3_!(S_N, V_{\lambda, \C}) \simeq \bigoplus_{\Pi_i =\Pi_{i,\infty} \otimes \Pi_{i,f}} H^3(\mathfrak{g}_{\C},K_\infty, V_{\lambda, \C} \otimes \Pi_{i,\infty})^{\oplus m(\Pi_i )} \otimes_{\C} \Pi_{i,f}^{K_N}
\end{equation}
indexed as above. We will denote by $\overline{\mathcal{H}}^{K_N}_{\Z_{(p)}}$ the subring of $\mrm{End}_{\Z_{(p)}}(\tilde{H}^3_!(S_N, V_{\lambda, \Z_{(p)}}))$ defined as the image of the homomorphism defined above $\mathcal{H}^{K_N}_{\Z_{(p)}} \rightarrow \mrm{End}_{\Z_{(p)}}(\tilde{H}^3_!(S_N, V_{\lambda, \Z_{(p)}}))$. For any $\Z_{(p)}$-algebra $R$, we denote by $\overline{\mathcal{H}}^{K_N}_R$ the base change 
$\overline{\mathcal{H}}^{K_N}_{\Z_{(p)}} \otimes_{\Z_{(p)}} R$. Note that as $\tilde{H}^3_!(S_N, V_{\lambda, \Z_{(p)}})$ is torsion free, the $\Z_{(p)}$-algebra $\overline{\mathcal{H}}^{K_N}_{\Z_{(p)}}$ is torsion free and hence $\overline{\mathcal{H}}^{K_N}_R$ is canonically isomorphic to the image of $\mathcal{H}^{K_N}_R$ in $\mrm{End}_{R}(\tilde{H}^3_!(S_N, V_{\lambda, \Z_{(p)}})\otimes_{\Z_{(p)}} R)$.\\

Let $\Pi=\Pi_\infty \otimes \Pi_f$ be the cuspidal automorphic representation of $G(\A)$ appearing in the statement of Theorem \ref{ichino-main}. As $\Pi_\infty|_{G(\R)_+} \in P(V_{\lambda, \C})$, the $\mathcal{H}^{K_N}_{\C}$-module $\Pi_f^{K_N}$ appears in the decomposition (\ref{dec-level1}). As $N$ denotes the paramodular conductor of $\Pi$, the space of fixed points $\Pi_f^{K_N}$ is one-dimensional. Let $E(\Pi_f)$ be the rationality field of $\Pi_f$. By definition $E(\Pi_f)$ is a subfield of $\C$ and it follows from \cite[Theorem 3.2.2]{bhr} that $E(\Pi_f)$ is a number field. Let $\theta_\Pi: \overline{\mathcal{H}}_{\Z_{(p)}}^{K_N} \subset \overline{\mathcal{H}}_{\C}^{K_N} \rightarrow \C$ the character giving the action of $\overline{\mathcal{H}}_{\Z_{(p)}}^{K_N}$ on $\Pi_f^{K_N}$. As $\Pi_f^{K_N}$ is one-dimensional, it follows from \cite[Lemme I.1]{waldspurger} that $\Pi_f$ is defined over $E(\Pi_f)$ and so, the character $\theta_\Pi$ has values in $E(\Pi_f)$.

\begin{lem} \label{order} 
Let $\Pi=\Pi_\infty \otimes \Pi_f$ be a cuspidal automorphic representation of $G(\A)$ of conductor $N$. 
Then for the character $\theta_\Pi$ of $\overline{\mathcal{H}}^{K_N}_{\Z_{(p)}}$, 
$\im \theta_\Pi \otimes_{\Z_{(p)}} \Q$ is isomorphic to $E(\Pi_f)$.
\end{lem}

\begin{proof} We consider the decomposition of the $\mathbb{C}$-algebra $\overline{\mathcal{H}}^{K_N}_{\C}=\prod \overline{\mathcal{H}}^{K_N}_{i,\C}$
corresponding to the decomposition \eqref{dec-level1} of the $\overline{\mathcal{H}}^{K_N}_{\C}$-module $\tilde{H}^3_!(S_N, V_{\lambda, \C})$. 
Since $\Pi$ has conductor $N$, the component $\overline{\mathcal{H}}^{K_N}_{i_0,\C}$ corresponding to $\Pi$ is isomorphic to $\mathbb{C}$. Thus 
$\im \theta_\Pi \otimes_{\Z_{(p)}} \Q$ must be a number field because it is an artinian  subalgebra of $\mathcal{H}^{K_N}_{\C,i_0}\cong \mathbb{C}$. 
By the definition of the coefficient field the number field $E(\Pi_f)$, $\im \theta_\Pi \otimes_{\Z_{(p)}} \Q$ is isomorphic to $E(\Pi_f)$.  
\end{proof}

Let us denote by ${M}(\Pi_f, V_{\lambda, \Q})$ the finite dimensional $\Q$-vector space
\begin{eqnarray*} 
{M}(\Pi_f, V_{\lambda, \Q})=\left(\tilde{H}^3_!(S_{K_N}, V_{\lambda, \Z_{(p)}}) \otimes_{\overline{\mathcal{H}}_{\Z_{(p)}}^{K_N}} \overline{\mathcal{H}}_{\Z_{(p)}}^{K_N}/\ker \theta_\Pi \right) \otimes_{\Z_{(p)}} \Q.
\end{eqnarray*}
This is a direct factor of the $\Q$-vector space $\tilde{H}^3_!(S_{K_N}, V_{\lambda, \Q})$. Let 
\begin{equation} \label{projection}
p_M: \tilde{H}^3_!(S_{K_N}, V_{\lambda, \Q}) \rightarrow {M}(\Pi_f, V_{\lambda, \Q})
\end{equation}
denote the projection and let ${M}(\Pi_f, V_{\lambda, \Z_{(p)}})$ be defined as
\begin{equation} \label{O-cohomology}
M(\Pi_f, V_{\lambda, \Z_{(p)}})=p_M\left(\tilde{H}^3_!(S_{K_N}, V_{\lambda, \Z_{(p)}})\right).
\end{equation}
Let $L(\Pi_f, V_{\lambda, \Z_{(p)}})$ denote the $\Z_{(p)}$-lattice of $M(\Pi_f, V_{\lambda, \Q})$ defined as
\begin{equation} \label{lattice}
L(\Pi_f, V_{\lambda, \Z_{(p)}})= {M}(\Pi_f, V_{\lambda, \Q}) \cap \tilde{H}^3_!(S_{K_N}, V_{\lambda, \Z_{(p)}}).
\end{equation}
It is clear that $L(\Pi_f, V_{\lambda, \Z_{(p)}})$ is a sub $\Z_{(p)}$-module of $M(\Pi_f, V_{\lambda, \Z_{(p)}})$. For any $\Z_{(p)}$-algebra $R$, we will denote by $L(\Pi_f, V_{\lambda, R})$ the $R$-module ${L}(\Pi_f, V_{\lambda, \Z_{(p)}}) \otimes_{\Z_{(p)}} R$. According to Corollary \ref{order}, we have
$$
L(\Pi_f, V_{\lambda, \C})=M(\Pi_f, V_{\lambda, \C}) = \bigoplus_{\sigma: E(\Pi_f) \rightarrow \C} \tilde{H}^3_!(S_{K_N}, V_{\lambda, \Z_{(p)}}) \otimes_{\theta^\sigma_\Pi} \C
$$
where for any embedding $\sigma: E(\Pi_f) \rightarrow \C$ the character $\theta^\sigma_\Pi: \overline{\mathcal{H}}^{K_N}_{\Z_{(p)}} \rightarrow \C$ is the composite $\theta^\sigma_\Pi: \overline{\mathcal{H}}^{K_N}_{\Z_{(p)}} \rightarrow E(\Pi_f) \overset{\sigma}{\rightarrow} \C$. 

\begin{thm} \label{sigma-endo} Let $\Pi=\Pi_\infty \otimes \Pi_f$ be an irreducible cuspidal automorphic representation of $G(\A)$ with trivial central character and such that $\Pi_\infty|_{G(\R)_+} \in P(V_{\lambda, \C})$. Then the following statements hold for any $\sigma \in \mrm{Aut}(\C)$.
\begin{enumerate}
\item[\rm{(1)}] There exists an irreducible cuspidal automorphic representation $^\sigma\!\,\Pi \simeq \bigotimes'_v \Pi_{\sigma, v}$ of $G(\A)$ with trivial central character such that $\Pi_{\sigma, \infty}|_{G(\R)_+} \in P(V_{\lambda, \C})$ and whose non-archimedean component is equivalent to $^\sigma\!\,\Pi_f$.
\item[\rm{(2)}] If one assumes that $\Pi$ is globally generic and endoscopic then so is $^\sigma\!\,\Pi$.
\end{enumerate}
\end{thm}

\begin{proof} Statement (1) follows from the proof of \cite[Proposition 2.4]{gan-raghuram}  combined with the fact that for Siegel threefolds, $L^2$-cohomology coincides with cuspidal cohomology (see \cite[Proposition 1]{mokrane-tilouine} ). As $\mrm{PGSp}(4) \simeq \mrm{SO}(5)$, statement (2) is a particular case of \cite[Theorem 10.1 and Theorem 9.5]{gan-raghuram} .
\end{proof}

The decomposition (\ref{dec-level1}) and the remark at the end of section \ref{heckealgsection} imply that we have a canonical isomorphism
\begin{equation} \label{dec_again}
L(\Pi_f, V_{\lambda, \C}) \simeq \bigoplus_{\sigma: E(\Pi_f) \rightarrow \C}\bigoplus_{\Pi_\infty|_{G(\R)_+} \in P(V_{\lambda, \C})} H^3(\mathfrak{g}_{\C},K_\infty, V_{\lambda, \C} \otimes \Pi_\infty)^{\oplus m(\Pi_\infty \otimes ^\sigma\!\,\Pi_f)} \otimes_{\C}\, ^\sigma\!\,\Pi_f^{K_N}.
\end{equation}
Furthermore, for any $r, s$, the $\C$-vector spaces $H^3(\mathfrak{g}_{\C},K_\infty, V_{\lambda, \C} \otimes \Pi_\infty^{r,s})$ are one-dimensional (see for example \cite[Proposition 3.7]{lemma2}). As a consequence, we have 
\begin{equation} \label{multiplicity-dimension}
\dim_{\C} \left( L(\Pi_f, V_{\lambda, {\C}}) \otimes_{{\Z_{(p)}}} \C \right)=2\sum_{\sigma: E(\Pi_f) \rightarrow \C} m(\Pi_\infty^H \otimes\,\!   ^\sigma\!\,\Pi_f)+m(\Pi_\infty^W \otimes\,\! ^\sigma\!\,\Pi_f).
\end{equation}
Recall that we say that a globally generic cuspidal automorphic representation $\Pi$ of $G(\A)$ is endoscopic if its functorial lift to $\mrm{GL}(4, \A)$ is not cuspidal. 

\begin{pro} \label{multiplicity} Let $\Pi=\Pi_\infty \otimes \Pi_f$ be a globally generic irreducible unitary endoscopic cuspidal automorphic representation of $G(\A)$. Then 
\begin{enumerate}
\item[\rm{(1)}] we have $m(\Pi)=1$,
\item[\rm{(2)}] we have $m(\Pi_\infty^H \otimes \Pi_f)=0$.
\end{enumerate}
\end{pro}

\begin{proof} The first statement follows from the main result of \cite{jiang-soudry}. According to \cite[Proposition 2.2 (a)]{asgari-shahidi}, $\Pi$ is obtained as a Weil lifting from $\mrm{GSO}(2,2, \A)$. As $\Pi$ is assumed to be globally generic, for every prime $l$ the representation $\Pi_l$ has a Whittaker model. Hence, applying \cite[Theorem 5.2 (4)]{weissauer2} to $\Pi$ we obtain
$$
m(\Pi_\infty^H \otimes \Pi_f)=m(\overline{\Pi}_\infty^H \otimes \Pi_f)=\frac{1}{2}(1+(-1)^{\epsilon_\infty})
$$ 
where $\epsilon_\infty=0$ or $1$ if $\Pi_\infty^H$ has or has not a Whittaker model. But it is well known that $\Pi_\infty^H$ is the archimedean component of an automorphic representation associated to a cuspidal Siegel modular form, which does not have a Whittaker model. As a consequence we obtain $m(\Pi_\infty^H \otimes \Pi_f)=m(\overline{\Pi}_\infty^H \otimes \Pi_f)=0.$
\end{proof}

\begin{cor} Let the notation and assumptions be as in the previous result. Then
\begin{eqnarray*}
\mrm{rk}_{{\Z_{(p)}}}L(\Pi_f, V_{\lambda, {\Z_{(p)}}})=2[E(\Pi_f):\Q].\\
\end{eqnarray*}
\end{cor}

\begin{proof} This follows from the fact that $\mrm{rk}_{{\Z_{(p)}}}L(\Pi_f, V_{\lambda, {\Z_{(p)}}})=\dim_{\C} L(\Pi_f, V_{\lambda, {\C}})$ and from \eqref{multiplicity-dimension} combined with Theorem \ref{sigma-endo} and Proposition \ref{multiplicity}.
\end{proof}

\subsection{Poincar\'e duality} Let $\check{V}_{\lambda, \Z_{(p)}}$ denote the dual local system $\check{V}_{\lambda, \Z_{(p)}}=\underline{\Hom}(V_{\lambda, \Z_{(p)}}, \Z_{(p)})$. Let us define
\begin{equation} \label{s'-tors}
S'_{\mathrm{tors}}=\{p \text{ prime }, p|\#H^4_c(S_{K_N(3)}, V_{\lambda, \Z})_{\mathrm{tors}}\}.
\end{equation}

\begin{lem} \label{Pduality} Assume that $p \notin S'_{\mathrm{tors}}$. Then we have a canonical isomorphism
$$
H^3(S_{K_N(3)}, \check{V}_{\lambda, \Z_{(p)}}) \simeq \Hom_{\Z_{(p)}}(H^3_c(S_{K_N(3)}, {V}_{\lambda, \Z_{(p)}}), \Z_{(p)}).
$$
\end{lem}

\begin{proof} Recall that $p_3: S_{K_N(3)} \rightarrow \bullet$ denotes the canonical continuous map from $S_{K_N(3)}$ to the topological space reduced to a point. According to \cite[Proposition 3.1.10]{kashiwara-schapira} we have a canonical isomorphism
$$
R\underline{\Hom}(Rp_{3\,!}V_{\lambda, \Z_{(p)}}, \Z_{(p)}) \simeq Rp_{3\,*}R\underline{\Hom}(V_{\lambda, \Z_{(p)}},p_3^! \Z_{(p)})
$$
in the derived category of abelian groups. Hence, we have two spectral sequences
\begin{eqnarray*}
E_2^{p,q}=\mrm{Ext}^p_{\Z_{(p)}}(H^{-q}_c(S_{K_N(3)}, {V}_{\lambda, \Z_{(p)}}), \Z_{(p)}) &\implies& E_\infty^{p+q},\\
{E'_2}^{p,q}=H^p(S_{K_N(3)}, \underline{\mrm{Ext}}^q(V_{\lambda, \Z_{(p)}}, p_3^!\Z_{(p)})) &\implies& {E'_\infty}\!\!^{p+q}
\end{eqnarray*}
where $E_\infty^{p+q} \simeq {E'_\infty}\!\!^{p+q}$. According to \cite[Proposition 3.3.2 (i)]{kashiwara-schapira} , as $S_{K_N(3)}$ is smooth of real dimension $6$, we have $p_3^! \Z_{(p)}=\Z_{(p)}[6]$. As a consequence, using the fact that $V_{\lambda, \Z_{(p)}}$ is a sheaf of free $\Z_{(p)}$-modules of finite type, we have $\underline{\mrm{Ext}}^q(V_{\lambda, \Z_{(p)}}, p_3^!\Z_{(p)})=0$ for $q \neq -6$ and $\underline{\mrm{Ext}}^{-6}(V_{\lambda, \Z_{(p)}}, p_3^!\Z_{(p)})=\check{V}_{\lambda, \Z_{(p)}}$. This implies immediately ${E'_2}^{3,-6} \simeq {E'_\infty}\!\!^{-3}$.\\
\indent On the other hand, as $\Z_{(p)}$ is a discrete valuation ring, we have $E_2^{p, q}=0$ for any $p \geq 2$ and any $q \in \Z$. Hence $E_2^{0,-3}=E_\infty^{0,-3}$ and $E_\infty^{2, -5}=E_\infty^{3,-6}=0$.  The assumption that $p$ is outside $S'_{\mathrm{tors}}$ implies that $E_2^{1,-4}=E_\infty^{1,-4}=0$. Hence $E_2^{0,-3}=E_\infty^{-3}$. As $E_\infty^{-3} \simeq {E'_\infty}\!\!^{-3}$ we have a canonical isomorphism $E_2^{0,-3} \simeq {E'_2}^{3,-6}$ which proves the claim.
\end{proof}

By the previous Lemma, assuming that $p \notin S'_{\mathrm{tors}}$, we have a non-degenerate pairing
$$
\langle \,,\,\rangle: H^3(S_{K_N(3)}, \check{V}_{\lambda, \Z_{(p)}}) \times H^3_c(S_{K_N(3)}, {V}_{\lambda, \Z_{(p)}}) \rightarrow \Z_{(p)}.
$$
By composing with the natural inclusion $H^3_!(S_{K_N(3)}, \check{V}_{\lambda, \Z_{(p)}}) \hookrightarrow H^3(S_{K_N(3)}, \check{V}_{\lambda, \Z_{(p)}})$ on the first factor, we obtain a pairing
$$
\langle \,,\, \rangle: H^3_!(S_{K_N(3)}, \check{V}_{\lambda, \Z_{(p)}}) \times H^3_c(S_{K_N(3)}, {V}_{\lambda, \Z_{(p)}}) \rightarrow \Z_{(p)}.
$$

\begin{lem} \label{duality!}For any $x \in H^3_!(S_{K_N(3)}, \check{V}_{\lambda, \Z_{(p)}})$ and for any $y \in \ker(H^3_c(S_{K_N(3)}, {V}_{\lambda, \Z_{(p)}}) \rightarrow H^3(S_{K_N(3)}, {V}_{\lambda, \Z_{(p)}}))$ we have $\langle x, y \rangle=0$.
\end{lem}

\begin{proof} Let us denote by $\langle \,,\, \rangle_{\C}: H^3_!(S_{K_N(3)}, \check{V}_{\lambda, \C}) \times H^3_c(S_{K_N(3)}, {V}_{\lambda, \C}) \rightarrow \C$ be the pairing obtained from $\langle \,,\,\rangle$ by base change to $\C$. Let $x_{\C}$ denote the image of $x$ in $H^3_!(S_{K_N(3)}, \check{V}_{\lambda, \C})$ and let $y_{\C}$ denote the image of $y$ in $H^3_c(S_{K_N(3)}, {V}_{\lambda, \C})$. It is enough to prove that $\langle x_{\C}, y_{\C} \rangle_{\C}=0$. By the comparison isomorphism between Betti and de Rham cohomology and by compatibility of the Poincar\'e duality in Betti and de Rham cohomology, it is enough to prove the following statement: for any $x'_{\C} \in H^3_{dR, !}(S_{K_N(3)},  \check{\mathcal{V}}_\lambda)$ and any $y'_{\C} \ \in \ker(H^3_{dR, c}(S_{K_N(3)}, {\mathcal{V}}_{\lambda}) \rightarrow H^3_{dR}(S_{K_N(3)}, {\mathcal{V}}_{\lambda}))$, we have $\langle x'_{\C}, y'_{\C} \rangle_{dR, \C}=0$ where $\mathcal{V}_{\lambda}$ is the complex vector bundle associated with the local system $V_{\lambda, \C}$, $\check{\mathcal{V}}_\lambda$ is the dual complex vector bundle and $\langle \,,\,\rangle_{dR, \C}$ denotes the Poincar\'e duality pairing in de Rham cohomology. Let us abusively denote by $x'_{\C}$ and $y'_{\C}$ closed differential forms in the cohomology class of $x'_{\C}$ and $y'_{\C}$ respectively. Because $x \in \im(H^3_c(S_{K_N(3)}, \check{V}_{\lambda, \Z_{(p)}}) \rightarrow H^3(S_{K_N(3)}, \check{V}_{\lambda, \Z_{(p)}}))$, there exists a compactly supported closed differential $3$-form $x''$ and a differential $2$-form $u$ such that $x'_{\C}=x''+du$ and because $y \in \ker(H^3_c(S_{K_N(3)}, {V}_{\lambda, \Z_{(p)}}) \rightarrow H^3(S_{K_N(3)}, {V}_{\lambda, \Z_{(p)}}))$ there exists a differential $2$-form $y''$ such that the compactly supported differential form $y'_{\C}$ satisfies $y'_{\C}=dy''$. We need to prove
$\int_{S_{K_N(3)}} x'_{\C} \wedge y'_{\C}=0.$ But
$$
x'_{\C} \wedge y'_{\C}=x'' \wedge dy'' + du \wedge y'_{\C}=d(-x'' \wedge y''+ u \wedge y'_{\C}).
$$
As the differential form $-x'' \wedge y''+ u \wedge y'_{\C}$ is compactly supported, the statement follows from Stokes theorem.
\end{proof}

\begin{lem} Let $\lambda=\lambda(k, k', 0)$ be a dominant weight with trivial central character. Assume that $k+k'+3 \leq p$. Then, there exists a non-degenerate $G_{\Z_{(p)}}$-equivariant pairing $$[\,,\,]: V_{\lambda, \Z_{(p)}} \times V_{\lambda, \Z_{(p)}} \rightarrow \Z_{(p)}$$
where $V_{\lambda, \Z_{(p)}} \times V_{\lambda, \Z_{(p)}}$ is endowed with the diagonal action and $\Z_{(p)}$ with the trivial action.
\end{lem}

\begin{proof} Let us denote by $V_{\lambda, \F_{p}}$, resp. $\check{V}_{\lambda, \F_{p}}$  the reduction modulo $p$ of $V_{\lambda, \Z_{(p)}}$, resp. of $\check{V}_{\lambda, \Z_{(p)}}$. It follows from the Lemma in section 1.9 of \cite{polo-tilouine} that $V_{\lambda, \F_{p}}$ is irreducible. Furthermore as isomorphism classes of irreducible representations of $G_{\F_p}$ are determined by their highest weight, we have $V_{\lambda, \F_{p}} \simeq \check{V}_{\lambda, \F_{p}}$. Let $r$ be the composition of the canonical projection $V_{\lambda, \Z_{(p)}} \rightarrow V_{\lambda, \F_{p}}$ and of the isomorphism $V_{\lambda, \F_{p}} \simeq \check{V}_{\lambda, \F_{p}}$. Let us fix $v \in  V_{\lambda, \Z_{(p)}}$ such that $r(v) \neq 0$, let $w \in \check{V}_{\lambda, \Z_{(p)}}$ be a lifting of $r(v)$ and let $i: V_{\lambda, \Z_{(p)}} \rightarrow \check{V}_{\lambda, \Z_{(p)}}$ be the unique $G_{\Z_{(p)}}$-equivariant map sending $v$ to $w$. This map $i$ is well defined because $V_{\lambda, \Z_{(p)}}$ is irreducible by Nakayama's lemma. The reduction modulo $p$ of  $i$ is the isomorphism an isomorphism $V_{\lambda, \F_{p}} \simeq \check{V}_{\lambda, \F_{p}}$ considered above. Hence $i$ is an isomorphism by Nakayama's lemma. 
\end{proof}

From now on, we fix a dominant weight $\lambda=\lambda(k,k',0)$ with trivial central character. Let $S_{\mathrm{weight}}$ denote the finite set
\begin{equation} \label{s-weight}
S_{\mathrm{weight}}=\{p \text{ primes}\ \vert \  p<k+k'+3\}
\end{equation} and assume from now on that $p \notin S_{\mathrm{weight}}$. The non-degenerate bilinear form of the previous lemma allows to identify $V_{\lambda, \Z_{(p)}}$ and $\check{V}_{\lambda, \Z_{(p)}}$. According to Lemma \ref{duality!}, there exists a unique pairing
$$
\langle \,,\, \rangle: H^3_!(S_{K_N(3)}, {V}_{\lambda, \Z_{(p)}}) \times H^3_!(S_{K_N(3)}, {V}_{\lambda, \Z_{(p)}}) \rightarrow \Z_{(p)}.
$$
such that the diagram
\begin{equation} \label{duality-diagram}
\begin{CD}
H^3_!(S_{K_N(3)}, {V}_{\lambda, \Z_{(p)}})@. \times H^3_c(S_{K_N(3)}, {V}_{\lambda, \Z_{(p)}}) @>\langle\,,\,\rangle>> \Z_{(p)}\\
@| @VVV       @|                                                                                                                                            \\
H^3_!(S_{K_N(3)}, {V}_{\lambda, \Z_{(p)}}) @. \times H^3_!(S_{K_N(3)}, {V}_{\lambda, \Z_{(p)}}) @>\langle\,,\,\rangle>> \Z_{(p)}
\end{CD}
\end{equation}
commutes. Let us define
\begin{equation} \label{s''-tors}
S''_{\mathrm{tors}}=\{p \text{ prime }, p|\#(H^3(S_{K_N(3)}, V_{\lambda, \Z})/H^3_!(S_{K_N(3)}, V_{\lambda, \Z}))_{\mathrm{tors}}\}
\end{equation}

\begin{cor}\label{dual-int} Let $p \notin S_{\mathrm{weight}} \cup S'_{\mathrm{tors}} \cup S''_{\mathrm{tors}}$. Then the natural map
$$
H^3_!(S_{K_N(3)}, {V}_{\lambda, \Z_{(p)}}) \rightarrow \Hom_{\Z_{(p)}}(H^3_!(S_{K_N(3)}, {V}_{\lambda, \Z_{(p)}}), \Z_{(p)})
$$
induced by the pairing above is an isomorphism.
\end{cor}

\begin{proof} The commutative diagram \eqref{duality-diagram} induces the commutative diagram
\begin{equation} \label{diagram1}
\xymatrix{  
H^3(S_{K_N(3)}, {V}_{\lambda, \Z_{(p)}}) \ar[r]^-{\sim} & \ \  \Hom_{\Z_{(p)}}(H^3_c(S_{K_N(3)}, {V}_{\lambda, \Z_{(p)}}), \Z_{(p)})\\ 
H^3_!(S_{K_N(3)}, {V}_{\lambda, \Z_{(p)}}) \ \  \ 
\ar@{^{(}->}[u]
\ar@{^{(}->}[r] 
 &  \Hom_{\Z_{(p)}}(H^3_!(S_{K_N(3)}, {V}_{\lambda, \Z_{(p)}}), \Z_{(p)}) \ar@{^{(}->}[u]
\\
} 
\end{equation}
where the upper horizontal map is an isomorphism according to Lemma \ref{Pduality}. In particular, the lower horizontal map is injective. Let's prove the surjectivity of this map. The diagram obtained by tensoring over $\Z_{(p)}$ with $\Q$ the diagram above is 
\begin{equation} \label{diagram2}
\xymatrix{  
H^3(S_{K_N(3)}, {V}_{\lambda, \Q}) \ar[r]^-{\sim} & \ \  \Hom_{\Q}(H^3_c(S_{K_N(3)}, {V}_{\lambda, \Q}), \Q)\\ 
H^3_!(S_{K_N(3)}, {V}_{\lambda, \Q}) \ \  \ 
\ar@{^{(}->}[u]
\ar[r]^-{\sim} 
 &  \Hom_{\Q}(H^3_!(S_{K_N(3)}, {V}_{\lambda, \Q}), \Q). \ar@{^{(}->}[u]
\\
} 
\end{equation}
Note that the lower horizontal map is an isomorphism because it is an injection between two $\Q$-vector spaces of the same finite dimension. Let $x \in \Hom_{\Z_{(p)}}(H^3_!(S_{K_N(3)}, {V}_{\lambda, \Z_{(p)}}), \Z_{(p)})$ and let $x' \in H^3(S_{K_N(3)}, {V}_{\lambda, \Z_{(p)}})$ denote the image of $x$ by the composite of the right hand vertical map and of the inverse of the upper horizontal map of \eqref{diagram1}. The image $x'_{\Q}$ of $x'$ in $H^3(S_{K_N(3)}, {V}_{\lambda, \Q})$ is in fact an element of $H^3_!(S_{K_N(3)}, {V}_{\lambda, \Q})$. As a consequence, the image $x''$ of $x'$ in $H^3(S_{K_N(3)}, {V}_{\lambda, \Z_{(p)}})/H^3_!(S_{K_N(3)}, {V}_{\lambda, \Z_{(p)}})$ is torsion. By our assumption $p \notin S''_{\mathrm{tors}}$, this implies $x''=0$. Hence there exists $y \in H^3_!(S_{K_N(3)}, {V}_{\lambda, \Z_{(p)}})$ which maps to $x$ by the lower horizontal map of \eqref{diagram1}.
\end{proof}

Recall that by definition we have $\tilde{H}^3_!(S_{K_N}, V_{\lambda, \Z_{(p)}})=H^3_!(S_{K_N(3)}, {V}_{\lambda, \Z_{(p)}})^{K_N/K_N(3)}$.

\begin{cor}  \label{nd-inv-bilinear} Let $p \notin S_{N, 3} \cup S_{\mathrm{weight}} \cup S'_{\mathrm{tors}} \cup S''_{\mathrm{tors}}$. Then the $\Z_{(p)}$-bilinear map 
\begin{equation} \label{inv-bilinear}
\langle\,\,,\,\,\rangle:\tilde{H}^3_!(S_{K_N}, V_{\lambda, \Z_{(p)}})  \times \tilde{H}^3_!(S_{K_N}, V_{\lambda, \Z_{(p)}}) \rightarrow \Z_{(p)}
\end{equation}
obtained by restricting the bilinear map of the lower horizontal line of diagram \eqref{duality-diagram} is non-degenerate.
\end{cor}

\begin{proof} For any $g \in K_N/K_N(3)$ and any $x, y \in H^3_!(S_{K_N(3)}, {V}_{\lambda, \Z_{(p)}})$, we have $\langle g^*x, g^*y \rangle=\langle x, y \rangle$ as $g$ acts as an automorphism of the $\Q$-scheme $S_{K_N(3)}$ and hence is orientation preserving. Let us prove that the map
\begin{equation} \label{arrow1}
\tilde{H}^3_!(S_{K_N}, V_{\lambda, \Z_{(p)}}) \rightarrow \Hom_{\Z{(p)}}( \tilde{H}^3_!(S_{K_N}, V_{\lambda, \Z_{(p)}}), \Z_{(p)})
\end{equation}
defined as $x \mapsto (z \mapsto \langle x,z \rangle)$ is injective. Let $x \in \tilde{H}^3_!(S_{K_N}, V_{\lambda, \Z_{(p)}})$ .  Then, for any $g \in K_N/K_N(3)$ and any $y \in H^3_!(S_{K_N(3)}, {V}_{\lambda, \Z_{(p)}})$, we have
$
\langle x, y \rangle=\langle g^*x, g^*y \rangle=\langle x, g^*y\rangle.
$
By summing over all $g \in K_N/K_N(3)$ and dividing by $ |K_N/K_N(3)|$, we obtain
\begin{equation} \label{inv}
\langle x, y \rangle=  |K_N/K_N(3)|^{-1} \langle x,\sum_{g \in K_N/K_N(3)}g^*y \rangle.
\end{equation} Assume that $\langle x, z \rangle=0$ for any $z \in H^3_!(S_{K_N(3)}, {V}_{\lambda, \Z_{(p)}})^{K_N/K_N(3)}$. Then for any $y \in H^3_!(S_{K_N(3)}, {V}_{\lambda, \Z_{(p)}})$ we have $\langle x, y \rangle=|K_N/K_N(3)|^{-1} \langle x,\sum_{g \in K_N/K_N(3)}g^*y \rangle=0$ because $\sum_{g \in K_N/K_N(3)}g^*y$ is an element of $\tilde{H}^3_!(S_{K_N}, {V}_{\lambda, \Z_{(p)}})$. According to Corollary \ref{dual-int} this implies that $x=0$. Hence \eqref{arrow1} is injective.\\
\indent To prove its surjectivity let $\chi \in \Hom_{\Z{(p)}}( \tilde{H}^3_!(S_{K_N}, {V}_{\lambda, \Z_{(p)}}), \Z_{(p)})$. Let us denote by $\chi' \in  \Hom_{\Z{(p)}}( H^3_!(S_{K_N(3)}, {V}_{\lambda, \Z_{(p)}}), \Z_{(p)})$ the element $\chi'=\chi \circ \rho$ of where $\rho$ is the projection $H^3_!(S_{K_N(3)}, {V}_{\lambda, \Z_{(p)}}) \rightarrow \tilde{H}^3_!(S_{K_N}, {V}_{\lambda, \Z_{(p)}})$. By Corollary \ref{dual-int} there exists 
an element $x \in  H^3_!(S_{K_N(3)}, {V}_{\lambda, \Z_{(p)}})$ such that $\langle x, y \rangle=\chi'(y)$ for any $y \in H^3_!(S_{K_N(3)}, {V}_{\lambda, \Z_{(p)}})$. In particular for any element $z \in \tilde{H}^3_!(S_{K_N}, {V}_{\lambda, \Z_{(p)}})$ we have $\langle x, z \rangle=\chi'(z)=(\chi \circ \rho)(z)=\chi(z)$.  But for any such $z$ we have $\chi(z)=\langle x, z \rangle=\langle \rho(x), z \rangle$ by \eqref{inv}. Hence \eqref{arrow1} is surjective.
\end{proof}

For $g \in G(\A_f)$, we introduced the Hecke operator $T_g: H^3(S_{K_N(3)}, V_{\lambda, \Z}) \rightarrow H^3(S_{K_N(3)}, V_{\lambda, \Z})$ in definition \ref{def-hecke}. The dual Hecke operator $T_g^*: H^3_c(S_{K_N(3)}, V_{\lambda, \Z}) \rightarrow H^3_c(S_{K_N(3)}, V_{\lambda, \Z})$ is deduced from the sequence of maps defining $T_g$ by applying the Verdier duality functor 
$$
\mathbb{D}(X)=R\underline{\Hom}(X,p_3^! \Z)
$$
and using the fact that $\mathbb{D} f_*=f_! \mathbb{D}$ and $\mathbb{D} f^*=f^! \mathbb{D}$. Let us denote again by $T_g$ and $T_g^*$ the endomorphisms deduced after extending scalars to $\Z_{(p)}$. Then we have $\langle T_g x, y \rangle=\langle x, T_g^* y \rangle$ for any $g \in G(\A_f)$ and any $x \in H^3(S_{K_N(3)}, {V}_{\lambda, \Z_{(p)}})$, $y \in H^3_c(S_{K_N(3)}, {V}_{\lambda, \Z_{(p)}})$. 

\begin{lem} Assume that $p \notin S_{N,3} \cup S_{\mathrm{weight}} \cup S_{\mathrm{tors}} \cup S'_{\mathrm{tors}} \cup S''_{\mathrm{tors}}$ where $S_{K_N/K_N(3)}$, $S_{\mathrm{tors}}$, $S'_{\mathrm{tors}}$, $S_{\mathrm{weight}}$ and $S''_{\mathrm{tors}}$ are defined by \eqref{s_G}, \eqref{s-tors}, \eqref{s'-tors}, \eqref{s-weight} and \eqref{s''-tors} respectively. Then we have a $\Z_{(p)}$-linear pairing
\begin{equation} \label{pairing2}
L(\Pi_f, V_{\lambda, \Z_{(p)}}) \times L({\Pi}_f, V_{\lambda, \Z_{(p)}}) \rightarrow \Z_{(p)}.
\end{equation}
\end{lem} 

\begin{proof} Recall that in particular, $p$ does not divide the order of $K_N/K_N(3)$, and so the $\Z_{(p)}$-module $\tilde{H}^3_!(S_{K_N}, {V}_{\lambda, \Z_{(p)}})$ is a direct factor of $H^3_!(S_{K_N(3)}, V_{\lambda, \Z_{(p)}})$. As $L(\Pi_f, V_{\lambda, \Z_{(p)}})$ is a $\Z_{(p)}$-torsion free quotient of $\tilde{H}^3_!(S_{K_N}, {V}_{\lambda, \Z_{(p)}})$ by definition, it is also a direct factor of $H^3_!(S_{K_N(3)}, V_{\lambda, \Z_{(p)}})$. Let $\check{\Pi}_f$ be the contragredient of $\Pi_f$. Note that, as the representation $V_{\lambda, \Z_{(p)}}$ has trivial central character, the representation $\check{\Pi}_f$ contributes to the cohomology with coefficients $V_{\lambda, \Z_{(p)}}$. Similarly as before, the $\Z_{(p)}$-module $M(\check{\Pi}_f, {V}_{\lambda, \Z_{(p)}})$ is a direct factor of $H^3_!(S_{K_N(3)}, {V}_{\lambda, \Z_{(p)}})$. Hence the $\Z_{(p)}$-module
$L(\Pi_f, V_{\lambda, \Z_{(p)}}) \times L(\check{\Pi}_f, {V}_{\lambda, \Z_{(p)}})$ is a direct factor of $H^3_!(S_{K_N(3)}, V_{\lambda, \Z_{(p)}}) \times H^3_!(S_{K_N(3)}, V_{\lambda, \Z_{(p)}})$ and by restricting the pairing \eqref{duality-diagram} we obtain a pairing
$
L(\Pi_f, V_{\lambda, \Z_{(p)}}) \times L(\check{\Pi}_f, V_{\lambda, \Z_{(p)}}) \rightarrow \Z_{(p)}.
$
One of the assumptions on $\Pi$ in section \ref{after_ichino}, is that the central character of $\Pi$ is trivial. According to \cite[Lemma 1.1]{weissauer1}, this implies that $\Pi \simeq \check{\Pi}$ where $\check{\Pi}$ is the contragredient of $\Pi$. This concludes the proof.
\end{proof}

The following results will be useful in the next section. Let $T'$ be the maximal compact subtorus of $\mathrm{Sp}(4, \mathbb{R})$ defined by
$$
T'= \left\{ \begin{pmatrix}
x &  & y & \\
 & x' &  & y'\\
-y &  & x & \\
 & -y' &  & x' \\
\end{pmatrix} \,\,|\,\, x^2+y^2=x'^2+y'^2=1 \right\}.
$$
The Lie algebra of $T'$ is the compact Cartan subalgebra of $\mathfrak{sp}_4$ that we denoted by $\mathfrak{h}$ in section \ref{discrete_series_classification}. Let $\mbb{R}^\times_+$ be the identity component of the center of $G(\mbb{R})$. For integers $n, n', c$ such that $n+n' \equiv c \pmod 2$, let $\lambda'(n, n', c): \mbb{R}^\times_+ T' \longrightarrow \mathbb{C}^\times$ denote the character defined by
$$
\begin{pmatrix}
x &  & y & \\
 & x' &  & y'\\
-y &  & x & \\
 & -y' &  & x' \\
\end{pmatrix} \longmapsto (x+iy)^n (x'+iy')^{n'}(x^2+y^2)^\frac{c-n-n'}{2},
$$
and by $\lambda'(n, n')$ the restriction of $\lambda'(n, n', c)$ to $T'$. Note that the simple root $e_1-e_2$, respectively $2e_2$, defined in section \ref{discrete_series_classification}, coincides with the differential at the identity matrix of the restriction to $T'$ of the character $\lambda'(1, -1, 0)$, respectively $\lambda(0, 2, 0)$. Let $J \in \mrm{Sp}(4, \C)$ be the matrix
$$
J=\frac{1}{\sqrt{2}} \begin{pmatrix}
1 &  & i & \\
 & 1 &  & i\\
i &  & 1 &  \\
 & i  &  & 1 \\
\end{pmatrix}.
$$

\begin{lem} \label{tweight-t'weight} Let $w \in V_{\lambda, \C}$ be a vector of weight $\lambda(u, u', c)$ for the action of the algebraic torus T. For the action of the torus $\R^\times_+ T'$, the vector $v=Jw$ has weight $\lambda'(u,u',c)$ and the vector $\overline{v}=\overline{J}w$ has weight $\lambda'(-u,-u', c)$.
\end{lem}

\begin{proof} The first statement follows from \cite[Lemma 4.25]{lemma2}. The second statement follows from the fact that
\begin{eqnarray*}
\overline{J}^{-1}\begin{pmatrix}
x &  & y & \\
 & x' &  & y'\\
-y &  & x & \\
 & -y' &  & x' \\
\end{pmatrix} \overline{J} &=& J \begin{pmatrix}
x &  & y & \\
 & x' &  & y'\\
-y &  & x & \\
 & -y' &  & x' \\
\end{pmatrix} \overline{J}\\
&=& \begin{pmatrix}
x-iy &  &  & \\
 & x'-iy' &  & \\
 &  & x+iy & \\
 &  &  & x'+iy' \\
\end{pmatrix}
\end{eqnarray*}
and hence
\begin{eqnarray*}
\begin{pmatrix}
x &  & y & \\
 & x' &  & y'\\
-y &  & x & \\
 & -y' &  & x' \\
\end{pmatrix} \overline{J}w &=& \overline{J}  \begin{pmatrix}
x-iy &  &  & \\
 & x'-iy' &  & \\
 &  & x+iy & \\
 &  &  & x'+iy' \\
\end{pmatrix} w\\
&=& (x-iy)^{-k}(x'-iy')^{k'}(x^2+y^2)^{\frac{c+k-k'}{2}} \overline{J}w\\
&=&  (x+iy)^k (x'+iy')^{-k'} (x^2+y^2)^{\frac{c-k+k'}{2}}\overline{J}w\\
&=& \lambda'(k,-k',c) \left( \begin{pmatrix}
x &  & y & \\
 & x' &  & y'\\
-y &  & x & \\
 & -y' &  & x' \\
\end{pmatrix} \right) \overline{J}w.
\end{eqnarray*}
\end{proof}

\begin{lem} \label{norm-v} Let $w \in V_{\lambda, \C}$ be a vector of weight $\lambda(-k, k', 0)$, let
$
[\,,\,]_{\C}: V_{\lambda, \C} \otimes V_{\lambda, \C} \rightarrow \C
$
denote the pairing obtained from $[\,,\,]$ after extending scalars to $\C$. Then
$$
[Jw, \overline{J}w] \neq 0.
$$
\end{lem}

\begin{proof}The pairing $[\,,\,]_{\C}$ is $G_{\C}$-equivariant. In particular, given two vectors $v_1$ and $v_2$ of weights $\lambda'(u_1, u'_1, 0)$ and $\lambda'(u_2, u'_2, 0)$, we have 
$$
[v_1, v_2]_{\C} \neq 0 \iff u_1+u_2=u'_1+u'_2=0.
$$
The weight $\lambda(-k,k',0)$ belongs to the orbit under the action of $W$ of the dominant weight $\lambda(k,k',0)$, hence has multiplicity one in $V_{\lambda, \C}$. Then it follows from Lemma \ref{tweight-t'weight} that $\lambda'(-k, k', 0)$, which is the weight of $Jw$, has multiplicity one in $V_{\lambda, \C}$ and a similar argument applies to $\lambda(k, -k', 0)$, which is the weight of $\overline{J}w$. As a consequence, if $[Jw, \overline{J}w]_{\C} = 0$, then $[Jw, w']_{\C}=0$ for any $w' \in V_{\lambda, \C}$. This contradicts the fact that $[\,,\,]_{\C}$ is non-degenerate. Hence $
[Jw, \overline{J}w] \neq 0.
$
\end{proof}

Let 
$$
X_{(1, -1)}=d \kappa \left(  \begin{pmatrix}
 & 1 \\
 &  \\
\end{pmatrix}\right)= \frac{1}{2} \begin{pmatrix}
 & 1 &  & -i\\
-1 &  & -i & \\
 & i &  & 1 \\
i &   & -1 & \\
\end{pmatrix} \in \mathfrak{k}_{\C}
$$
and let
$$
X_{(-1, 1)}=d \kappa \left(  \begin{pmatrix}
 &  \\
1 &  \\
\end{pmatrix}\right)= \frac{1}{2} \begin{pmatrix}
 & 1 &  & i\\
-1 &  & i & \\
 & -i &  & 1 \\
-i &   & -1 & \\
\end{pmatrix} \in \mathfrak{k}_{\C}
$$
These are root vectors corresponding to the positive (resp. negative), compact root. Let us denote $v=Jw$ and $\overline{v}=\overline{J}w$. 

\begin{lem} \label{pairing-computation} For any $i, j \in \Z$ we have
$$
\left[X_{(1,-1)}^i v, X_{(-1,1)}^j \overline{v}\right]_{\C}=\left\{
    \begin{array}{ll}
        0 & \mbox{ if } i \neq j \\
        (-1)^i \frac{i! (k+k')!}{(k+k'-i)!}[v, \overline{v}] & \mbox{ if } i = j.
    \end{array}
\right.
$$
\end{lem}

\begin{proof} The first statement follows from weight reasons. Let us prove the second one. As $[\,,\,]_{\C}$ is compatible with the action of the Lie algebra $\mathfrak{k}_{\C}$, we have 
$$
\left[X_{(1,-1)}^i v, X_{(-1,1)}^i \overline{v}\right]_{\C}=(-1)^i \left[v, X_{(1,-1)}^i  X_{(-1,1)}^i \overline{v}\right]_{\C}
$$
The irreducible sub $\C[K_\infty]$-module of $V_{\lambda, \C}$ generated by the vector $\overline{v}$ is equivalent to $\tau_{(k', -k)}$ and $\overline{v}$ is a highest weight vector of it. By an easy inductive application of \eqref{formula5} and \eqref{formula6} we obtain 
$
 X_{(1,-1)}^i  X_{(-1,1)}^i \overline{v}=\frac{i! (k+k')!}{(k+k'-i)!} \overline{v}$ and the conclusion follows.
\end{proof}

\section{Definition of the periods} For $\Pi_\infty \in P(V_{\lambda, \C})$ we have 
$$
 H^3(\mathfrak{g}_{\C},K_\infty, V_{\lambda, \C} \otimes \Pi_\infty)=\mrm{Hom}_{K_\infty}\left( \bigwedge^3 \mathfrak{g}_{\C}/\mathfrak{k}_{\C}', V_{\lambda, \C} \otimes \Pi_\infty \right)
$$
and it is known that this $\C$-vector space is one-dimensional (see \cite[Proposition 3.7]{lemma2} for example). The character $\lambda$ has the form $\lambda=\lambda(k,k',c)$ for $k \geq k' \geq 0$ and $k+k' \equiv c \pmod 2$.  Recall that $\Pi_\infty^W|_{G(\R)_+}=\Pi_\infty^{2,1} \oplus \Pi_\infty^{1,2}$.

\begin{lem} \label{archi-vector}
Let $\phi \in \Pi_\infty^{2,1}$ be a lowest weight vector of the minimal $K_\infty$-type, let $w \in V_{\lambda, \C}$ be a vector of weight $\lambda(-k, k', c)$ and let $v=Jw \in V_{\lambda, \C}$. 
\begin{enumerate}
\item[\rm{(1)}] There exists a unique non-zero element 
$$
[\phi, v] \in \mrm{Hom}_{K_\infty}\left( \bigwedge^3 \mathfrak{g}_{\C}/\mathfrak{k}_{\C}', V_{\lambda, \C} \otimes \Pi_\infty^W \right)
$$
such that
$$
[\phi, v](X_{(2, 0)} \wedge X_{(1, 1)} \otimes X_{(0, -2)})=\sum_{i = 0}^{k+k'} (-1)^i  X_{(1, -1)}^i v \otimes {X}_{(1, -1)}^{k+k'+4-i} \phi
$$
where $X_{(2, 0)}$, $X_{(1, 1)} \in \mathfrak{p}^+$ and $X_{(0, -2)} \in \mathfrak{p}^-$ are the root vectors defined in section \ref{discrete_series_classification}. 
\item[\rm{(2)}] The map $[\phi, v]$ factors through the canonical projection
$$\bigwedge^3 \mathfrak{g}_{\C}/\mathfrak{k}_{\C}' \rightarrow \bigwedge^2 \mathfrak{p}^+ \otimes_{\C} \mathfrak{p}^- \rightarrow \tau_{(3,-1)}$$
where $\tau_{(3,-1)}$ is the irreducible sub $\C[K_\infty]$-module of $\bigwedge^2 \mathfrak{p}^+ \otimes_{\C} \mathfrak{p}^-$ generated by the highest weight vector $X_{(2, 0)} \wedge X_{(1, 1)} \otimes X_{(0, -2)}$.
\end{enumerate}
\end{lem}

\begin{proof} To prove the existence part of statement (1), note that $X_{(2, 0)} \wedge X_{(1, 1)} \otimes X_{(0, -2)}$ is a highest weight vector of weight $\lambda'(3, -1,0)$ of the $\C[K_\infty]$-module $\bigwedge^3 \mathfrak{g}_{\C}/\mathfrak{k}_{\C}'$. Hence, as we know that $\mrm{Hom}_{K_\infty}\left( \bigwedge^3 \mathfrak{g}_{\C}/\mathfrak{k}_{\C}', V_{\lambda, \C} \otimes \Pi_\infty^W \right)$ has dimension $1$, to define a morphism in this space it is enough to define the image of $X_{(2, 0)} \wedge X_{(1, 1)} \otimes X_{(0, -2)}$, under the condition that this image is a vector of $V_{\lambda, \C} \otimes \Pi_\infty^W$ which is a highest weight vector and which has the same weight as $X_{(2, 0)} \wedge X_{(1, 1)} \otimes X_{(0, -2)}$. According to Lemma \ref{tweight-t'weight}, the vector $v$ has weight $\lambda'(-k, k', c)$ and hence the vector $X_{(1, -1)}^i v$ has weight $\lambda'(-k+i, k'-i,c)$. On the other hand according to Proposition \ref{lpaquet}, the vector $\phi \in \Pi_\infty^W$ has weight $\lambda'(-k'-1, k+3, -c)$ and hence the vector $X_{(1, -1)}^{k+k'+4-i} \phi$ has weight $\lambda'(k+3-i, -k'-1-i, -c)$. As a consequence, the vector $\sum_{i = 0}^{k+k'} (-1)^i  X_{(1, -1)}^i v \otimes {X}_{(1, -1)}^{k+k'+4-i} \phi$ has the same weight as $X_{(2, 0)} \wedge X_{(1, 1)} \otimes X_{(0, -2)}$. Furthermore
$$
X_{(1,-1)} \left( \sum_{i = 0}^{k+k'} (-1)^i  X_{(1, -1)}^i v \otimes {X}_{(1, -1)}^{k+k'+4-i} \phi \right)
$$
\begin{align*}
& =\sum_{i = 0}^{k+k'} (-1)^i  X_{(1, -1)}^{i+1} v \otimes {X}_{(1, -1)}^{k+k'+4-i} \phi+\sum_{i = 0}^{k+k'} (-1)^i  X_{(1, -1)}^i v \otimes {X}_{(1, -1)}^{k+k'+5-i} \phi\\
& = \sum_{i = 0}^{k+k'-1} (-1)^i  X_{(1, -1)}^{i+1} v \otimes {X}_{(1, -1)}^{k+k'+4-i} \phi+\sum_{i = 1}^{k+k'} (-1)^i  X_{(1, -1)}^i v \otimes {X}_{(1, -1)}^{k+k'+5-i} \phi\\
& = \sum_{i = 1}^{k+k'} (-1)^{i-1}  X_{(1, -1)}^{i} v \otimes {X}_{(1, -1)}^{k+k'+5-i} \phi+\sum_{i = 1}^{k+k'} (-1)^i  X_{(1, -1)}^i v \otimes {X}_{(1, -1)}^{k+k'+5-i} \phi\\
& = 0.
\end{align*}
This means that $\sum_{i = 0}^{k+k'} (-1)^i  X_{(1, -1)}^i v \otimes {X}_{(1, -1)}^{k+k'+4-i} \phi$ is a highest weight vector. As a consequence the element $[\phi, v]$ of statement (1) of the Lemma exists. Its unicity follows from the fact that $\mrm{Hom}_{K_\infty}\left( \bigwedge^3 \mathfrak{g}_{\C}/\mathfrak{k}_{\C}', V_{\lambda, \C} \otimes \Pi_\infty^W \right)$ has dimension $1$. Statement (2) is a direct consequence of the $1$-dimensionality of $\mrm{Hom}_{K_\infty}\left( \bigwedge^3 \mathfrak{g}_{\C}/\mathfrak{k}_{\C}', V_{\lambda, \C} \otimes \Pi_\infty^W \right)$ and of the construction of $[\phi, v]$.
\end{proof}

Let $\Pi=\Pi_\infty \otimes \Pi_f$ be as in section \ref{after_ichino}. Let $\Pi_f^0$ a model of $\Pi_f$ over the rationality field $E(\Pi_f)$. Assume in addition that $\Pi$ has level $1$. For any embedding $\sigma: E(\Pi_f) \rightarrow \C$, let us define the representation $\, ^\sigma\!\,\Pi_f=\Pi_f^0 \otimes_\sigma \C$. Let us introduce the $\C$-linear map $^\sigma\! \omega_{\phi, v}$ defined by
\begin{equation} \label{r} 
^\sigma\! \omega_{\phi, v}:\, ^\sigma\!\,\Pi_f^{K_N} \rightarrow L(\Pi_f, V_{\lambda, \C}),\, \psi \mapsto [\phi, v] \otimes \psi.
\end{equation}

\begin{defn} Let $M(\Pi_f, V_{\lambda, \R})$ denote the sub $\R$-vector space of $M(\Pi_f, V_{\lambda, \C})$ of vectors which are fixed under the involution defined on each factor
$$
H^3(\mathfrak{g}_{\C}, K_\infty, V_{\lambda, \C} \otimes \Pi_\infty) \otimes \,\!^\sigma \Pi_f^{K_N}=\mrm{Hom}_{K_\infty}\left( \bigwedge^3 \mathfrak{g}_{\C}/\mathfrak{k}_{\C}', V_{\lambda, \C} \otimes \Pi_\infty \otimes  \,\!^\sigma \Pi_f^{K_N} \right)
$$
of the decomposition \eqref{dec_again} by $\overline{h}:X \mapsto \overline{h(\overline{X})}$. 
\end{defn}

\begin{rems} Note that this action is well defined because for any cusp form $\psi \in \Pi_\infty \otimes  \,\!^\sigma \Pi_f^{K_N}$, the cusp form $\overline{\psi}$ still belongs to $ \Pi_\infty \otimes  \,\!^\sigma \Pi_f^{K_N}$ via the identification of $\Pi_\infty \otimes  \,\!^\sigma \Pi_f$ with its contragredient via the Petersson inner product.
\end{rems}

\begin{lem} \label{norm-basis} Let $\re: L(\Pi_f, V_{\lambda, \C}) \rightarrow L(\Pi_f, V_{\lambda, \R})$ be the $\R$-linear projection defined as $h \mapsto \frac{1}{2}(h+\overline{h})$.
\begin{enumerate}
\item[(1)] For any $\phi, v$ as above, the composition
$$
\bigoplus_{\sigma: E(\Pi_f) \rightarrow \C} \,\!^\sigma\!\,\Pi_f^{K_N} {\hookrightarrow} L(\Pi_f, V_{\lambda, \C}) \overset{\re}{\twoheadrightarrow} L(\Pi_f, V_{\lambda, \R}),
$$
is an isomorphism of $\R$-vector spaces of dimension $2[E(\Pi_f): \Q]$ 
where the first map is $\bigoplus_{\sigma: E(\Pi_f) \rightarrow \C} \,\!^\sigma \omega_{\phi, v}$.
\item[(2)] For any embedding $\sigma: E(\Pi_f) \rightarrow \C$, let $^\sigma\!\varphi_\infty \in {}^\sigma\Pi_\infty$ and ${}^\sigma\!\varphi_f \in {}^\sigma\Pi_f^{K_N}$ be vectors such that $^\sigma\!\varphi =\,\!^\sigma\!\varphi_\infty \otimes {}^\sigma\!\varphi_f$ via the isomorphism \eqref{iso}. Let $\sigma_1, \ldots, \sigma_r$ denote the embeddings of $E(\Pi_f)$ in $\C$. Then
$$
\big( \!\re {}^{\sigma_1}\omega_{\varphi_\infty, v}( {}^{\sigma_1}\varphi_f), \ldots, \re {}^{\sigma_r}\omega_{\varphi_\infty, v}( {}^{\sigma_r}\varphi_f),\re {}^{\sigma_1}\omega_{\varphi_\infty, v}( \sqrt{-1} \,{}^{\sigma_1}\!\varphi_f), \ldots, \re {}^{\sigma_r}\omega_{\varphi_\infty, v}( \sqrt{-1} \,{}^{\sigma_r}\!\varphi_f) \big)
$$ is a basis of the $\R$-vector space $L(\Pi_f, V_{\lambda, \R})$.
\end{enumerate}
\end{lem}

\begin{proof} The second statement is an direct consequence of the first. To prove the first, as the dimensions of the $\R$-vector spaces $\bigoplus_{\sigma: E(\Pi_f) \rightarrow \C} \,\!^\sigma\!\,\Pi_f^{K_N}$ and $L(\Pi_f, V_{\lambda, \R})$ are finite and equal to $2[E(\Pi_f):\Q]$, it is enough to prove the surjectivity of the $\R$-linear map of the statement. To this end, let us fix a non-zero vector $\,\!^\sigma\!\,\psi \in \,\!^\sigma\!\,\Pi_f^{K_N}$ for any $\sigma$. The vectors $\!^\sigma\!\,\psi$ and $\sqrt{-1} \,^\sigma\!\,\psi$ form a basis of the $\R$-vector space underlying $\,\!^\sigma\!\,\Pi_f^{K_N}$. Furthermore by the isomorphism \eqref{dec_again}, Theorem \ref{sigma-endo}, Proposition \ref{multiplicity} and the remark at the beginning of Section 5 we have
$$
L(\Pi_f, V_{\lambda, \C}) \simeq \bigoplus_{\sigma: E(\Pi_f) \rightarrow \C} \left( \mrm{Hom}_{K_\infty} \left( \bigwedge^3 \mathfrak{g}_{\C}/\mathfrak{k}_{\C}', V_{\lambda, \C} \otimes \Pi_\infty^{2,1} \right) \oplus \right.
$$
$$\left. \mrm{Hom}_{K_\infty}\left( \bigwedge^3 \mathfrak{g}_{\C}/\mathfrak{k}_{\C}', V_{\lambda, \C} \otimes \Pi_\infty^{1,2} \right) \right) \otimes  \,\!^\sigma \Pi_f^{K_N} 
$$
and the complex conjugation exchanges the subspace $\mrm{Hom}_{K_\infty} \left( \bigwedge^3 \mathfrak{g}_{\C}/\mathfrak{k}_{\C}', V_{\lambda, \C} \otimes \Pi_\infty^{2,1} \right)$ and the subspace $\mrm{Hom}_{K_\infty} \left( \bigwedge^3 \mathfrak{g}_{\C}/\mathfrak{k}_{\C}', V_{\lambda, \C} \otimes \Pi_\infty^{1,2} \right)$. As a consequence, the vectors 
$$
\left\{ [\phi, v] \otimes \,\!^\sigma \psi,\, \sqrt{-1}[\phi, v] \otimes \,\!^\sigma \psi, \,\overline{[\phi, v]} \otimes \,\!^\sigma \psi,\, \sqrt{-1}\,\overline{[\phi, v]} \otimes \,\!^\sigma \psi \right\}_{\sigma: E(\Pi_f) \rightarrow \C}
$$  form a basis of the $\R$-vector space underlying $L(\Pi_f, V_{\lambda, \C})$. This implies that the vectors 
$$
\left\{ ([\phi, v]+ \overline{[\phi, v]}) \otimes \,\!^\sigma \psi,\, ([\phi, v]+ \overline{[\phi, v]}) \otimes \sqrt{-1}\,^\sigma \psi   \right\}_{\sigma: E(\Pi_f) \rightarrow \C}
$$
 form a basis of the $\R$-vector space $L(\Pi_f, V_{\lambda, \R})$. For any $\sigma: E(\Pi_f) \rightarrow \C$ the vector $([\phi, v]+ \overline{[\phi, v]}) \otimes \,\!^\sigma \psi$ (resp. $ ([\phi, v]+ \overline{[\phi, v]}) \otimes \sqrt{-1}\,^\sigma \psi $) is the image of $\,\!^\sigma \psi$
 (resp. of $\sqrt{-1}\,^\sigma \psi $) by the the composite map
$$
\bigoplus_{\sigma: E(\Pi_f) \rightarrow \C} \,\!^\sigma\!\,\Pi_f^{K_N} {\hookrightarrow} L(\Pi_f, V_{\lambda, \C}) \overset{\re}{\twoheadrightarrow} L(\Pi_f, V_{\lambda, \R})
$$
of the statement. This proves the surjectivity of this map as claimed.
\end{proof}

\begin{rems} For any $1 \leq i \leq r$, the vectors $\re {}^{\sigma_i}\omega_{\varphi_\infty, v}( {}^{\sigma_i}\varphi_f)$ and $\re {}^{\sigma_i}\omega_{\varphi_\infty, v}( \sqrt{-1} \,{}^{\sigma_i}\varphi_f)$ do not depend on the choice of $^{\sigma_i}\varphi_\infty$ and of $^{\sigma_i}\varphi_f$ and only depend on $^{\sigma_i}\varphi=\!^{\sigma_i}\varphi_\infty \otimes \!^{\sigma_i}\varphi_f$ and $v$.
\end{rems}

According to Lemma \ref{norm-v}, the pairing $[v, \overline{v}]$ is a non-zero complex number. As a consequence we can normalize $v$ in such a way that $[v, \overline{v}]=1$. Following \cite[\S 6]{hida} , we can introduce the period of interest in this work. Let us choose a basis $(\delta_1, \ldots, \delta_{2r})$ of the free ${\Z_{(p)}}$-module $L(\Pi_f, V_{\lambda, {\Z_{(p)}}})$. Let us denote by $(\omega_1, \ldots, \omega_{2r})$ the basis of the second assertion in the previous lemma and let $U \in \mrm{GL}_{2r}(\R)$ be such that $(\delta_1, \ldots, \delta_{2r})U=(\omega_1, \ldots, \omega_{2r})$. Then we define
\begin{equation} \label{period}
\Omega(\Pi_f, \varphi, v , (\delta_1, \ldots, \delta_{2r}))=\det(U).
\end{equation}

\begin{rems} \label{remark-period} Under the above normalization of $v$, the vector $v$ is unique up to multiplication by $\pm1$. Furthermore, if we change the basis $(\delta_1, \ldots, \delta_{2r})$ by another basis $(\delta'_1, \ldots, \delta_{2r}')$ the period $\Omega(\Pi_f, \varphi, v , (\delta_1, \ldots, \delta_{2r}))$ is changed by an element of $\Z_{(p)}^\times$. Hence, the image of the real number $\Omega(\Pi_f, \varphi, v , (\delta_1, \ldots, \delta_{2r}))$ in $\R^\times/\Z_{(p)}^\times$ is independent of the choice of $v$ normalized as above and of the basis $(\delta_1, \ldots, \delta_{2r})$. In what follows, the image of $\Omega(\Pi_f, \varphi, v , (\delta_1, \ldots, \delta_{2r}))$ in $\R^\times/\Z_{(p)}^\times$ will be denoted by $\Omega(\Pi_f)$.  
\end{rems}

\section{Discriminant and adjoint $L$-values}\label{section:Discriminant and adjoint $L$-values}

In what follows $\Pi$ denotes an irreducible cuspidal automorphic representation of $G(\A)$ satisfying the assumptions of section \ref{after_ichino}. Let $\mathbf{1}$ denote the generator of the one-dimensional $\C$-vector space $\bigwedge^6 \mathfrak{sp}_{4, \C}/\mathfrak{k}_{\C}$ defined as
$$
\mathbf{1}=X_{(2,0)} \wedge X_{(1,1)} \wedge X_{(0,2)} \wedge X_{(-2,0)} \wedge X_{(-1,-1)} \wedge X_{(0, -2)}. 
$$
This determines a left translation invariant measure $d\mu$ on $\mrm{Sp}(4, \R)/K_\infty=G(\R)_+/K'_\infty$ in a standard way. By our normalization of the vectors $X_{(2,0)}, X_{(1,1)},  X_{(0,2)}, X_{(-2,0)}, X_{(-1,-1)}$ and $X_{(0, -2)}$ (see Section \ref{discrete_series_classification}), this measure coincides with the standard invariant measure $dX dY/\det(Y)^3$ via the isomorphism $G(\R)_+/K'_\infty \simeq \mathcal{H}_+$. Let $dg_\infty$ be the left invariant measure on $G(\R)_+/\R_+^\times$ attached to $d\mu$ by the construction \eqref{measures}. For every prime number $p$, let $dg_p$ be the unique translation invariant measure on $G(\Q_p)$ such that $\mrm{vol}(G(\Z_p), dg_p)=1$ and let $dg$ be the measure on $Z(\A)\backslash G(\A)$ defined by $dg=\prod_{v} dg_v$.

\begin{pro} \label{pp} The pairing obtained from \eqref{pairing2} after extending coefficients from $\Z_{(p)}$ to $\C$ is given by
\begin{multline} \label{pairing3}
\left( \bigoplus_{\sigma: E(\Pi_f) \rightarrow \C}\bigoplus_{\Pi_\infty \in P(V_{\lambda, \C})} H^3(\mathfrak{g}_{\C},K'_\infty, V_{\lambda, \C} \otimes \Pi_\infty)^{\oplus m(\Pi_\infty \otimes ^\sigma\!\,\Pi_f)} \otimes_{\C}\, ^\sigma\!\,\Pi_f^{K_N} \right)\\
\otimes \left( \bigoplus_{\sigma: E(\Pi_f) \rightarrow \C}\bigoplus_{\Pi_\infty \in P(V_{\lambda, \C})} H^3(\mathfrak{g}_{\C},K'_\infty, V_{\lambda, \C} \otimes \Pi_\infty)^{\oplus m(\Pi_\infty \otimes ^\sigma\!\,\Pi_f)} \otimes_{\C}\, ^\sigma\!\,\Pi_f^{K_N} \right) \rightarrow \C
\end{multline}
where \eqref{pairing3} is induced by the composite
\begin{multline*}
\left( \Hom_{K'_\infty}\left( \bigwedge^3 \mathfrak{g}_{\C}/\mathfrak{k}_{\C}', V_{\lambda, \C} \otimes \Pi_\infty \right) \otimes \,\!^\sigma\!\,\Pi_f^{K_N} \right) \otimes \left( \Hom_{K'_\infty}\left( \bigwedge^3 \mathfrak{g}_{\C}/\mathfrak{k}_{\C}', V_{\lambda, \C} \otimes \Pi'_\infty \right) \otimes \,\!^{\sigma'}\!\Pi_f^{K_N} \right) \\
\rightarrow \Hom_{K'_\infty}\left( \bigwedge^6 \mathfrak{g}_{\C}/\mathfrak{k}_{\C}', V_{\lambda, \C} \otimes V_{\lambda, \C} \otimes \Pi_\infty \otimes \Pi'_\infty \right) \otimes  \,\!^{\sigma}\!\Pi_f^{K_N} \otimes  \,\!^{\sigma'}\!\Pi_f^{K_N}\\
\rightarrow \Hom_{K'_\infty}\left( \bigwedge^6 \mathfrak{g}_{\C}/\mathfrak{k}_{\C}', \Pi_\infty \otimes \Pi'_\infty \right) \otimes  \,\!^{\sigma}\!\Pi_f^{K_N} \otimes  \,\!^{\sigma'}\!\Pi_f^{K_N}
\rightarrow \C
\end{multline*}
where the first map is the exterior product, the second map is induced by $[\,,\,]_{\C}$ and the third is the composition of the evaluation at $\textbf{1}$ followed by $\int_{\R_+^\times G(\Q) \backslash G(\A)/K_N}\,\,dg$.
\end{pro}

\begin{proof} This follows from the compatibility of the Poincar\'e duality with the comparison isomorphism (\ref{dec_again}) between Betti and de Rham cohomology and the description given in \cite[(5)]{borel1} of the Poincar\'e duality in de Rham cohomology in terms of $(\mathfrak{g}_{\C}, K'_\infty)$-cohomology. 
\end{proof}

For any $\Z_{(p)}$-algebra $R$, let us denote by 
$$
\langle\,,\,\rangle_{R}: L(\Pi_f, V_{\lambda, R}) \times L(\Pi_f, V_{\lambda, R}) \rightarrow R
$$
the pairing deduced from \eqref{pairing2} after extending scalars to $A$. Then, we have the commutative diagram
\begin{equation} \label{cd-poincare}
\begin{CD}
L(\Pi_f, V_{\lambda, \R}) \times L(\Pi_f, V_{\lambda, \R}) @>\langle\,,\,\rangle_{\R}>> \R\\
@VVV                                                                                                                                         @VVV\\
L(\Pi_f, V_{\lambda, \C}) \times L(\Pi_f, V_{\lambda, \C}) @>\langle\,,\,\rangle_{\C}>> \C\\
\end{CD}
\end{equation}
where the vertical arrows are the natural inclusions.

\begin{lem} There exists $\epsilon_\lambda \in \Z$ such that for any $v, w \in L(\Pi_f, V_{\lambda, \C})$ we have
$$
\langle w, v\rangle_{\C}=(-1)^{\epsilon_\lambda} \langle v, w \rangle_{\C}.
$$
\end{lem}

\begin{proof} In the statement of Proposition \ref{pp}, the first map in the sequence defining \eqref{pairing3} is alternate as it is defined as the exterior product of differential forms of degree $3$. The second map is induced by the $G_{\C}$-invariant bilinear form $[\,,\,]_{\C}$ on $V_{\lambda, \C} \otimes V_{\lambda, \C}$. As $V_{\lambda, \C}$ is irreducible, it is an easy consequence of Schur Lemma that such a bilinear form is unique up to a scalar. Let $[\,,\,]'_{\C}$ be the bilinear form on $V_{\lambda, \C} \otimes V_{\lambda, \C}$ defined by $[v,w]'_{\C}=[w,v]_{\C}$. There exists $\lambda \in \C$ such that $[v,w]'_{\C}=\lambda [v,w]_{\C}$. Let $v, w \in V_{\lambda, \C}$ such that $[v,w]_{\C} \neq 0$. Then
\begin{eqnarray*}
[v,w]_{\C} &=& \lambda [v,w]'_{\C}\\
&=& \lambda [w,v]_{\C}\\
&=& \lambda^2 [w,v]'_{\C}\\
&=& \lambda^2 [v,w]_{\C}.
\end{eqnarray*}
Hence $\lambda = \pm 1$ and $[\,,\,]_{\C}$ is either symmetric or alternate. The conclusion now follows from the fact that the last map in the sequence defining \eqref{pairing3} is symmetric.
\end{proof}

\begin{lem} \label{std-tamagawa} As measures on $Z(\A) G(\Q) \backslash G(\A)$, we have
$$
dg=\frac{\pi^3}{270} dg^{\mathrm{Tam}}.
$$
\end{lem}

\begin{proof} By Weil conjecture on Tamagawa numbers proved by Kottwitz (see \cite{kottwitz}), we have 
$$
\int_{Z(\A) G(\Q) \backslash G(\A)} dg^{\mathrm{Tam}}=1.
$$
On the other hand, as $\mrm{vol}(K_\infty, dg)=\mrm{vol}(G(\Z_p), dg)=1$ we have
$$
\int_{Z(\A) G(\Q) \backslash G(\A)} dg = \int_{Z(\A) G(\Q) \backslash G(\A)/K_\infty G(\widehat{\Z})} dg.
$$
It follows from the definition of $dg$ that there is an isomorphism of measured spaces $$(Z(\A) G(\Q) \backslash G(\A)/K_\infty G(\widehat{\Z}), dg) \simeq (\mrm{PSp}(4, \Z) \backslash \mathcal{H}_+, dX dY/\det(Y)^3)
$$
and in particular
$$
\int_{Z(\A) G(\Q) \backslash G(\A)/K_\infty G(\widehat{\Z})} dg=\int_{\mrm{PSp}(4, \Z) \backslash \mathcal{H}_+} dX dY/\det(Y)^3.
$$
Let $\xi(s)$ denotes the complete Riemann zeta function $\xi(s)=\pi^{-\frac{s}{2}}\Gamma(\frac{s}{2}) \zeta(s)$. According to \cite[Theorem 11]{siegel} the last displayed integral is equal to $2 \xi(2) \xi(4)$. The conclusion now follows from the well known equalities $\zeta(2)=\pi^2/6$ and $\zeta(4)=\pi^4/90$.
\end{proof}

\begin{pro} \label{pd-computation} Let $\sigma: E(\Pi_f) \rightarrow \C$ be an embedding, let $^{\sigma} \!\varphi_f \in \!\,^\sigma \Pi_f$ be the normalized vector, let ${}^\sigma \omega_{\varphi_\infty, v}( {}^\sigma \varphi_f) \in L(\Pi_f, V_{\lambda, \C})$ be defined by \eqref{r}. For any $0 \leq i \leq k+k'$, any $0 \leq r \leq 4$, any $0 \leq u \leq r$ such that $i-u \geq 0$, let us denote
\begin{eqnarray*}
r_{i, u, r}^{k,k'} &=& \frac{(k+k'+u-i)!(k+k'+4-i)!(i+r-u)!}{(i-u)!(k+k'-i)!(k+k'+4-i-r+u)},\\
s_{i, u}^{k, k'} &=& \frac{(i-u)!}{(k+k'-i+u)!},\\
t_{i, u, r}^{k, k'} &=& \frac{(k+k'+4+u-r-i)!}{(i+r-u)!}
\end{eqnarray*}
let us denote $a_0=-1, a_1=-\frac{1}{4}, a_2=\frac{1}{72}, a_3=-\frac{1}{72}, a_4=-\frac{1}{576}$, let us define 
$C'_{k,k'} \in \Q$ by
\begin{multline}\label{equation:definitionC'_{k,k'}}
C'_{k,k'}
= \frac{(-1)^{k+k'}(k+k')! (k+k'+4)!}{3!^2} \sum_{r=0}^4  \sum_{i=0}^{k+k'} \sum_{\substack{0 \leq u, u' \leq r \\  0 \leq i-u}}
\\ 
(-1)^{r+u+u'}a_r \binom{r}{u} \binom{r}{u'} r^{k,k'}_{i, u, r} r^{k,k'}_{i-u+u', u', r} s_{i, u}^{k, k'} t_{i, u, r}^{k, k'}
\end{multline}
Then we have 
\begin{equation}
\langle {}^\sigma \omega_{\varphi_\infty, v}( {}^\sigma \varphi_f), \overline{{}^\sigma \omega_{\varphi_\infty, v}( {}^\sigma \varphi_f)} \rangle_{\C}
=\frac{\pi^3}{3^3 \cdot 5} C'_{k,k'} \prod_{l | N} (l^2+1)^{-1} \langle {}^\sigma \varphi, {}^\sigma \varphi \rangle.
\end{equation}
\end{pro}

\begin{proof} We need to calculate the image of the vector ${}^\sigma \omega_{\varphi_\infty, v}( {}^\sigma \varphi_f) \otimes \overline{{}^\sigma \omega_{\varphi_\infty, v}( {}^\sigma \varphi_f)}$ by the sequence of maps defining \eqref{pairing3}. Let us introduce the following notation: $e_1=X_{(2,0)}, e_2=X_{(1,1)}, e_3=X_{(0,2)}, e_4=X_{(-2,0)}, e_5=X_{(-1,-1)}, e_6=X_{(0,-2)}$. Then, by definition of the exterior product, we have
$$
({}^\sigma \omega_{\varphi_\infty, v}( {}^\sigma \varphi_f) \wedge \overline{{}^\sigma \omega_{\varphi_\infty, v}( {}^\sigma \varphi_f)})(\textbf{1})
$$
\begin{align*}
&= \frac{1}{3!^2}\sum_{\sigma \in \mathcal{S}_6}\epsilon(\sigma) ({}^\sigma \omega_{\varphi_\infty, v}( {}^\sigma \varphi_f) \otimes \overline{{}^\sigma \omega_{\varphi_\infty, v}( {}^\sigma \varphi_f)})(e_{\sigma(1)} \otimes e_{\sigma(2)} \otimes e_{\sigma(3)} \otimes e_{\sigma(4)} \otimes e_{\sigma(5)} \otimes e_{\sigma(6)})\\
&= \frac{1}{3!^2}\sum_{\sigma \in \mathcal{S}_6}\epsilon(\sigma) {}^\sigma \omega_{\varphi_\infty, v}( {}^\sigma \varphi_f)(e_{\sigma(1)} \otimes e_{\sigma(2)} \otimes e_{\sigma(3)}) \otimes \overline{{}^\sigma \omega_{\varphi_\infty, v}( {}^\sigma \varphi_f)}(e_{\sigma(4)} \otimes e_{\sigma(5)} \otimes e_{\sigma(6)})
\end{align*}
According to the second statement of Lemma \ref{archi-vector} we have
$
{}^\sigma \omega_{\varphi_\infty, v}( {}^\sigma \varphi_f)(e_{\sigma(1)}, e_{\sigma(2)}, e_{\sigma(3)})=0
$
whenever the image of $e_{\sigma(1)} \otimes  e_{\sigma(2)} \otimes e_{\sigma(3)}$ by the projection 
$$
\bigotimes^3 \mathfrak{g}_{\C}/\mathfrak{k}_{\C}' \rightarrow \bigwedge^3 \mathfrak{g}_{\C}/\mathfrak{k}_{\C}' \rightarrow \bigwedge^2 \mathfrak{p}^+ \otimes_{\C} \mathfrak{p}^- \rightarrow \tau_{(3,-1)}
$$
is zero. The weight vectors
$$
X_{(2,0)} \wedge X_{(1,1)} \otimes X_{(0,-2)},\, X_{(2,0)} \wedge X_{(1,1)} \otimes X_{(-1,-1)},\, X_{(2,0)} \wedge X_{(1,1)} \otimes X_{(-2,0)},
$$
$$
X_{(2,0)} \wedge X_{(0,2)} \otimes X_{(0,-2)},\, X_{(2,0)} \wedge X_{(0,2)} \otimes X_{(-1,-1)},\, X_{(2,0)} \wedge X_{(0,2)} \otimes X_{(-2,0)},
$$
$$
X_{(1,1)} \wedge X_{(0,2)} \otimes X_{(0,-2)},\, X_{(1,1)} \wedge X_{(0,2)} \otimes X_{(-1,-1)},\, X_{(1,1)} \wedge X_{(0,2)} \otimes X_{(-2,0)}
$$
form a basis of $\bigwedge^2 \mathfrak{p}^+ \otimes_{\C} \mathfrak{p}^-$. Computing the signatures, we have
$$
({}^\sigma \omega_{\varphi_\infty, v}( {}^\sigma \varphi_f) \wedge \overline{{}^\sigma \omega_{\varphi_\infty, v}( {}^\sigma \varphi_f)})(\textbf{1})
$$
\begin{align*}
& = \frac{1}{3!^2} \Big(  {}^\sigma \omega_{\varphi_\infty, v}( {}^\sigma \varphi_f)(X_{(2,0)} \wedge X_{(1,1)} \otimes X_{(0,-2)}) \otimes \overline{{}^\sigma \omega_{\varphi_\infty, v}( {}^\sigma \varphi_f)}(X_{(0,2)} \otimes X_{(-1,-1)} \wedge X_{(-2,0)})\\
& - {}^\sigma \omega_{\varphi_\infty, v}( {}^\sigma \varphi_f)(X_{(2,0)} \wedge X_{(1,1)} \otimes X_{(-1,-1)}) \otimes \overline{{}^\sigma \omega_{\varphi_\infty, v}( {}^\sigma \varphi_f)}(X_{(0,2)} \otimes X_{(0,-2)} \wedge X_{(-2,0)})\\
& + {}^\sigma \omega_{\varphi_\infty, v}( {}^\sigma \varphi_f)(X_{(2,0)} \wedge X_{(1,1)} \otimes X_{(-2,0)}) \otimes \overline{{}^\sigma \omega_{\varphi_\infty, v}( {}^\sigma \varphi_f)}(X_{(0,2)} \otimes X_{(0,-2)} \wedge X_{(-1,-1)})\\
& - {}^\sigma \omega_{\varphi_\infty, v}( {}^\sigma \varphi_f)(X_{(2,0)} \wedge X_{(0,2)} \otimes X_{(0,-2)}) \otimes \overline{{}^\sigma \omega_{\varphi_\infty, v}( {}^\sigma \varphi_f)}(X_{(1,1)} \otimes X_{(-1,-1)} \wedge X_{(-2,0)})\\
& - {}^\sigma \omega_{\varphi_\infty, v}( {}^\sigma \varphi_f)(X_{(2,0)} \wedge X_{(0,2)} \otimes X_{(-1,-1)}) \otimes \overline{{}^\sigma \omega_{\varphi_\infty, v}( {}^\sigma \varphi_f)}(X_{(1,1)} \otimes X_{(0,-2)} \wedge X_{(-2,0)})\\
& - {}^\sigma \omega_{\varphi_\infty, v}( {}^\sigma \varphi_f)(X_{(2,0)} \wedge X_{(0,2)} \otimes X_{(-2, 0)}) \otimes \overline{{}^\sigma \omega_{\varphi_\infty, v}( {}^\sigma \varphi_f)}(X_{(1,1)} \otimes X_{(0,-2)} \wedge X_{(-1,-1)})\\
& +  {}^\sigma \omega_{\varphi_\infty, v}( {}^\sigma \varphi_f)(X_{(1,1)} \wedge X_{(0,2)} \otimes X_{(0, -2)}) \otimes \overline{{}^\sigma \omega_{\varphi_\infty, v}( {}^\sigma \varphi_f)}(X_{(2,0)} \otimes X_{(-1,-1)} \wedge X_{(-2,0)})\\
& -  {}^\sigma \omega_{\varphi_\infty, v}( {}^\sigma \varphi_f)(X_{(1,1)} \wedge X_{(0,2)} \otimes X_{(-1, -1)}) \otimes \overline{{}^\sigma \omega_{\varphi_\infty, v}( {}^\sigma \varphi_f)}(X_{(2,0)} \otimes X_{(0,-2)} \wedge X_{(-2,0)})\\
& + {}^\sigma \omega_{\varphi_\infty, v}( {}^\sigma \varphi_f)(X_{(1,1)} \wedge X_{(0,2)} \otimes X_{(-2, 0)}) \otimes \overline{{}^\sigma \omega_{\varphi_\infty, v}( {}^\sigma \varphi_f)}(X_{(2,0)} \otimes X_{(0,-2)} \wedge X_{(-1,-1)}) \Big).
\end{align*}
As a consequence, using that $\overline{\mathfrak{p}^+ \otimes_{\C} \bigwedge^2 \mathfrak{p}^-}=\bigwedge^2 \mathfrak{p}^+ \otimes_{\C} \mathfrak{p}^-$ and $\overline{X}_{(r,s)}=X_{(-r,-s)}$, we have
\begin{align*}
&\,\,\,\,\,\,\,\,\,\,\,\,\,\,\,\,\,\,\,\,\,\,\,\,\,\,\,\,\,\,\,\,\,\,\,\,\,\,\,\,\,\,\,\,\,\,\,\,\,\,\,\,\,\,\,\,\,\,\,\,\,\,\,\,\,\,\,\,\,\,\,\,\,\,\,\,\,\, ({}^\sigma \omega_{\varphi_\infty, v}( {}^\sigma \varphi_f) \wedge \overline{{}^\sigma \omega_{\varphi_\infty, v}( {}^\sigma \varphi_f)})(\textbf{1}) \\
& = \frac{1}{3!^2} \Big( - {}^\sigma \omega_{\varphi_\infty, v}( {}^\sigma \varphi_f)(X_{(2,0)} \wedge X_{(1,1)} \otimes X_{(0,-2)}) \otimes \overline{{}^\sigma \omega_{\varphi_\infty, v}( {}^\sigma \varphi_f)(X_{(2,0)} \wedge X_{(1,1)} \otimes X_{(0,-2)})}\\
& + {}^\sigma \omega_{\varphi_\infty, v}( {}^\sigma \varphi_f)(X_{(2,0)} \wedge X_{(1,1)} \otimes X_{(-1,-1)}) \otimes \overline{{}^\sigma \omega_{\varphi_\infty, v}( {}^\sigma \varphi_f)(X_{(2,0)} \wedge X_{(0,2)} \otimes X_{(0,-2)})}\\
& - {}^\sigma \omega_{\varphi_\infty, v}( {}^\sigma \varphi_f)(X_{(2,0)} \wedge X_{(1,1)} \otimes X_{(-2,0)}) \otimes \overline{{}^\sigma \omega_{\varphi_\infty, v}( {}^\sigma \varphi_f)(X_{(1,1)} \wedge X_{(0,2)} \otimes X_{(0,-2)})}\\
& + {}^\sigma \omega_{\varphi_\infty, v}( {}^\sigma \varphi_f)(X_{(2,0)} \wedge X_{(0,2)} \otimes X_{(0,-2)}) \otimes \overline{{}^\sigma \omega_{\varphi_\infty, v}( {}^\sigma \varphi_f)(X_{(2,0)} \wedge X_{(1,1)} \otimes X_{(-1,-1)})}\\
& + {}^\sigma \omega_{\varphi_\infty, v}( {}^\sigma \varphi_f)(X_{(2,0)} \wedge X_{(0,2)} \otimes X_{(-1,-1)}) \otimes \overline{{}^\sigma \omega_{\varphi_\infty, v}( {}^\sigma \varphi_f)(X_{(2,0)} \wedge X_{(0,2)} \otimes X_{(-1,-1)})}\\
& + {}^\sigma \omega_{\varphi_\infty, v}( {}^\sigma \varphi_f)(X_{(2,0)} \wedge X_{(0,2)} \otimes X_{(-2, 0)}) \otimes \overline{{}^\sigma \omega_{\varphi_\infty, v}( {}^\sigma \varphi_f)(X_{(1,1)} \wedge X_{(0,2)} \otimes X_{(-1,-1)})}\\
& -  {}^\sigma \omega_{\varphi_\infty, v}( {}^\sigma \varphi_f)(X_{(1,1)} \wedge X_{(0,2)} \otimes X_{(0, -2)}) \otimes \overline{{}^\sigma \omega_{\varphi_\infty, v}( {}^\sigma \varphi_f)(X_{(2,0)} \wedge X_{(1,1)} \otimes X_{(-2,0)})}\\
& +  {}^\sigma \omega_{\varphi_\infty, v}( {}^\sigma \varphi_f)(X_{(1,1)} \wedge X_{(0,2)} \otimes X_{(-1, -1)}) \otimes \overline{{}^\sigma \omega_{\varphi_\infty, v}( {}^\sigma \varphi_f)(X_{(2,0)} \wedge X_{(0,2)} \otimes X_{(-2,0)})}\\
& - {}^\sigma \omega_{\varphi_\infty, v}( {}^\sigma \varphi_f)(X_{(1,1)} \wedge X_{(0,2)} \otimes X_{(-2, 0)}) \otimes \overline{{}^\sigma \omega_{\varphi_\infty, v}( {}^\sigma \varphi_f)(X_{(1,1)} \wedge X_{(0,2)} \otimes X_{(-2,0)})}\Big).
\end{align*}
A standard basis of $\tau_{(3,-1)}$ is computed in Lemma \ref{std-basis} and the matrix of the projection $p: \bigwedge^2 \mathfrak{p}^+ \otimes_{\C} \mathfrak{p}^- \rightarrow \tau_{(3,-1)}$ is computed in Lemma \ref{projection-matrix}. From these two results we deduce the following equalities which give the image by $p$ of the basis vectors in terms of the highest weight vector $X_{(2,0)} \wedge X_{(1,1)} \otimes X_{(0,-2)}$ of $\tau_{(3,-1)}$:
\begin{align*}
& p(X_{(2,0)} \wedge X_{(1,1)} \otimes X_{(0,-2)})= X_{(2,0)} \wedge X_{(1,1)} \otimes X_{(0,-2)},\\
& p(X_{(2,0)} \wedge X_{(1,1)} \otimes X_{(-1,-1)})=-\frac{1}{2}\mrm{Ad}_{X_{(-1,1)}}(X_{(2,0)} \wedge X_{(1,1)} \otimes X_{(0,-2)}),\\
& p(X_{(2,0)} \wedge X_{(1,1)} \otimes X_{(-2,0)})=\frac{1}{12} \mrm{Ad}_{X_{(-1,1)}}^2(X_{(2,0)} \wedge X_{(1,1)} \otimes X_{(0,-2)}),\\
& p(X_{(2,0)} \wedge X_{(0,2)} \otimes X_{(0,-2)})=\frac{1}{4}\mrm{Ad}_{X_{(-1,1)}}(X_{(2,0)} \wedge X_{(1,1)} \otimes X_{(0,-2)}),\\
& p(X_{(2,0)} \wedge X_{(0,2)} \otimes X_{(-1,-1)})=-\frac{1}{6}\mrm{Ad}_{X_{(-1,1)}}^2(X_{(2,0)} \wedge X_{(1,1)} \otimes X_{(0,-2)}),\\
& p(X_{(2,0)} \wedge X_{(0,2)} \otimes X_{(-2, 0)}) = \frac{1}{24}\mrm{Ad}_{X_{(-1,1)}}^3(X_{(2,0)} \wedge X_{(1,1)} \
\otimes X_{(0,-2)}),\\
& p(X_{(1,1)} \wedge X_{(0,2)} \otimes X_{(0, -2)})=\frac{1}{12} \mrm{Ad}_{X_{(-1,1)}}^2(X_{(2,0)} \wedge X_{(1,1)} \otimes X_{(0,-2)}),\\
& p(X_{(1,1)} \wedge X_{(0,2)} \otimes X_{(-1, -1)})=-\frac{1}{12}\mrm{Ad}_{X_{(-1,1)}}^3(X_{(2,0)} \wedge X_{(1,1)} \
\otimes X_{(0,-2)}),\\
& p(X_{(1,1)} \wedge X_{(0,2)} \otimes X_{(-2, 0)})=\frac{1}{24}\mrm{Ad}_{X_{(-1,1)}}^4(X_{(2,0)} \wedge X_{(1,1)} \
\otimes X_{(0,-2)}).
\end{align*}
Using these equalities, we find
$$
({}^\sigma \omega_{\varphi_\infty, v}( {}^\sigma \varphi_f) \wedge \overline{{}^\sigma \omega_{\varphi_\infty, v}( {}^\sigma \varphi_f)})(\textbf{1})
$$
\begin{align*}
& = \frac{1}{3!^2}\sum_{0 \leq i,j \leq k+k'}  (-1)^{i+j} \Big\{-  \Big( X_{(1, -1)}^i v \otimes {X}_{(1, -1)}^{k+k'+4-i}\, {}^\sigma\!\varphi \Big) \otimes  \Big(  X_{(-1, 1)}^j \overline{v} \otimes {X}_{(-1, 1)}^{k+k'+4-j}\, \overline{{}^\sigma\!\varphi} \Big)\\
& - \frac{1}{4} \mrm{Ad}_{X_{(-1,1)}} \Big( X_{(1, -1)}^i v \otimes {X}_{(1, -1)}^{k+k'+4-i}\, {}^\sigma\!\varphi \Big) \otimes \mrm{Ad}_{X_{(1,-1)}}\Big( X_{(-1, 1)}^j \overline{v} \otimes {X}_{(-1, 1)}^{k+k'+4-j}\, \overline{{}^\sigma\!\varphi} \Big)\\
& + \frac{1}{72} \mrm{Ad}^2_{X_{(-1,1)}}\Big( X_{(1, -1)}^i v \otimes {X}_{(1, -1)}^{k+k'+4-i}\, {}^\sigma\!\varphi \Big) \otimes \mrm{Ad}^2_{X_{(1,-1)}}\Big(  X_{(-1, 1)}^j \overline{v} \otimes {X}_{(-1, 1)}^{k+k'+4-j}\, \overline{{}^\sigma\!\varphi} \Big)\\
& - \frac{1}{72} \mrm{Ad}^3_{X_{(-1,1)}}\Big(   X_{(1, -1)}^i v \otimes {X}_{(1, -1)}^{k+k'+4-i}\, {}^\sigma\!\varphi \Big) \otimes \mrm{Ad}^3_{X_{(1,-1)}}\Big(  X_{(-1, 1)}^j \overline{v} \otimes {X}_{(-1, 1)}^{k+k'+4-j}\, \overline{{}^\sigma\!\varphi} \Big)\\
& - \frac{1}{576} \mrm{Ad}^4_{X_{(-1,1)}}\Big(   X_{(1, -1)}^i v \otimes {X}_{(1, -1)}^{k+k'+4-i}\, {}^\sigma\!\varphi \Big) \otimes \mrm{Ad}^4_{X_{(1,-1)}}\Big(  X_{(-1, 1)}^j \overline{v} \otimes {X}_{(-1, 1)}^{k+k'+4-j}\, \overline{{}^\sigma\!\varphi} \Big)\Big\}.
\end{align*}
We have
$$
\mrm{Ad}^r_{X_{(-1,1)}}\Big( X_{(1, -1)}^i v \otimes {X}_{(1, -1)}^{k+k'+4-i}\, {}^\sigma\!\varphi \Big)=  \sum_{u=0}^r \binom{r}{u} r_{i, u, r}^{k,k'} X_{(1, -1)}^{i-u} v \otimes X_{(1, -1)}^{k+k'+4-i-(r-u)} \, {}^\sigma\!\varphi,
$$
where we use the convention that $X_{(1, -1)}^{i-u} v=0$ if $i-u<0$, and similarly
$$
\mrm{Ad}^r_{X_{(1,-1)}}\Big(  X_{(-1, 1)}^j \overline{v} \otimes {X}_{(-1, 1)}^{k+k'+4-j}\, \overline{{}^\sigma\!\varphi} \Big)=\sum_{u=0}^r \binom{r}{u}r_{j, u, r}^{k,k'} X_{(-1, 1)}^{j-u} \overline{v} \otimes X_{(-1, 1)}^{k+k'+4-j-(r-u)} \, \overline{{}^\sigma\!\varphi}.
$$
with the convention that $X_{(-1, 1)}^{j-u} \overline{v}=0$ if $j-u<0$. Then we have
\begin{align*}
& ({}^\sigma \omega_{\varphi_\infty, v}( {}^\sigma \varphi_f) \wedge \overline{{}^\sigma \omega_{\varphi_\infty, v}( {}^\sigma \varphi_f)})(\textbf{1})= \frac{1}{3!^2}\sum_{r=0}^4 \sum_{0 \leq i,j \leq k+k'}\sum_{\substack{0 \leq u, u' \leq r \\  0 \leq i-u \\ 0 \leq j-u'}} \\
&   (-1)^{i+j} a_r  \binom{r}{u} \binom{r}{u'} r_{i, u, r}^{k,k'} r_{j, u', r}^{k,k'}   (X_{(1, -1)}^{i-u} v \otimes X_{(1, -1)}^{k+k'+4-i-(r-u)} \, {}^\sigma\!\varphi) \otimes (X_{(-1, 1)}^{j-u'} \overline{v} \otimes X_{(-1, 1)}^{k+k'+4-j-(r-u')} \, \overline{{}^\sigma\!\varphi}).
\end{align*}
Let us denote by $[({}^\sigma \omega_{\varphi_\infty, v}( {}^\sigma \varphi_f) \wedge \overline{{}^\sigma \omega_{\varphi_\infty, v}( {}^\sigma \varphi_f)})(\textbf{1})]$ the element of ${}^\sigma \Pi \otimes {}^\sigma \overline{\Pi}$ defined as the image of $({}^\sigma \omega_{\varphi_\infty, v}( {}^\sigma \varphi_f) \wedge \overline{{}^\sigma \omega_{\varphi_\infty, v}( {}^\sigma \varphi_f)})(\textbf{1})$ by the map induced by the pairing $[\,,\,]_{\C}$. Then, thanks to Lemma \ref{pairing-computation}, we have

\begin{align*}
& [({}^\sigma \omega_{\varphi_\infty, v}( {}^\sigma \varphi_f) \wedge \overline{{}^\sigma \omega_{\varphi_\infty, v}( {}^\sigma \varphi_f)})(\textbf{1})]= \frac{(k+k')!}{3!^2}\sum_{r=0}^4 \sum_{i=0}^{k+k'}\sum_{\substack{0 \leq u, u' \leq r \\  0 \leq i-u}} \\
&  (-1)^{j-u} a_r   \binom{r}{u} \binom{r}{u'} r_{i, u, r}^{k,k'} r_{i-u+u', u', r}^{k,k'} s_{i, u}^{k,k'} X_{(1, -1)}^{k+k'+4-i-(r-u)} \, {}^\sigma\!\varphi \otimes X_{(-1, 1)}^{k+k'+4-i-(r-u)} \, \overline{{}^\sigma\!\varphi}.
\end{align*}
As the pairing ${}^\sigma \Pi^{K_N} \otimes {}^\sigma \overline{\Pi}^{K_N} \rightarrow \C$ defined as $$\varphi \otimes \psi \mapsto \int_{\R_+^\times G(\Q) \backslash G(\A)/K_N} \varphi(g) {\psi}(g) dg$$ is a morphism of $(\mathfrak{g}_{\C}, K_\infty)$-modules we have
\begin{multline*}
\int_{\R_+^\times G(\Q) \backslash G(\A)/K_N} X_{(1,-1)}^i \varphi(g) X_{(-1,1)}^j \overline{\varphi}(g) dg
\\ 
=\left\{
    \begin{array}{ll}
        0 & \mbox{ if } i \neq j \\
        (-1)^i \frac{i! (k+k'+4)!}{(k+k'+4-i)!} \int_{\R_+^\times G(\Q)  \backslash G(\A)/K_N} \varphi(g) \overline{\varphi}(g) dg & \mbox{ if } i = j.
    \end{array}
\right.
\end{multline*}
As a consequence 
\begin{align*}
&  \langle {}^\sigma \omega_{\varphi_\infty, v}( {}^\sigma \varphi_f), \overline{{}^\sigma \omega_{\varphi_\infty, v}( {}^\sigma \varphi_f)} \rangle_{\C}\\
&=\frac{(-1)^{k+k'}(k+k')! (k+k'+4)!}{3!^2} \sum_{r=0}^4  \sum_{i=0}^{k+k'} \sum_{\substack{0 \leq u, u' \leq r \\  0 \leq i-u}}\\
& (-1)^{r+u+u'}a_r \binom{r}{u} \binom{r}{u'} r^{k,k'}_{i, u, r} r^{k,k'}_{i-u+u', u', r} s_{i, u}^{k, k'} t_{i, u, r}^{k, k'} \int_{\R_+^\times G(\Q)  \backslash G(\A)/K_N} {}^\sigma \varphi(g) \overline{{}^\sigma\varphi}(g) dg
\end{align*}
We have
\begin{eqnarray*}
\int_{\R_+^\times G(\Q)  \backslash G(\A)/K_N} |{}^\sigma \varphi(g)|^2 dg &=& \mrm{vol}(K_N, dg)^{-1} \int_{\R_+^\times G(\Q)  \backslash G(\A)}|{{}^\sigma\varphi}(g)|^2 dg\\
&=& 2 \mrm{vol}(K_N, dg)^{-1} \int_{Z(\A) G(\Q)  \backslash G(\A)} |{{}^\sigma \varphi}(g)|^2 dg \\
&=& 2 \mrm{vol}(K_N, dg)^{-1} \int_{Z(\A) G(\Q)  \backslash G(\A)} |{{}^\sigma \varphi}(g)|^2 dg\\
&=& \frac{\pi^3}{135} \prod_{l | N} (l^2+1)^{-1} \langle {}^\sigma\varphi, {}^\sigma\varphi \rangle\\
\end{eqnarray*}
where the last equality follows from Lemma \ref{std-tamagawa}, the definition of $\langle {}^\sigma\varphi, {}^\sigma\varphi \rangle$ and \cite[Lemma 3.3.3]{roberts-schmidt}. 
\end{proof}
Namikawa indicated the authors a conjecture on the simplification of the factor $C'_{k,k'}$ by 
the constant $C_{k,k'}$ below according to some numerical calculations and Yasuda pointed out the following proof based on 
a formula on Pochhammer symbols \eqref{equation:Pochhammersymbol} stated below. We thank to Namikawa for the conjecture and to Yasuda for indicating us the formula and its proof.  
\begin{lem}
The constant $C'_{k,,k'}$ is equal to the constant $C_{k,k'}$ defined as follows: 
$$
C_{k,k'} = \frac{(-1)^{k+k'}(k+k' +4)! (k+k' +5)!}{3^3 \cdot 5}. 
$$ 
\end{lem}
\begin{proof}
Let us recall the Pochhammer symbol as follows: 
$$
(x)_n = x(x+1) (x+2) \cdots (x+n-1) = \frac{(x+n-1)!}{(x-1)!} 
$$
where $x$ is a non-negative integer and $n $ is a natural number.
First, we can check that  
\begin{equation}
r^{k,k'}_{i, u, r} r^{k,k'}_{i-u+u', u', r} s_{i, u}^{k, k'} t_{i, u, r}^{k, k'}
= \frac{(k+k'-i+1)_4 (k+k'-i+u-u'+1)_4 (i-u +1)_r}{(k+k'-i+u+1)_{4-r}}. 
\end{equation}
The right-hand side of \eqref{equation:definitionC'_{k,k'}} is equal to 
\begin{multline}\label{equation:translation_of_C'_{k,k'}1}
\frac{(-1)^{k+k'}(k+k')! (k+k'+4)!}{3!^2} \sum_{r=0}^4  \sum_{i=0}^{k+k'} \sum^r_{u=0} \sum^r_{u'=0}
\\ 
(-1)^{r+u+u'}a_r \binom{r}{u} \binom{r}{u'} \frac{(k+k'-i+1)_4 (k+k'-i+u-u'+1)_4 (i-u +1)_r}{(k+k'-i+u+1)_{4-r}}. 
\end{multline}
Here we note that the sum $\displaystyle{\sum_{\substack{0 \leq u, u' \leq r \\  0 \leq i-u}}}$ in the original expression 
was replaced by $\displaystyle{\sum^r_{u=0} \sum^r_{u'=0}}$ since the symbol $(i-u +1)_r$ in the numerator 
is zero when $i-u<0$. 
Now we recall that the following formula holds for any integer $b$ and any integers $l$, $m$ satisfying 
$0\leq l \leq m$: 
\begin{equation}\label{equation:Pochhammersymbol}
\sum^l_{a=0} (-1)^a \binom{l}{a}(b-a)_{m} = \binom{m}{l} l! (b)_{m-l} . 
\end{equation}
In fact, we can show this formula on the induction wit respect to $l$ for the value 
$\mathrm{LHS}(b,l,m)$ (resp. $\mathrm{RHS}(b,l,m)$) on the left-hand side (resp. the right-hand side). 
For $l=0$, the equality $\mathrm{LHS}(b,l,m)=\mathrm{RHS}(b,l,m)$ is trivially true. The induction argument 
for $l>0$ proceeds since we have $\mathrm{LHS}(b,l,m) = \mathrm{LHS}(b,l-1,m) - \mathrm{LHS}(b-1,l-1,m) $ 
and $\mathrm{RHS}(b,l,m) = \mathrm{RHS}(b,l-1,m) - \mathrm{RHS}(b-1,l-1,m) $. 
\par 
By applying \eqref{equation:Pochhammersymbol} with $a=u'$, $b=k+k'-i+u+1$, $l=r$ and $m=4$, the expression 
\eqref{equation:translation_of_C'_{k,k'}1} is simplified as follows:  
\begin{multline}\label{equation:translation_of_C'_{k,k'}2}
\frac{(-1)^{k+k'}(k+k')! (k+k'+4)!}{3!^2} \sum_{r=0}^4  \sum_{i=0}^{k+k'} \sum^r_{u=0} 
\\ 
(-1)^{r+u}a_r \binom{r}{u} \binom{4}{r} r! (k+k'-i+1)_4 (i-u +1)_r . 
\end{multline}
By applying \eqref{equation:Pochhammersymbol} again with $a=u$, $b=i+1$, $l=m=r$, the expression 
\eqref{equation:translation_of_C'_{k,k'}2} is further simplified as follows: 
\begin{equation}\label{equation:translation_of_C'_{k,k'}3}
\frac{(-1)^{k+k'}(k+k')! (k+k'+4)!}{3!^2} \sum_{r=0}^4  \sum_{i=0}^{k+k'} 
(-1)^{r}a_r  \binom{4}{r} (r!)^2 (k+k'-i+1)_4 . 
\end{equation}
Since we have 
$\displaystyle{ \sum_{r=0}^4  
(-1)^{r}a_r  \binom{4}{r} (r!)^2 =\frac{4}{3}}$, 
the expression 
\eqref{equation:translation_of_C'_{k,k'}3} is equal to: 
\begin{equation}\label{equation:translation_of_C'_{k,k'}4}
\frac{(-1)^{k+k'}(k+k')! (k+k'+4)!}{3!^2} \frac{4}{3}  \sum_{i'=0}^{k+k'}  (i'+1)_4 .  
\end{equation}
by putting $i'=k+k'-i$. Finally, by noting that 
\begin{equation*}
\sum_{i'=0}^{k+k'}  (i'+1)_4 = \frac{1}{5} \sum_{i'=0}^{k+k'}  \left( (i'+1)_5 -(i')_5 \right) = \frac{(k+k'+1)_5}{5},  
\end{equation*}
the expression \eqref{equation:translation_of_C'_{k,k'}4} is equal to: 
\begin{equation}\label{equation:translation_of_C'_{k,k'}5}
\frac{(-1)^{k+k'}(k+k'+4)! (k+k'+5)!}{3^3 \cdot 5}  .  
\end{equation}
This completes the proof. 
\end{proof}
Recall that we have fixed a basis $(\delta_1, \ldots, \delta_{2r})$ of $L(\Pi_f, V_{\lambda, \Z_{(p)}})$.

\begin{defn} The discriminant $d(\Pi_f)$ of the pairing \eqref{pairing2} is defined as follows:
$$
d(\Pi_f)= \det ( \langle \delta_i, \delta_j \rangle_{\Z_{(p)}} )_{1 \leq i, j \leq 2r}.
$$
It is an element of $\Z_{(p)}$ whose image in $\Z_{(p)}/(\Z_{(p)}^\times)^2$ is independent of the choice of the $\Z_{(p)}$-basis $(\delta_i)_{1 \leq i \leq 2r}$ of $L(\Pi_f, V_{\lambda, \Z_{(p)}})$.
\end{defn}

\begin{defn} For $x, y \in \R$ we write $x \sim y$ if there exists $s \in (\Z_{(p)}^\times)^2$ such that $x=sy$.
\end{defn}
Recall that $C_N =\prod_{l\vert N} (l+l^{-1})^{-1}(l^2 +1)^{-1}$.
\begin{thm} \label{main1} Assume that $p \notin S_{N, 3} \cup S_{\mathrm{weight}} \cup S_{\mathrm{tors}} \cup S'_{\mathrm{tors}}$ where where $S_{N, 3}$, $S_{\mathrm{tors}}$, $S'_{\mathrm{tors}}$, $S_{\mathrm{weight}}$ are defined by \eqref{s_G}, \eqref{s-tors}, \eqref{s'-tors} and \eqref{s-weight} respectively. Then we have:
$$
d(\Pi_f) \sim  \left( \left( \frac{2^{k+k'+13}C_{k,k'}C_N \pi^{3k+k'+12}}{k+k'+5} \right)^r \Omega(\Pi_f)^{-1} \prod_{\sigma: E(\Pi_f) \rightarrow \C} L(1, \,\!^\sigma \Pi, \mathrm{Ad}) \right)^2 .
$$
\end{thm}

\begin{proof}  Recall that $(\omega_1, \ldots, \omega_{2r})$ denotes the $\R$-basis of $L(\Pi_f, V_{\lambda, \R})$ normalized in the second assertion of Lemma \ref{norm-basis}. Let $T$ denote the matrix $T=( \langle \delta_i, \delta_j \rangle_{\Z_{(p)}} )_{1 \leq i, j \leq 2r}$ and let $S$ denote the matrix $S=( \langle \omega_i, \omega_j \rangle_{\R} )_{1 \leq i, j \leq 2r}$. We have
$$
d(\Pi_f) \sim \det T\sim \Omega(\Pi_f)^{-2} \det S
$$
where the first equality is the definition of the discriminant, and the second from the definition of $\Omega(\Pi_f)$ (see \eqref{period}). Hence we have to compute the pairings $\langle \omega_i, \omega_j \rangle_{\R}$. By commutativity of the diagram \eqref{cd-poincare} we have $\langle \omega_i, \omega_j \rangle_{\R}=\langle \omega_i, \omega_j \rangle_{\C}$. Because of the vanishing of the Petersson inner product, we have $\langle \omega_i, \omega_j \rangle_{\C}=0$ for any $1 \leq i,j \leq 2r$ such that $j \neq i$, $j \neq i+r$ and $i \neq j+r$. Furthermore, as the Poincar\'e duality pairing is a morphism of Hodge structures, for any $1 \leq i \leq r$ we have 
$$
\langle \,\! ^{\sigma_i} r_\infty(^{\sigma_i} \varphi_f),\,\! ^{\sigma_i} r_\infty(^{\sigma_i} \varphi_f) \rangle_{\C}=\langle \overline{\,\! ^{\sigma_i}r_\infty(^{\sigma_i}\varphi_f)}, \overline{\,\! ^{\sigma_i}r_\infty(^{\sigma_i}\varphi_f)} \rangle_{\C}=0.
$$
As a consequence, for $1 \leq i \leq r$ we have
\begin{eqnarray*}
\langle \omega_i, \omega_i \rangle_{\C} &=& \left\{
    \begin{array}{ll}
        2^{-1} \langle \,\! ^{\sigma_i}r_\infty(^{\sigma_i}\varphi_f), \overline{\,\! ^{\sigma_i}r_\infty(^{\sigma_i}\varphi_f)} \rangle_{\C} & \mbox{if } \epsilon_\lambda \mbox{ is even,}\\
         0 & \mbox{if } \epsilon_\lambda \mbox{ is odd}, \\ 
    \end{array}
\right.\\
\langle \omega_i, \omega_{i+r} \rangle_{\C} &=& \left\{
    \begin{array}{ll}
         0 & \mbox{if } \epsilon_\lambda \mbox{ is even,}\\
         -2^{-1}\sqrt{-1} \langle \,\! ^{\sigma_i}r_\infty(^{\sigma_i}\varphi_f), \overline{\,\! ^{\sigma_i}r_\infty(^{\sigma_i}\varphi_f)} \rangle_{\C} & \mbox{if } \epsilon_\lambda \mbox{ is odd},\\ 
    \end{array}
\right.\\
\langle \omega_{i+r}, \omega_{i+r} \rangle_{\C} &=& \left\{
    \begin{array}{ll}
         2^{-1} \langle \,\! ^{\sigma_i}r_\infty(^{\sigma_i}\varphi_f), \overline{\,\! ^{\sigma_i}r_\infty(^{\sigma_i}\varphi_f)} \rangle_{\C} & \mbox{if } \epsilon_\lambda \mbox{ is even,}\\
         0 & \mbox{if } \epsilon_\lambda \mbox{ is odd}. \\ 
    \end{array}
\right.\\
\end{eqnarray*}
The statement now follows from Proposition \ref{pd-computation} and Theorem \ref{ichino-main}.
\end{proof}

\section{The congruence criterion}\label{section:The congruence criterion}

Let $\overline{\mathcal{H}}^N_{\mathrm{sph}} \subset \overline{\mathcal{H}}^{K_N}$ denote the spherical Hecke algebra outside $N$ over $\Z$. Let $\Pi' \simeq \Pi'_\infty \otimes \Pi'_f$ be a cuspidal representation which contributes to $\tilde{H}^3_!(S_{K_N}, V_{\lambda, \C})$. This means that $(\Pi'_f)^{K_N}$ is non-zero (but not necessarily one-dimensional) and $\Pi'_\infty \in P(V_{\lambda, \C})$. Let $E(\Pi'_f)$ be the rationality field of $\Pi'_f$ and let $\theta_{\Pi'}: \overline{\mathcal{H}}^{N}_{\mathrm{sph}} \rightarrow E(\Pi'_f)$ be the character such that for any $g \in \bigotimes'_{v \nmid N\infty} (\Pi'_v)^{G(\Z_v)}$ and any $h \in \overline{\mathcal{H}}^{N}_{\mathrm{sph}}$ we have $h g=\theta_{\Pi'}(h)g$. Let $\mathcal{O}_{E(\Pi'_f)}$ be the ring of integers of $E(\Pi'_f)$. According to Corollary \ref{order} and to the fact that $\overline{\mathcal{H}}^{K_N}$ is a $\Z$-module of finite type we have $\im \theta_{\Pi'} \subset \mathcal{O}_{E(\Pi'_f)}$.

\begin{defn} \label{congruence-def} Let $\mathfrak{P}$ be a prime ideal of $\overline{\Q}$. The cuspidal representation $\Pi'$ is congruent to $\Pi$ modulo $\mathfrak{P}$ if there exists a number field $E$ containing $E(\Pi'_f)$ and $E(\Pi_f)$, with ring of integers $\mathcal{O}_E$ such that the following diagram commutes
\begin{equation*}\label{equation:maindiagram0}
\xymatrix{  
& \ \ \mathcal{O}_{E(\Pi_f)} \ar@{^{(}->}[r]  & \mathcal{O}_E \ar[rd] &\\ 
\overline{\mathcal{H}}^N_{\mathrm{sph}}
\ar[ru]_{\theta_{\Pi}} 
\ar[rd]_{\theta_{\Pi'}}
 & & & \kappa
\\
& 
\ \ \mathcal{O}_{E(\Pi'_f)}  \ar@{^{(}->}[r]  & \mathcal{O}_E \ar[ru] &
 \\
} 
\end{equation*}
where $\kappa=\mathcal{O}_E/\mathfrak{P} \cap \mathcal{O}_E$. In this case we write $\Pi' \equiv \Pi \pmod{\mathfrak{P}}$.
\end{defn}

Recall that $\Pi$ is a cuspidal representation of $G(\A)$, which, among other things is assumed to have trivial central character, to be globally generic and endoscopic. This means that the functorial lift $\Sigma$ of $\Pi$ to $\mrm{GL}(4, \A)$ is not cuspidal. In particular, according to \cite[Proposition 2.2]{asgari-shahidi}, there exists $\sigma_1$ and $\sigma_2$ two inequivalent unitary cuspidal automorphic representations of $\mrm{GL}(2, \A)$ with trivial central characters such that $\Sigma$ is the isobaric sum $\Sigma=\sigma_1 \boxplus \sigma_2$ and such that $\Pi$ is obtained as a Weil lifting from $(\sigma_1, \sigma_2)$. According to \cite[Corollary 4.2]{weissauer2} the cuspidal representations $\sigma_1$ and $\sigma_2$ correspond to primitive cuspforms $f_1$ and $f_2$ of respective weights $k_1=k+k'+4$ and $k_2=k-k'+2$ and of respective levels $N_1$ and $N_2$. For $i=1,2$ let $c(f_i)$ denote Ghate's congruence number attached to $f_i$ as defined in the introduction. We say that the cuspidal automorphic representation attached to a holomorphic Siegel modular form $F$ is stable  if $F$ is not a Saito-Kurokawa lift and its functorial lift to $\mrm{GL}(4, \A)$ (see \cite{saha} for its construction) is cuspidal. 

\begin{thm} \label{main} 
Let us assume that the conductor $N$ of the automorphic representation $\Pi \simeq \Pi_\infty \otimes \Pi_f$ is prime to $p$ and that $p \notin S_{N,3} \cup S_{\mathrm{weight}} \cup S_{\mathrm{tors}} \cup S'_{\mathrm{tors}} \cup S''_{\mathrm{tors}}$ where $S_{N,3}$, $S_{\mathrm{tors}}$, $S'_{\mathrm{tors}}$, $S_{\mathrm{weight}}$ and $S''_{\mathrm{tors}}$ are defined by \eqref{s_G}, \eqref{s-tors}, \eqref{s'-tors}, \eqref{s-weight} and \eqref{s''-tors} respectively. Assume the following conditions:
\begin{enumerate}
\item[\rm{(a)}] the residual $\mrm{Gal}(\overline{\Q}/\Q)$-representations $\overline{\rho}_{f_1}$ and $\overline{\rho}_{f_1}$ of $f_1$ and $f_2$ are irreducible,
\item[\rm{(b)}] the prime $\mathfrak{p}$ does not divide $c'(f_1)$ nor $c'(f_2)$ for any prime $\mathfrak{p}$ above $p$ in $\overline{\mathbb{Q}}$,
\item[\rm{(c)}] the prime $p$ divides $$
\left( \left( \frac{2^{k+k'+13}C_{k,k'}C_N\pi^{3k+k'+12}}{k+k'+5} \right)^r \Omega(\Pi_f)^{-1} \prod_{\sigma: E(\Pi_f) \rightarrow \C} L(1, \,\!^\sigma \Pi, \mathrm{Ad}) \right)^2, 
$$
\end{enumerate}
Then, there exists a prime divisor $\mathfrak{P}$ of $p$ in $\overline{\Q}$ and a cuspidal representation $\Pi' \simeq \bigotimes_v' \Pi'_v$ of $G(\A)$ such that 
\begin{enumerate}
\item[\rm{(1)}] the non archimedean part $\Pi'_f$ of $\Pi'$ satisfies $(\Pi'_f)^{K_N} \neq 0$,
\item[\rm{(2)}] we have $\Pi'_\infty \in P(V_{\lambda, \C})$,
\item[\rm{(3)}] the cuspidal representation $\Pi'$ is stable,
\item[\rm{(4)}] we have $\Pi' \not\simeq \,^\sigma\!\,\Pi$ for all $\sigma \in \mrm{Aut}(\C)$,
\item[\rm{(5)}] we have $\Pi' \equiv \Pi \pmod{\mathfrak{P}}$.
\end{enumerate}
\end{thm}

For the proof of this theorem, we need the following result.

\begin{lem} \label{lem-hida} Let $\langle \,,\,\rangle_{\Q}$ denote the bilinear form deduced from \eqref{pairing2} after base change from $\Z_{(p)}$ to $\Q$. Let
$$
L(\Pi_f, V_{\lambda, \Z_{(p)}})^*=\{x \in L(\Pi_f, V_{\lambda, \Q})\,|\, \forall y \in L(\Pi_f, V_{\lambda, \Z_{(p)}}), \langle x, y \rangle_{\Q} \in \Z_{(p)}\}.
$$ be the lattice dual of $L(\Pi_f, V_{\lambda, \Z_{(p)}})$. Then as lattices in $L(\Pi_f, V_{\lambda, \Q})$, we have
$$
M(\Pi_f, V_{\lambda, \Z_{(p)}})=L(\Pi_f, V_{\lambda, \Z_{(p)}})^*.
$$
\end{lem}

\begin{proof} We would like to apply \cite[(4.6)]{hida} so we verify the assumptions of this result. According to Corollary \ref{nd-inv-bilinear}, the $\Z_{(p)}$-lattice $L=\tilde{H}^3_!(S_{K_N}, {V}_{\lambda, \Z_{(p)}})$ of the $\Q$-vector space $V=\tilde{H}^3_!(S_{K_N}, {V}_{\lambda, \Q})$ is self-dual for the bilinear form deduced from \eqref{inv-bilinear} by base change from $\Z_{(p)}$ to $\Q$, that we also denote by $\langle\,,\,\rangle_{\Q}$. Let $W_1$ denote the subspace $W_1=M(\Pi_f, V_{\lambda, \Q})$ of $V$ and let $W_2$ denote the kernel of the projection $p_M: V \rightarrow W_1$ \eqref{projection}. Then $V=W_1 \oplus W_2$. By the proof of Lemma \ref{order}, there exists $e \in \overline{\mathcal{H}}^{K_N}_{\Q}$ an idempotent such that $W_1=eV$. Then $W_2 = (1-e)V$. We claim that $W_1$ is orthogonal to $W_2$ for $\langle \,,\,\rangle$. In order to prove this, it is enough to prove that $W_{1} \otimes_{\Q} \C$ and $W_2 \otimes_{\Q} \C$ are orthogonal for $\langle \,,\,\rangle_{\C}$. For any $v,w \in V \otimes_{\Q} \C$, we have
$
\langle ev, (1-e)w \rangle_{\C} = \langle v, e^*(1-e)w \rangle_{\C}=0
$
where the first equality follows from Proposition \ref{adjoint-hecke} and Proposition \ref{pp}. As $e$ is the projector on the $\Pi_f$ isotypical component, $e^*$ is the projector on the $\check{\Pi}_f$ isotypical component, where $\check{\Pi}_f$ denoted the contragredient of $\Pi_f$. But as $\Pi_f$ has trivial central character, we have $\Pi_f \simeq \check{\Pi}_f$ and so $e^*=e$. As a consequence $
\langle ev, (1-e)w \rangle_{\C}=0$. By definition, we have $L(\Pi_f, V_{\lambda, \Z_{(p)}})=L\cap W_1$ and $M(\Pi_f, V_{\lambda, \Z_{(p)}})$ is the projection of $L$ on $W_1$ along $W_2$. Hence the statement of the lemma follows from \cite{hida} (4.6). 
\end{proof}

\begin{proof}[Proof of Theorem \ref{main}] Recall that, for any $\Z_{(p)}$-algebra $R$, we denote by $\overline{\mathcal{H}}^{K_N}_R$ the algebra $\overline{\mathcal{H}}^{K_N}_{\Z_{(p)}} \otimes_{\Z_{(p)}} R$ where $\overline{\mathcal{H}}^{K_N}_{\Z_{(p)}}$ denotes the image of the abstract Hecke algebra $\mathcal{H}^{K_N}_{\Z_{(p)}}$ with $\Z_{(p)}$-coefficients in $\mrm{End}_{\Z_{(p)}}(\tilde{H}^3_!(S_{K_N}, {V}_{\lambda, \Z_{(p)}}))$.  Let $Y$ denote the orthogonal complement of $L(\Pi_f, V_{\lambda, \R})$ in $\tilde{H}^3_!(S_{K_N}, {V}_{\lambda, \R})$. As $L(\Pi_f, V_{\lambda, \R})$ is stable by the action of $\overline{\mathcal{H}}^{K_N}_{\R}$, it follows from Proposition \ref{adjoint-hecke} that $Y$ is also stable by $\overline{\mathcal{H}}^{K_N}_{\R}$ so that the decomposition
$
L(\Pi_f, V_{\lambda, \R}) \oplus Y= \tilde{H}^3_!(S_{K_N}, {V}_{\lambda, \R})
$
holds as $\overline{\mathcal{H}}^{K_N}_{\R}$-modules. Define the finitely generated $\Z_{(p)}$-module $M_Y$ by the equation  $M_Y=p_Y(\tilde{H}^3_!(S_{K_N}, V_{\lambda, \Z_{(p)}}))$ where  $p_Y: \tilde{H}^3_!(S_{K_N}, V_{\lambda, \R}) \rightarrow Y$ is the canonical projection. The finitely generated $\Z_{(p)}$-module $L_Y$ is defined as $L_Y=Y \cap \tilde{H}^3_!(S_{K_N}, V_{\lambda, \Z_{(p)}})$. These $\Z_{(p)}$-modules are stable by the action of $\overline{\mathcal{H}}^{K_N}_{\Z_{(p)}}$ on $\tilde{H}^3_!(S_{K_N}, V_{\lambda, \Z_{(p)}})$ and we have $L_Y \subset M_Y$. According to Theorem \ref{main1} the prime $p$ divides $d(\Pi_f)$. It follows from \cite[Proposition 4.3]{hida} that $|d(\Pi_f)|=[L(\Pi_f, V_{\lambda, \Z_{(p)}})^*: L(\Pi_f, V_{\lambda, \Z_{(p)}})]$. Hence, it follows from Lemma \ref{lem-hida} that $p$ divides the index $[M(\Pi_f, V_{\lambda, \Z_{(p)}}): L(\Pi_f, V_{\lambda, \Z_{(p)}})]$. By replacing in the proof of \cite[Theorem 7.1]{hida} the symbols $L, L_f, M_f$ and $R$ by the symbols $\tilde{H}^3_!(S_{K_N}, V_{\lambda, \Z_{(p)}}), L(\Pi_f, V_{\lambda, \Z_{(p)}}), M(\Pi_f, V_{\lambda, \Z_{(p)}})$ and $\overline{\mathcal{H}}^{N}_{\mathrm{sph}, \Z_{(p)}}$ respectively, we obtain a prime divisor $\mathfrak{P}$ of $p$ in $\overline{\Q}$ and a $\overline{\mathcal{H}}^{K_N}_{\mathrm{sph}, \Z_{(p)}}$-module contributing to $Y$ which is congruent to $\Pi$ modulo $\mathfrak{P}$. This amounts to the existence of a cuspidal automorphic representation $\Pi' \simeq \bigotimes'_v \Pi'_v$ such that $(\Pi'_f)^{K_N}$ is non zero, such that $\Pi'_\infty \in P(V_{\lambda, \C})$, such that $\Pi' \not\simeq \,^\sigma\!\,\Pi$ for all $\sigma \in \mrm{Aut}(\C)$ and such that $\Pi' \equiv \Pi \pmod{\mathfrak{P}}$. 
\par 
In the rest of the proof, we show that $\Pi'$ is stable under the conditions \rm{(a)} and \rm{(b)}. In order to show this with contradiction,  
let us assume that $\Pi'$ is not stable. According to the classification of the discrete spectrum of $\mathrm{GSp}(4)$ conjectured in \cite{arthur}, proved in \cite{gee-taibi}, this means that $\Pi'$ is either of Yoshida type, Saito-Kurokawa type, Howe Piatetski-Shapiro type or one-dimensional type. The mod $\mathfrak{P}$ semi-simple Galois representation $\overline{\rho}_{\Pi}$ attached to $\Pi$ is given by $\overline{\rho}_{\Pi} \simeq \overline{\rho}_{f_1} \oplus \overline{\rho}_{f_2}(-1-k')$. Hence the last three types mentionned above are excluded by our condition \rm{(a)}. Now, let us  assume that 
$\Pi'$ is of Yoshida-type. By \cite[Corollary 4.2]{weissauer2}, there exists a pair $(\sigma'_1, \sigma'_2)$  of cuspidal representations of $\mrm{GL}(2, \A)$
which correspond to normalized eigen elliptic cuspforms 
$f'_1$ and $f'_2$ of weight $k'_1 = k+k' +4$ and $k'_2 = k-k' +2$ respectively such that $\overline{\rho}_{\Pi'} \simeq \overline{\rho}_{f_1'} \oplus \overline{\rho}_{f_2'}(-1-k')$. Since we have the congruence $\Pi' \equiv \Pi \pmod{\mathfrak{P}}$, we have 
\begin{equation}\label{equation:isogaloisPi}
\overline{\rho}_\Pi \simeq \overline{\rho}_{\Pi'}.
\end{equation}
As
\begin{equation}\label{equation:isogaloisrho}
\overline{\rho}_\Pi \simeq \overline{\rho}_{f_1} \oplus \overline{\rho}_{f_2} (-1-k'  ), \ \ \ 
\overline{\rho}_{\Pi'} \simeq \overline{\rho}_{f'_1} \oplus \overline{\rho}_{f'_2} (-1-k'  ). 
\end{equation}  
As $\overline{\rho}_{f_1}$ and $\overline{\rho}_{f_2}$ are irreducible, the equation \eqref{equation:isogaloisPi} implies either $\overline{\rho}_{f_1} \simeq \overline{\rho}_{f'_1}$ or $\overline{\rho}_{f_1} \simeq \overline{\rho}_{f'_2} (-1-k'  )$. 
We will prove that these equalities never happen. For $i=1,2$ let $N_i$ be the conductor of $f_i$ and let $N'_i$ be the conductor of $f'_i$. We have $N_1 N_2=N$, in particular $(N_1, N_2)=1$ and $N_1$, $N_2$ are square-free. Moreover, as $\Pi'$ has level $N$, we have $N'_1 N'_2 |N_1N_2$.
\par 
We will prove that the first case $\overline{\rho}_{f_1} \simeq \overline{\rho}_{f'_1}$
is excluded by contradiction assuming that $\overline{\rho}_{f_1} \simeq \overline{\rho}_{f'_1}$ holds. When the conductor $N'_1$ of $f'_1$ divides the conductor $N_1$ of $f_1$, this contradicts to the assumption that 
$\mathfrak{p}$ does note divide $c'(f_1)$ because the existence of such $f'_1$ by assumption, must imply that a prime $\mathfrak{p}$ above $p$ 
in $\overline{\mathbb{Q}}$ divides $c' (f_1)$ according to Theorem \ref{theorem:Hida2}. 
When $N_1$ divides $N'_1$, the conductor $N'_2 $ of $f'_2$ must divide the conductor $N_2$ of $f_2$ because we have $N'_1N'_2 | N_1N_2$.
Hence a similar argument as above shows that this contradicts to the assumption that no prime $\mathfrak{p}$ above $p$ in $\overline{\mathbb{Q}}$ divides $c'(f_2)$. 
Finally suppose that none of $N_1 $ nor $N'_1$ divides the other. Then the equality $\overline{\rho}_{f_1} \simeq \overline{\rho}_{f'_1}$  
implies that the analytic conductor of the mod $p$ Galois representation $\overline{\rho}_{f_1}$ is a strict divisor of $N_1$. 
Hence the Serre conjecture proved by Khare-Wintenberger \cite{khare-wintenberger1}, \cite{khare-wintenberger2} implies that 
there must exist a normalized eigen cuspform $f'$ congruent to $f_1$ modulo $p$ whose conductor is a strict divisor of $N_1$. Because $p \notin S_{\mrm{weight}}$  the modular form $f'$ has the same weight as $f_1$. This is again a contradiction to the assumption that no prime $\mathfrak{p}$ above $p$ in $\overline{\mathbb{Q}}$ divides $c'(f_1)$.\\
In order to exclude the second case $\overline{\rho}_{f_1} \simeq \overline{\rho}_{f'_2} (-1-k'  )$, 
let us recall about the mod $p$ Galois representation $\overline{\rho}_f$ 
of an elliptic cusp form $f$ of weight $k<p-2$ and of level $N$ prime to $p$. 
When $f$ is ordinary at $p$, $\overline{\rho}_f$ is reducible when restricted to the decomposition group $D_p$ at $p$ and 
the restriction to the inertia subgroup is given as 
\begin{equation}\label{equation:restrictiontoineartia}
\left. \overline{\rho}_{f}\right\vert_{I_p} \simeq 1 \oplus \omega^{1-k}. 
\end{equation} 
When $f$ is non ordinary at $p$, the restriction to the inertia subgroup of $\overline{\rho}_{f}$ is given as 
\begin{equation}\label{equation:restrictiontoineartia2}
\left. \overline{\rho}_{f}\right\vert_{I_p} \simeq \omega^{k-1}_2 \oplus ({\omega'_2})^{k-1}, 
\end{equation} 
where $\omega_2 $ and $\omega'_2$ are fundamental characters of level $2$ such that $\omega_2$ and $\omega'_2$ are both of order $p^2-1$ and 
we have $\omega'_2 = (\omega_2 )^p$. Suppose that $f_1$ and $f_2'$ are both ordinary. 
The fact that $p \notin S_{\mathrm{weight}}$ implies that $\overline{\rho}_{f'_2} (-1-k'  )$ does not contain the trivial character and hence the second case $\overline{\rho}_{f_1} \simeq \overline{\rho}_{f'_2} (-1-k'  )$ is excluded. When one of $f_1$ and $f_2'$ is ordinary and 
the other is non ordinary, the seconde case $\overline{\rho}_{f_1} \simeq \overline{\rho}_{f'_2} (-1-k'  )$ is excluded again 
since it is not difficult to see that a representation of the type \eqref{equation:restrictiontoineartia} is never isomorphic to a twist of a representation of the type \eqref{equation:restrictiontoineartia2}. 
This is a contradiction. Finally when $f_1$ and $f_2'$ are both non ordinary, 
the second case $\overline{\rho}_{f_1} \simeq \overline{\rho}_{f'_2} (-1-k'  )$ is excluded
by the assumption $p \notin S_{\mathrm{weight}}$. 
\par 
By the above discussion $\Pi'$ is either stable.
\end{proof}



\begin{thebibliography}{9999}

\bibitem[AK13]{ak} M. Agarwal, K. Klosin, {\it Yoshida lifts and the Bloch-Kato conjecture for the convolution L-function}, J. Number Theory 133 (2013), no. 8, 2496-2537.

\bibitem[An87]{andrianov} A.~N.~Andrianov, {\it Quadratic forms and Hecke operators}, Grundlehren der Mathematischen Wissenschaften, 286, Springer-Verlag, Berlin, 1987. xii+374.

\bibitem[Ar04]{arthur} J. Arthur, {\it Automorphic representations of GSp(4)}, Contributions to automorphic forms, geometry and number theory, John Hopkins Univ. Press, Baltimore, MD, (2004), 65-81.

\bibitem[AS06]{asgari-shahidi} M.~Asgari, F.~Shahidi, {\it Generic transfer from $\mrm{GSp(4)}$ to $\mrm{GL}(4)$}, Compositio Math. 142, (2006), 541-550.

\bibitem[ASch01]{asgari-schmidt} M.~Asgari, R.~Schmidt, {\it Siegel modular forms and representations}, Manuscripta Math. 104, (2001), 173-200.

\bibitem[Ay08]{ayoub1} J.~Ayoub, {\it Les six op\'erations de Grothendieck et le formalisme des cycles \'evanescents dans le monde motivique I}, Ast\'erisque no. 314, (2007), x+466.

\bibitem[Ay10]{ayoub2} J.~Ayoub, {\it Note sur les op\'erations de Grothendieck et la r\'ealisation de Betti}, J. Inst. Math. Jussieu 9, no. 2, (2010), 225-263.

\bibitem[BDSP12]{bdsp} S. B\"ocherer, N. Dummigan, R. Schulze-Pillot, {\it Yoshida lifts and Selmer groups},
J. Math. Soc. Japan 64 (2012), no. 4, 1353-1405.

\bibitem[BHR94]{bhr} D.~Blasius, M.~Harris, D.~Ramakrishnan, {\it Coherent cohomology, limits of discrete series and Galois conjugation}, Duke Math. Journ. 73 (1994), no. 3, 647-685.

\bibitem[BaR17]{br} B.~Balasubramanyam, A.~Raghuram, {\it Special values of adjoint $L$-functions and congruences for automorphic forms on GL(n) over a number field}, Amer. Journ. of Math 139, No.3, (2017), 641-679.

\bibitem[Bo81]{borel1} A.~Borel, {\it Stable real cohomology of arithmetic groups II}, in Manifolds and Lie groups (Notre Dame, Ind., 1980), pp. 21-55, Progr. Math., 14, Birkh\"auser, Boston, Mass., (1981).

\bibitem[CI19]{ch19} S.~Chen, A.~Ichino, {\it 
On Petersson norms of generic cusp forms and special values of adjoint $L$-functions for $\mathrm{GSp}_4$}, preprint, arXiv:1902.06429 [math.NT].



\bibitem[DS05]{diamond-shurman} F.~Diamond, J.~Shurman, {\it A first course in modular forms}, Graduate Texts in Mathematics, 228. Springer-Verlag, New York, 2005. xvi+436.

\bibitem[Di05]{dimitrov} M.~Dimitrov, {\it Galois representations modulo $p$ and cohomology of Hilbert modular varieties}, Ann. Sci. Ecole Norm. Sup. (4)  38  (2005),  no. 4, 505--551. 



\bibitem[FH91]{fulton-harris} W.~Fulton, J.~Harris, {\it Representation theory. A first course}, GTM 129, Springer-Verlag, New-York, (1991).

\bibitem[GaR13]{gan-raghuram} W.~T.~Gan, A.~Raghuram, {\it Arithmeticity for periods of automorphic forms}, Automorphic representations and $L$-functions, Tata Inst. Fundam. Res. Stud. Math., 22, Mumbai, (2013), 187-229. 

\bibitem[GeT18]{gee-taibi} T. Gee, O. Ta\"ibi, {\it Arthur's multiplicity formula for $\mrm{GSp}_4$ and restriction to $\mrm{Sp}_4$}, preprint, arXiv:1807.03988 [math.NT].

\bibitem[Gh02]{ghate} E.~Ghate, {\it Adjoint $L$-values and primes of congruence for Hilbert modular forms}, Compositio Math. 132 (2002), no. 3, 243--281.




\bibitem[Hi81a]{hida} H.~Hida, {\it Congruence of cusp forms and special values of their zeta functions}, Invent. Math. 63, (1981), no. 2, 225-261.

\bibitem[Hi81b]{hida'} H. Hida, {\it On congruence divisors of cusp forms as factors of the special values of their zeta functions}, 
Invent. Math. 64 (1981), no. 2, 221-262. 

\bibitem[I08]{ichino} A.~Ichino, {\it On critical values for adjoint $L$-functions for GSp(4)}, preprint, (2008), available at https://www.math.kyoto-u.ac.jp/$\sim$ichino/ad.pdf.


\bibitem[JSo07]{jiang-soudry} D.~Jiang, D.~Soudry, {\it The multiplicity one theorem for generic automorphic forms of GSp(4)}, Pacific Journ. of Math. Vol. 229, no. 2, (2007), 381-388.

\bibitem[KaS94]{kashiwara-schapira} M.~Kashiwara, P. ~Schapira, {\it Sheaves on manifolds}, A Series of Comprehensive Studies in Mathematics, 292, Springer-Verlag (1994), x+512.

\bibitem[Ka08]{katsurada} H.~Katsurada, {\it
Congruence of Siegel modular forms and special values
of their standard zeta functions}, Mathematische Zeitschrift, Volume 259, Issue 1, (2008), pp.~97--111. 

\bibitem[KhWI09]{khare-wintenberger1} C. Khare, J.-P. Witenberger, {\it Serre's modularity conjecture I},  Invent. Math. 178 (2009), no. 3, 485-504.

\bibitem[KhWII09]{khare-wintenberger2} C. Khare, J.-P. Witenberger, {\it Serre's modularity conjecture II},  Invent. Math. 178 (2009), no. 3, 505-586.

\bibitem[Kn86]{knapp} A.~W.~Knapp, {\it Representation theory of semisimple Lie groups. An overview based on examples}, Princeton Mathematical Series 36, Princeton University Press, Princeton, NJ, (1986), 1-774.

\bibitem[Ko88]{kottwitz} R.~E.~Kottwitz, {\it Tamagawa numbers}, Annals of Math., Second Series, Vol. 127, No. 3, (1988), 629-646.

\bibitem[LS13]{lan-suh} K.-W. Lan, J. Suh, {\it Vanishing theorems for torsion automorphic sheaves on general PEL-type Shimura varieties}, Adv. Math. 242 (2013), 228-286.

\bibitem[L05]{laumon} G.~Laumon, {\it Fonctions z\^etas des vari\'et\'es de Siegel de dimension $3$},  in Formes Automorphes (II): le cas du groupe GSp(4), Ast\'erisque No. 302, (2005), 1-66.

\bibitem[Le17]{lemma2} F.~Lemma, {\it On higher regulators of Siegel threefolds II: the connection to the special value}, Compositio Math. 153, no. 5, (2017), 889-946.

\bibitem[MT02]{mokrane-tilouine} F.~Mokrane, P.~Polo, J. Tilouine, {\it Cohomology of Siegel varieties}, Ast\'erisque 280, (2002), 1-135.

\bibitem[N15]{namikawa} K.~Namikawa, {\it On a congruence prime criterion for cusp forms on $GL_2$ over number fields}, J. reine angew. Math. 707, (2015), 149-207.

\bibitem[Pe15]{pet} D.~Petersen, {\it Cohomology of local systems on the moduli of principally polarized abelian surfaces}, Pacific J. Math. 275 (2015), no. 1, 39--61.

\bibitem[PSS14]{saha} A.~Pitale, A.~Saha, R.~Schmidt, {\it Transfer of Siegel cuspforms of degree $2$}, Mem. Amer. Math. Soc. (2014), 232 (1090).

\bibitem[PT02]{polo-tilouine} P.~Polo, J. Tilouine, {\it Besnstein-Gelfand-Gelfand complexes and cohomology of nilpotent groups over $\Z_{(p)}$ for representations with $p$-small weights}, Ast\'erisque 280, (2002), 97-135.

\bibitem[R10]{renard} D.~Renard, {\it Repr\'esentations des groupes r\'eductifs $p$-adiques}, Cours Sp\'ecialis\'es, SMF, Paris, 2010. vi+332.

\bibitem[Ri76]{ribet} K.~A.~Ribet, {\it A modular construction of unramified $p$-extensions of $\Q(\mu_p)$}, Invent. Math. 34 (1976), no. 3, 151--162.

\bibitem[Ri83]{ribet83} K. A. Ribet, {\it Mod p Hecke operators and congruences between modular forms},
Invent. Math. 71 (1983), no. 1, 193-205.

\bibitem[RoS07]{roberts-schmidt} B. Roberts, R. Schmidt, {\it Local newforms for $\mrm{GSp}(4)$}, Lecture Notes in Mathematics, 1918. Springer, Berlin, 2007. viii+307 pp.

\bibitem[Sie43]{siegel} C.~L.~Siegel, {Symplectic geometry}, Amer. J. of Math., 65, (1943), 1-86.



\bibitem[U95]{urban} E.~Urban, {\it Formes automorphes cuspidales pour $\mrm{GL}_2$ sur un corps quadratique imaginaire. Valeurs sp\'eciales de fonctions $L$ et congruences}, Compositio Math. 99 (1995), no. 3, 283--324.

\bibitem[Wa85]{waldspurger} J.-L.~Waldspurger, {\it Quelques propri\'et\'es arithm\'etiques de certaines formes automorphes sur GL(2)}, Compositio Math. 54, no. 2, (1985), 121-171.

\bibitem[We05]{weissauer1} R.~Weissauer, {\it Four dimensional Galois representations}, Formes automorphes II. Le cas du groupe GSp(4), Ast\'erisque No. 302, (2005), 67-150.

\bibitem[We09]{weissauer2} R.~Weissauer, {\it Endoscopy for $\mrm{GSp(4)}$ and the cohomology of Siegel modular threefolds}, Lecture Notes in Mathematics 1968, Springer-Verlag, Berlin, (2009), xviii+368.

\end{thebibliography}
\end{document}